\documentclass[reqno,11pt]{amsart}

\usepackage{lmodern}
\usepackage[english]{babel}

\usepackage[colorlinks=true, linkcolor=blue, citecolor=blue]{hyperref}

\usepackage{amssymb}   
\usepackage{amsmath, graphicx, rotating}
\usepackage{color}     
\usepackage{soul}
\usepackage[dvipsnames]{xcolor}     
\usepackage{tcolorbox}
\usepackage{comment}
\usepackage{enumitem}

\usepackage{ upgreek }
\usepackage{stmaryrd}
\SetSymbolFont{stmry}{bold}{U}{stmry}{m}{n}
\usepackage{amsthm}
\usepackage{float}

\usepackage{ bbm }
\usepackage{ stmaryrd }
\usepackage{ mathrsfs }
\usepackage{ frcursive }
\usepackage{ comment }

\usepackage{pgf, tikz}
\usetikzlibrary{shapes}
\usepackage{varioref}
\usepackage{enumitem}
\usepackage{longtable}

\usepackage{mathtools}

\usepackage{dsfont}

\definecolor{rouge}{rgb}{0.7,0.00,0.00}
\definecolor{vert}{rgb}{0.00,0.5,0.00}
\definecolor{bleu}{rgb}{0.00,0.00,0.8}

\usepackage[margin=1.3in]{geometry}

\newtheorem{theorem}{Theorem}[section]
\newtheorem*{theorem*}{Theorem}
\newtheorem{lemma}[theorem]{Lemma}

\newtheorem{corollary}[theorem]{Corollary}

\labelformat{hypothesis}{\textbf{M\kern-0.1mm#1}}

\newtheorem{condition}{Condition}

\labelformat{conditionA}{\textbf{A\kern-0.1mm#1}}

\renewcommand\dots{\hbox to 1em{.\hss.\hss.}}

\theoremstyle{definition}

\numberwithin{equation}{section}

\newcommand\ee{\varepsilon}

\DeclareMathOperator{\supp}{supp}

\DeclareMathOperator{\sh}{sinh}
\DeclareMathOperator{\ch}{cosh}

\DeclarePairedDelimiter\floor{\lfloor}{\rfloor}

\def\geq{\geqslant}
\def\leq{\leqslant}

\def\bb#1{\mathbb{#1}}

\def\scr#1{\mathscr{#1}}

\begin{document}
\title[Conditioned random walks]
{Local limit theorems for conditioned random walks\\by the heat kernel approximation} 
\author{ Ion~Grama}
\curraddr[I. Grama]{Univ Bretagne Sud, CNRS UMR 6205, LMBA, Campus de Tohannic, F-56000 Vannes, France}
\email[I. Grama]{ion.grama@univ-ubs.fr}
\author{Hui Xiao}
\curraddr[H. Xiao]{Academy of Mathematics and Systems Science, Chinese Academy of Sciences, Beijing 100190, China}
\email[H. Xiao]{xiaohui@amss.ac.cn}

\date{\today }

\maketitle

\begin{abstract}   
We study the random walk $(S_n)_{n\geq 1}$ with independent and identically distributed real-valued increments having zero mean and an absolute moment of order $2 + \delta$ for some $\delta > 0$. For any starting point $x \in \mathbb{R}$, let $\tau_x = \inf\{k \geq 1 : x + S_k < 0\}$ denote the first exit time of the random walk $x + S_n$ from the half-line $[0, \infty)$. In the previous work \cite{GX-2024-CCLT}, we established a Gaussian heat kernel approximation for both the persistence probability 
$\mathbb{P}(\tau_x > n)$ and the joint distribution $\mathbb{P}(x + S_n \leq \cdot, \tau_x > n)$, uniformly over $x \in \mathbb{R}$ as $n \to \infty$. In this paper, we leverage these results to establish a novel conditioned {\it local limit theorem} for the walk $(x + S_n)_{n \geq 1}$. For $\mathbb{Z}$-valued random walks, we prove that the joint probability $\mathbb{P}(x + S_n = y, \tau_x > n)$ is uniformly approximated by a distribution governed by the Gaussian heat kernel over all $x, y \in \mathbb{Z}$ as $n \to \infty$. Our new asymptotic unifies into a single comprehensive formula the classical local limit theorem by Caravenna \cite{Carav05}, as well as various results relying on specific assumptions on $x$ and $y$. As a corollary, we obtain a new uniform-in-$x$ asymptotic formula for the local probability $\mathbb{P}(\tau_x = n)$. We also extend our analysis to non-lattice random walks.
\end{abstract}


\section{Introduction}\label{Subsec-Intro}
Let $S_n = \sum_{i=1}^n X_i$, $n\geq 1$, be a random walk with independent and identically distributed real-valued increments $(X_i)_{i \geq 1}$  
having zero mean and finite variance $\sigma^2>0$. 
For simplicity of the exposition in this section 
we assume that $(S_n)_{n \geq 1}$ takes values on the lattice $\bb Z$ with minimal span equal to $1$, where $\bb Z$ denotes the set of integers.  
For any starting point $x \in \mathbb{Z}$, 
define the first passage time $\tau_x$ as the first moment when $x + S_n$ exits the non-negative half-line $\mathbb{R}_+ := [0, \infty).$
For $x, y \in \mathbb{Z}$, possibly depending on $n$, consider the local probability
\begin{align} \label{Objective-proba001}
\mathbb{P} \left( x + S_n = y, \tau_x > n - 1 \right).
\end{align}
Probabilities of the form \eqref{Objective-proba001} have been extensively studied in the literature. 
Notable contributions include the works of  
Borovkov \cite{Borovk62}, 
Feller \cite{Fel64}, Spitzer \cite{Spitzer}, 
Bolthausen \cite{Bolth}, Iglehart \cite{Igle74}, 
Kozlov \cite{Kozlov76}, 
Bertoin and Doney \cite{BertDoney94}, 
Caravenna and Chaumont \cite{Caravenna-Chaumont08}, 
Doney and Jones \cite{DJ12}, Denisov and Wachtel \cite{Den Wachtel 2011, DW24}, 
Kersting and Vatutin \cite{KV17}, among others. 

One of the reference results in this area is the local limit theorem by Caravenna \cite{Carav05}, which extends the classical local limit theorems of Gnedenko \cite{Gned48} and Stone \cite{Sto65} to conditioned random walks.
Reformulating it in a convenient way, the following asymptotic holds:
as $n \to \infty$, uniformly in $y \in \mathbb{Z}_+ = \{0, 1, 2, \ldots \}$,  
\begin{align} \label{eq-into-carav-001}  
\mathbb{P} \big(S_n = y, \tau_0 > n-1 \big) 
   =   \frac{2\bb E(-S_{\tau_0})}{\sqrt{2\pi}\sigma^2 n }\phi^+\left(\frac{y}{\sigma\sqrt{n}}\right)  +  \frac{o(1)}{n},  
\end{align}  
where $y\mapsto \phi^+(y) = ye^{-y^2/2}$ on $\bb R_+$ is the Rayleigh density function.

The following remarks are in order.  
First, in \eqref{eq-into-carav-001} the starting point is $x = 0$.  
Actually, the limit law $\phi^+$ remains unchanged as long as  $\frac{x}{\sqrt{n}} \to 0$ as $n \to \infty$.  
However, going beyond this restriction shows a different limiting law in place of the Rayleigh distribution.  
Second, the leading term in \eqref{eq-into-carav-001} is meaningful only for values of $y$ satisfying  
$y \in (\alpha_n \sqrt{n}, \alpha_n^{-1} \sqrt{n})$,  
where $\alpha_n \to 0$ is a sequence determined by the Landau symbol $o(1)$  
in \eqref{eq-into-carav-001}.  
Taking $y$ to be fixed yields only a rough upper bound for the local probability,  
as in this case, the leading term in \eqref{eq-into-carav-001} vanishes faster than the error term:  
 $\frac{1}{n}\phi^+\big(\frac{y}{\sigma\sqrt{n}}\big) = O(n^{-3/2}).$  
 
 Over time, numerous studies have explored cases beyond these constraints, examining  
scenarios where the starting point $x \geq 0$ and the arrival point $y \geq 0$ satisfy  
either $\frac{x}{\sqrt{n}} \to 0$ or $x \asymp \sqrt{n}$, and analogously,  
either $\frac{y}{\sqrt{n}} \to 0$ or $y \asymp \sqrt{n}$ as $n \to \infty$.  
Asymptotic results for some of these cases have been established by  
Alili and Doney \cite{Alili-Donney99},  
Bryn-Jones and Doney \cite{Bryn-Jones-Doney06},  
Vatutin and Wachtel \cite{VatWacht09},  
Doney \cite{Don12},  
Denisov, Tarasov and Wachtel \cite{DTW24, DTW24b},
and very recently in \cite{GX-2024-AIHP}.  
From this previous work, 
it is known that uniformly in $x, y \in \bb Z_+$ satisfying $\frac{x}{\sqrt{n}} \to 0$ and $\frac{y}{\sqrt{n}} \to 0$, 
\begin{align} \label{eq-intro-000a}
 \bb{P} \Big( x+S_n = y,  \tau_x >n-1\big) 
   \sim \frac{ 2V(x) \check V(y) }{\sqrt{2\pi}\sigma^3 n^{3/2} }, 
\end{align}
where $V$ and $\check V$ are the harmonic functions of the random walks $S_n$ and $\check S_n=-S_n$ killed upon exiting $\bb R_+$ (see the next section for precise definitions).
Notably, this asymptotic differs from that given by \eqref{eq-into-carav-001}. 
Further examples of exact asymptotics can be found in \cite{GX-2024-AIHP}; 
see also the consequences of Theorem \ref{Th-lattice-general equiv-c001} 
in Section \ref{sec:Main result and consequences}. 
These results highlight a variety of behaviors depending on the assumptions about $x$ and $y$.  
However, to the best of our knowledge, a complete understanding 
of how these diverse results interconnect has remained elusive--until now.  

In their previous work \cite{GX-2024-CCLT}, the authors established a Gaussian heat kernel approximation for both the persistence probability $\mathbb{P}(\tau_x > n)$ and the joint distribution  
$\mathbb{P}(x + S_n \leq u \sigma \sqrt{n}, \tau_x > n)$ for $u \geq 0$, uniformly over $x \in \mathbb{R}$ as $n \to \infty$.  
In this paper, we extend the heat kernel approach to analyze the asymptotic of the local probability \eqref{Objective-proba001}.  
This method reveals a new limiting law governed by the normalized heat kernel, 
which characterizes the asymptotic behavior of the probability \eqref{Objective-proba001} 
across different regimes of $x$ and $y$. 
Our result unifies various distinct behaviors into a single comprehensive formula, significantly generalizing Caravenna's conditioned local limit theorem \eqref{eq-into-carav-001} and the equivalence \eqref{eq-intro-000a}. 
Furthermore, our findings enable the study of moderate deviations for cases where both $x$ and $y$ grow to infinity at a rate faster than $\sqrt{n}$. 
We also provide precise upper and lower bounds in various regimes, enhancing the scope and applicability of our analysis.

Let us briefly outline the main result in the lattice case. 
Set $\phi(u)=\frac{1}{\sqrt{2\pi}}e^{-u^2/2}$ for the standard normal density function on $\bb R$. 
The key role in the following will be played by the function $\psi$ defined as follows:  
for any $x,y\in \mathbb{R}$,  
\begin{align} \label{Def-Levydens}
\psi(x,y) :=  \phi(x-y) -\phi(x+y) = \frac{1}{\sqrt{2\pi }}   \left( e^{-\frac{(x-y)^2}{2}} - e^{-\frac{(x+y)^2}{2}} \right).
\end{align}   
The function $\psi$ is closely related to the one-dimensional heat kernel and will be referred to as such.  
In addition, we introduce the following normalized version of $\psi$: for any $x,y\in \mathbb{R}$,
\begin{align} \label{heat-kernel-p_h-001} 
p(x,y) = \frac{\psi(x,y)}{H(x)H(y)}, 
\end{align}
where $H(x)=\int_{\bb R_+} \psi(x,y)dy$.
One can verify that the function $(x,y)\mapsto p(x,y)$ is well-defined for $x=0$ and/or $y=0$,
and that the normalization by $H(x)H(y)$ ensures that $p(x,y)$ remains positive for all $x, y \in \mathbb{R}$,
see Fig.\ \ref{plotpHH}. We also need the function $L(x)=\frac{H(x)}{x}$ for $x\in \bb R$, see Fig. \ref{FigFuncL-001}.  
For short,  denote $V_n(x)=V(x)L(\frac{x}{\sigma\sqrt{n}})$ and $\check V_n(x)=\check V(x)L(\frac{x}{\sigma\sqrt{n}})$,
$x\in \bb R$.

From the general results of this paper, 
it follows that under the moment condition of order $2+\delta$ with $\delta > 0$, the following 
heat kernel approximation holds: 
as $n\to\infty$, uniformly for $x, y \in \mathbb{Z}$,  
\begin{align}\label{eq-intro-001}
\mathbb{P}  \Big(x+ S_n =y,  \tau_x > n - 1 \Big)
=   \frac{V_n(x) \check V_n(y) }{ \sigma^3 n^{3/2}}   
p \left( \frac{x}{\sigma\sqrt{n}}, \frac{y}{\sigma\sqrt{n}}  \right)   
+ \frac{r_n(x,y)}{n^{3/2}}o(1),
\end{align}  
where $r_n(x,y)$ is an explicit rate function given in Theorem \ref{CLLT-lattice-n3/2 main result}.
In the case $x=0$ and $y \in (\alpha_n \sqrt{n}, \alpha_n^{-1} \sqrt{n})$ where $\alpha_n\to 0$ as $n\to\infty$, 
the expansion \eqref{eq-intro-001} refines the Caravenna asymptotic \eqref{eq-into-carav-001}
by providing a rate of convergence, see the derivation of \eqref{new extention Carav-001} 
in Section  \ref{sec:Main result and consequences}. 
Using the expression for $r_n(x,y)$, 
we derive the following asymptotic equivalence
which significantly generalizes \eqref{eq-intro-000a}: 
as $n \to \infty$,  uniformly for $x, y \in \mathbb{Z}_+$ satisfying 
$|y-x| \leq \sigma \sqrt{q n \log n}$ 
 with $q < \frac{\delta}{8(\delta+3)}$,  
\begin{align}\label{eq-intro-002}
\mathbb{P}  \Big(x+ S_n =y,  \tau_x > n - 1 \Big)
\sim   \frac{V_n(x) \check V_n(y) }{ \sigma^3  n^{3/2}  }   
p \left( \frac{x}{\sigma\sqrt{n}}, \frac{y}{\sigma\sqrt{n}}  \right).
\end{align}  
When $\frac{x}{\sqrt{n}}\to 0$ and $\frac{y}{\sqrt{n}}\to 0$,  \eqref{eq-intro-002} reduces to  \eqref{eq-intro-000a}. 
However, it also yields a non-trivial asymptotic for cases
where $x \asymp \sqrt{qn\log n}$ and/or $y \asymp \sqrt{q n\log n}$. 
Various 
asymptotics for different regimes of $x$ and $y$ are derived from \eqref{eq-intro-001} 
and \eqref{eq-intro-002} in Section \ref{sec:Main result and consequences}.

The above results also allow us to analyze the asymptotics of the local probability $\mathbb{P} ( \tau_x =n )$.
We summarize our main findings for $\mathbb{P} ( \tau_x =n )$,
with more complete statements in Section \ref{sec:loc Th for exit time-001}: 
from Theorem \ref{theorem local for tau lattice}, 
we have that, as $n \to \infty$, uniformly for $x \in \mathbb{Z}$,
\begin{align} \label{intro-local-tau_x-001}
  \mathbb{P} ( \tau_x =n+1 ) 
   = \phi\left(\frac{x}{\sigma\sqrt{n}}\right) \frac{2 V(x) \varkappa  }{\sigma^3 n^{3/2}}   
   + \frac{r_n(x)}{n^{3/2}} O(1),
\end{align}
where $r_n(x)$ is given explicitly and
$\varkappa = \sum_{k  \in \bb Z_+ } \check V(k) \bb P(\check S_1>k ).$
Using the exact expression for $r_n(x)$, 
we obtain the following equivalence: as $n\to\infty$, uniformly for
$x \in \mathbb{Z}_+$ satisfying $ x \leq \sigma \sqrt{2 q n \log n}$ with $q < \frac{\delta}{8(\delta+3)}$,  
\begin{align} \label{intro-local-tau_x-002}
  \mathbb{P} ( \tau_x =n+1 ) 
   \sim \phi\left(\frac{x}{\sigma\sqrt{n}}\right) \frac{2  V_n(x) \varkappa }{\sigma^3 n^{3/2}}. 
\end{align}
Our results extend asymptotic formulas obtained by 
Eppel \cite{Eppel-1979} and Vatutin and Wachtel \cite{VatWacht09} for $x=0$,
as well as  Doney's results \cite{Don12}  that established distinct formulas
for the regimes $\frac{x}{\sqrt{n}} \to 0$ and $x \asymp \sqrt{n}$.
To our knowledge, equation \eqref{intro-local-tau_x-001} represents the first result in the literature that provides 
an explicit rate of convergence for this problem.
 
 The preceding discussion focused on lattice random walks; however, we establish similar results  
in the non-lattice case, see Section \ref{sec: results non-lattice case}.  

The proof method employed in this paper differs from traditional approaches used to establish such results.  
The classical method relies on Wiener-Hopf identities; see Feller \cite{Fel64} and Spitzer \cite{Spitzer} for details.  
In contrast, our approach is significantly simpler and primarily based on the application of the Markov property.  
Using this property, the random process is divided into three consecutive segments,  
denoted as $I_1, I_2, I_3$.  
For two of these segments, $I_1$ and $I_3$,  
we apply the conditioned central limit theorem for both forward and reversed random walks,  
as established in \cite{GX-2024-CCLT}.  
For the middle segment, $I_2$, we use the standard local limit theorem.  
By appropriately convolving the resulting limit laws corresponding to $I_1$ and $I_2$,  
we first obtain a conditioned local limit theorem of order $n^{-1}$,  
extending Caravenna's result \cite{Carav05}.  
Further convolving this with the reversed conditioned central limit theorem for $I_3$,  
we derive a new global behavior for the conditioned sum $S_n$.  

This convolution-based approach was previously applied in \cite{GX-2024-AIHP}  
to study conditioned random walks in cases where the starting point $x$  
is either close to the origin or of order $\sqrt{n}$ as $n \to \infty$, separately.  
Notably, this method is particularly well-suited for handling dependent random variables,  
including Markov chains, products of random matrices, and dynamical systems;  
see \cite{GLL20}, \cite{GQX24}, \cite{GQX24b} and \cite{GQX23}.  
The main contribution of this paper is to demonstrate that the convolution approach  
can be used to derive a new limit law based on the heat kernel,  
which characterizes the uniform-in-$x$ and in-$y$  asymptotic behavior of local probabilities. 

The paper is organized as follows. In Section \ref{Sec-Nota-bak}, we introduce the notation and present a heat kernel-type conditioned central limit theorem, which will be used in the proofs. Section \ref{Sec-CondLocLimTh-lattice-001} focuses on our results in the lattice case, while the corresponding results for the non-lattice case are presented in Section \ref{sec: results non-lattice case}. In Section \ref{Sec-Auxiliary statements}, we state some auxiliary results, including useful properties of the heat kernel. Finally, Sections \ref{SecProof Theor-latticecase-001} and \ref{SecProof-non-lattice Theor-probtau-001} are dedicated to the proofs for the lattice and non-lattice cases, respectively.

We conclude this section by introducing some necessary notation.
Throughout, we denote by $c$ generic positive constants, and by $c_\alpha$  positive constant that depends only on the specified index. These constants may vary from line to line.
$\bb Z$ will denote the set of integers.
The indicator function of a set $B$ is denoted by $\mathds{1}_B$. For brevity, given a random variable $X$ and an event $B$, we write $\mathbb{E}(X; B)$ for the expectation $\mathbb{E}(X \mathds{1}_B)$.
For any $t > 0$, let
$\phi_t(x) = \frac{1}{\sqrt{2\pi t}} \exp\left(-\frac{x^2}{2t}\right)$ and
$\Phi_t(x) = \int_{-\infty}^x \phi_t(u) du$,  $x \in \mathbb{R}$
denote, respectively, the normal density function and cumulative distribution function with mean zero and variance $t$.
When $t=1$, we write $\phi(x) = \phi_1(x)$ and $\Phi(x) = \Phi_1(x)$.

The notation $f_n(\alpha) = o(g_n(\alpha))$ uniformly for $\alpha \in A_n$ as $n\to\infty$,  
means that there exists a sequence $(\ee_n)_{n\geq 1}$ such that $\lim_{n\to\infty}\ee_n=0$ and $f_n(\alpha) \leq \ee_n g_n(\alpha)$ for any $\alpha \in A_n$ and $n\geq 1$. 
In the same way, $f_n(\alpha) = O(g_n(\alpha))$ uniformly for $\alpha \in A_n$ as $n \to \infty$, means that
there exists $c>0$ such that $f_n(\alpha) \leq c g_n(\alpha)$ for any $\alpha \in A_n$ and $n\geq 1$.
Similarly, the notation $f_n(\alpha) \sim g_n(\alpha)$ uniformly for $\alpha \in A_n$ as $n \to \infty$, 
means that 
$\lim_{n \to \infty} \sup_{\alpha\in A_n} | f_n(\alpha) / g_n(\alpha) - 1| = 0$.

For any measurable function $g : \mathbb{R} \to \mathbb{R}$, we denote its $L^1$-norm by
$\| g\|_1=\int_{\bb R} |g(y)|dy$.
Finally, $f*g$ denotes the convolution of two functions $x\mapsto f(x)$ and $x\mapsto g(x)$ on $\bb R$, 
i.e.\ $f*g(x)=\int_{\bb R} f(x-z)g(z)dz$ for $x\in \bb R$.  
When $y\mapsto f(y)$ is a function on $\bb R$ and $(x,y) \mapsto g(x,y)$ is a function on $\bb R^2$, 
we write $f*g\,(x,y)=\int_{\bb R} f(y-z)g(x,z)dz$ for $x,y\in \bb R$.

\section{Notation and background results} \label{Sec-Nota-bak}
Assume that on a probability space $(\Omega, \mathscr{F}, \mathbb{P})$, we are given a sequence of independent and identically distributed real-valued random variables $(X_i)_{i \geq 1}$ with $\mathbb{E}(X_1) = 0$ and $\mathbb{E}(X_1^2) = \sigma^2 \in (0, \infty)$.
For each $n \geq 1$, set $S_n = \sum_{i=1}^{n} X_{i}.$
Given a starting point $x \in \mathbb{R} := (-\infty, \infty)$, we define the first passage time $\tau_x$ as the first time when the random walk $(x + S_n)_{n \geq 0}$ exits the non-negative half-line $\mathbb{R}_+ := [0, \infty)$:
\begin{align*}
\tau_x = \inf \left\{ k \geq 1: x + S_{k} < 0 \right\}, \quad \text{with} \quad \inf \emptyset = \infty.
\end{align*}
Our standing assumption is that the increment $X_1$ has a moment of order $2+\delta$ for some $\delta > 0$, that is, 
\begin{align} \label{Moment condition 2+delta} 
\mathbb{E}(|X_1|^{2+\delta}) < \infty. 
\end{align}
The asymptotic behavior of the probabilities in \eqref{Objective-proba001} is governed by a harmonic function $V(x)$ associated with the random walk $(x+S_n)_{n\geq 1}$ killed upon exiting $\mathbb{R}_+$, which we now proceed to introduce.  
It is known that under condition \eqref{Moment condition 2+delta}, for any $x \geq 0$, the expectation $\mathbb{E} |S_{\tau_x}|$ is finite; see for instance \cite{BertDoney94}.  
From this basic fact, by a monotone convergence argument, one can verify that for any $x \geq 0$,
\begin{align}\label{Def_V_002}
\lim_{n \to \infty} \mathbb{E} \left(x+S_n ; \tau_x > n\right) = V(x) := -\mathbb{E} S_{\tau_x}.
\end{align}
Moreover, using the Markov property and relation \eqref{Def_V_002}, it follows that for any $x \geq 0$,
\begin{align} \label{Doob transf}
\mathbb{E}\left( V(x+S_1); x+S_1 \geq 0 \right) = V(x),
\end{align}
which shows that the function $V$ is harmonic for the random walk $(x+S_n)_{n\geq 1}$ killed at the stopping time $\tau_x$.

Using identity \eqref{Doob transf}, one can uniquely extend the function $V$ to the whole real line $\bb R$
by setting $V(x)= \bb E V(x+S_1)\mathds 1 _{\{ x+S_1\geq 0  \}}$ for $x<0$, 
so that the harmonicity property \eqref{Doob transf} is preserved for any $x\in \bb R$. 
The harmonic function $V$ satisfies the following properties:
$V(x)\geq 0$ for $x\in \bb R$,   
$x\mapsto V(x)$ is non-decreeasing on $\bb R$,
$x \leq V(x) \leq c(1 + x) $ for any $x \geq 0$ and $\lim_{x\to \infty }\frac{V(x)}{x} = 1.$ 
The support of $V$ is given by (\cite[Example 2.10]{GLL18Ann}) 
\begin{align} \label{support-V}
\supp V = \{ x \in \bb R: \bb P(x + X_1 \geq 0) >0  \}. 
\end{align}
In particular, since $\bb E(X_1)=0$ and $\sigma^2>0$, 
from \eqref{support-V} it follows that $V(0)>0$ and hence $V$ is strictly positive on $\bb R_+$. 

Define the reversed random walk by letting $\check X_i=-X_i$ and $\check S_n  =\sum_{i=1}^{n}\check X_{i}= -S_n$, $n \geq 1.$ 
For any $x\in \bb R$, 
consider the first moment $\check \tau_x$ when the random walk $(x+\check S_n )_{n\geq 1}$ 
exits the non-negative half-line $\bb R_+$, 
which is defined as  
\begin{align}\label{definition of tau x star-001}
\check \tau_x = \inf \left\{ k \geq 1: x+\check S_{k} < 0 \right\}
 \quad \mbox{with} \quad   \inf\emptyset = \infty.   
\end{align}
The harmonic function of the random walk $(x+\check S_n )_{n\geq 1}$ killed at leaving $\bb R_+$ will be denoted by $\check V$.
The properties of the function $\check V$ are  similar to those of $V$.
For example, $\check V$ is non-decreasing and non-negative on $\bb R$ with $\check V(0) >0$, and moreover 
\begin{align}\label{support-dual-func}
\supp \check V = \{ x \in \bb R: \bb P(x +\check  X_1 \geq 0) >0  \}.
\end{align}

The relations of $V$and $\check V$ to various renewal functions in the same context is described 
in the Appendix \ref{sec: renewal func formulation}.

The function $\psi$ defined in the introduction (see \eqref{Def-Levydens}), is related to 
the one-dimensional heat kernel $(t,x,y)\mapsto \psi_t(x,y)$, which for any 
for any $t>0$ and $x,y\in \bb R$, is given by
\begin{align} \label{heat-kernel-001}
\psi_t(x,y) = \frac{1}{\sqrt{t}} \psi\left(\frac{x}{\sqrt{t}}, \frac{y}{\sqrt{t}}\right) = \phi_t(x-y)-\phi_t(x+y),
\end{align}
where we recall that $\phi_t$ stands for the normal density function with mean zero and variance $t$.
It is straightforward to verify that $\psi_t$ satisfies the heat equation with Dirichlet type initial conditions: 
for any $t\geq 0$ and $x,y\in \bb R$, it holds 
$\frac{\partial }{\partial t}\psi_t(x,y) = \frac{1}{2}\frac{\partial^2 }{\partial y^2}\psi_t(x,y),$  
with $\psi_t(x,0)=0$ for any $t>0$, and $\lim_{t\to0}\psi_t(x,y)=\delta(x-y)-\delta(x+y)$,
where $\delta$ is the Dirac delta function on $\bb R$, i.e.  $\delta(0) =\infty$ and $\delta(x) =0$ for $x\not=0.$   

Let us briefly describe some elementary properties of the heat kernel $\psi_t$.
The function $\psi_t$ is 
symmetric on $\bb R \times \bb R$, i.e. $\psi_t(x, y) = \psi_t(y, x)$ for any $x, y \in \bb R$, 
and nonnegative on $\bb R_+ \times \bb R_+$. 
Also the map $y\mapsto \psi_t(x, y)$ is odd, i.e. $\psi_t(x,-y) = -\psi_t(x, y)$ for any $x, y \in \bb R$.
There is a relation of the heat kernel $\psi_t$  
to the standard  Brownian motion $(B_t)_{t\geq 0}$ conditioned to stay positive. 
For $t>0$ and  $x\geq 0$, by L\'evy's formula (see \cite{Levy37}, Theorem 42.I, pp.194-195), 
it holds $\bb P \big( x+B_t \in dy, \tau_x^{B}>1 \big)=\psi_t(x,y)dy$, 
where  $\tau_x^{B}$ is the exit time of $(x+B_t)_{t\geq 0}$ from the half-line $\bb R_+$. 

For any $x\in \bb R$, the function $y\mapsto \psi(x,y)$ can be normalized by introducing 
the function $x\mapsto H(x)$ defined on the whole real line by setting, for any $x\in \bb R$,
\begin{align}\label{eq-definition of H-abc001} 
H(x) := \int_{0}^{\infty} \psi(x,y) dy =  2 \Phi(x) - 1. 
\end{align}
Then, the one-side normalized version of the heat kernel $\psi$ is defined as follows: 
for any $x\in \bb R$,
\begin{align} \label{def of func ell_H-001}
\ell(x,y) = \frac{\psi(x,y)}{H(x)}, 
\quad \ell(0,y) := \lim_{x\to 0} \ell(x,y)  = \phi^+(y) := y e^{-y^{2}/2}, \quad y\in \bb R. 
\end{align}
 The function $y\mapsto \phi^+(y)$ restricted on $\bb R_+$ is the Rayleigh density function.
By \eqref{eq-definition of H-abc001} and \eqref{def of func ell_H-001}, for any $x \in \bb R$, 
the function $y \mapsto \ell(x,y)$ is a probability density on $\bb R_+$: 
\begin{align*}
\int_{\bb R_+}  \ell(x, y) dy = 1. 
\end{align*}
Moreover,  for any $x, y \in \bb R$, it holds
$ \ell(x, y) = - \ell(x, -y)$ and $\ell(x, y) =  \ell(-x, y).$
The plot of $(x,y) \mapsto \ell(x,y)$ is shown in Fig. \ref{FigFunc-h-002}.

In the sequel we will need a normalized version of the function $H$   
which will play an important role in the formulation of the main results: 
for any $x\in\bb R$, set
\begin{align} \label{func L as integral of heat kern}
L(x) := \frac{H(x)}{x} = \frac{2\Phi(x)-1}{x} \quad\mbox{with}\quad L(0) := \lim_{x\to 0} L(x) = \frac{2}{\sqrt{2\pi}}.
  \end{align}
The function $L$ is positive and even on $\bb R$, i.e.\ $L(x)>0$ and $L(x)=L(-x)$, for any $x\in \bb R$. 
Moreover, $x\mapsto L(x)$ is decreasing on $\bb R_+$ 
and satisfies $L(x) \sim |x|^{-1}$, as $|x|\to\infty$. 
The plot of $x\mapsto L(x)$ is given in Fig. \ref{FigFuncL-001}. 
\begin{figure}
  \includegraphics[width=10cm]{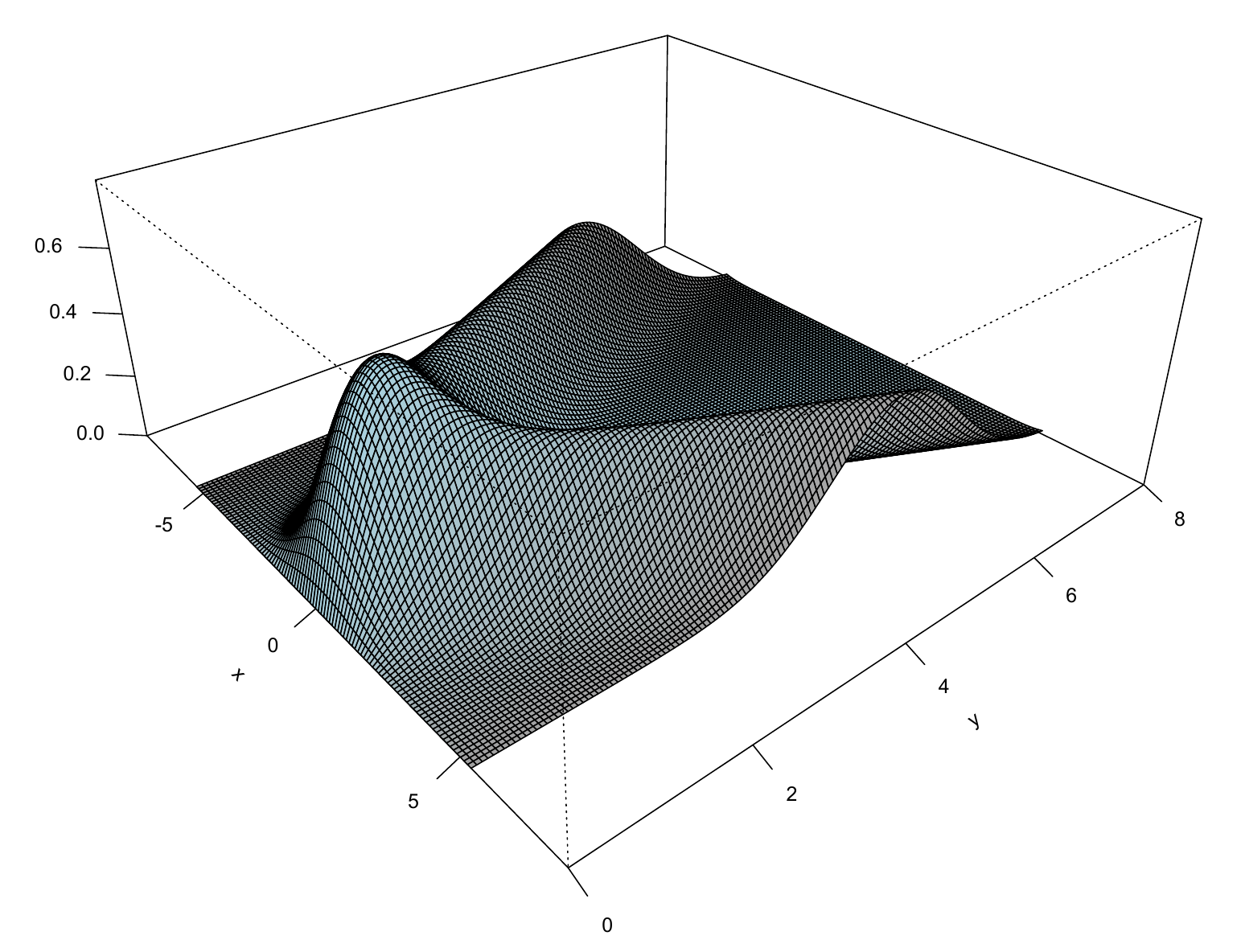}
  \caption{Plot of the function $(x,y)\mapsto \ell(x,y)$.}
  \label{FigFunc-h-002}
\end{figure}
\begin{figure}
  \includegraphics[width=9cm]{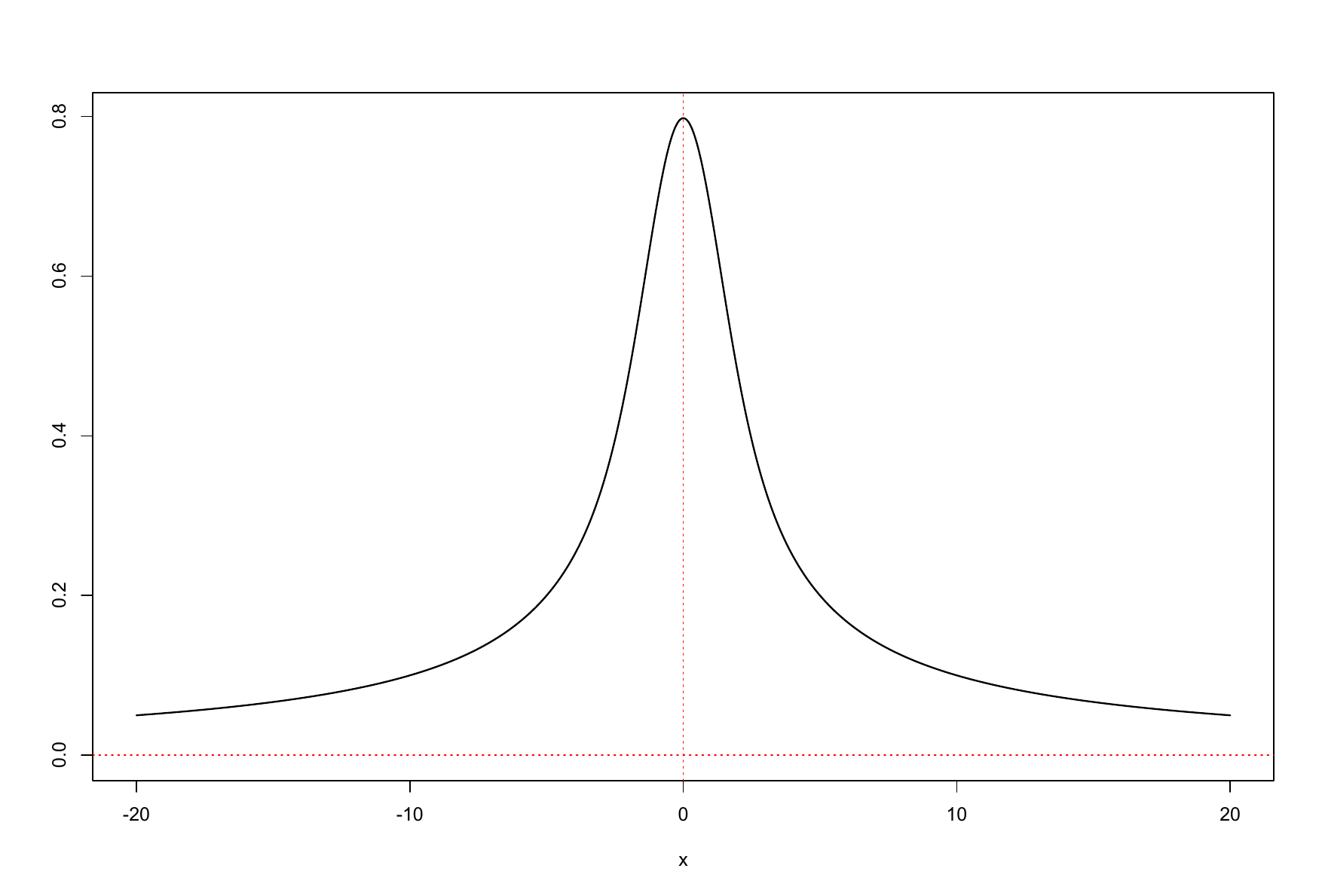}
  \caption{Plot of the function $x\to L(x)$.}
  \label{FigFuncL-001}
\end{figure}

To formulate the main results of the paper it will be convenient to introduce the following notation:  for any $n\geq 1$ and $x\in \bb R$,
\begin{align} \label{def of Q-001}
V_n(x) = V(x)L\left(  \frac{x}{ \sigma\sqrt{n}} \right), \qquad  \check V_n(x) = \check V(x)L\left(  \frac{x}{ \sigma\sqrt{n}} \right),
\end{align}
By  \cite[Lemma 2.4]{GX-2024-CCLT}, 
for any $x_0\in \supp V$ and $z_0\in \supp \check V$,
 there exist constants $c_1,c_2>0$ such that, for any $n\geq 1$, $x\in [x_0,\infty)$ and $z\in [z_0,\infty)$,
\begin{align} \label{lower bound of VL-001}
c_1 \leq V_{n}(x) \leq  c_2\sqrt{n}, \qquad c_1 \leq \check V_{n}(z) \leq  c_2\sqrt{n}.
\end{align}

The following first-order uniform asymptotics of the persistence probability $\bb{P} (\tau_x >n )$ 
and a central limit type theorem for the joint probability $\bb{P} (\frac{x+S_n}{\sigma \sqrt{n}}  \leq  u, \tau_x >n )$ 
have been determined in the paper \cite{GX-2024-CCLT}.

\begin{theorem}[\cite{GX-2024-CCLT}]\label{introTheor-probtauUN-001} 
Assume that $\bb E (X_1) = 0$, $\bb E (X^2_1)= \sigma^2  > 0$  and that there exists $\delta > 0$ 
such that $\bb E (|X_1|^{2+\delta})  < \infty.$ 
Denote $R_{n,\delta}(x) := n^{-\frac{\delta}{4} } + V_n(x)  n^{-\frac{\delta}{4(3+\delta)} }\log n.$
Then: \\
1. There exists $c >0$ such that for any  $n \geq 1$ and $x\in \bb R$,   
\begin{align} \label{UniformCondtau-001}
 \Bigg| \bb{P} \left( \tau_{x}>n\right) - \frac{V_n(x) }{\sigma \sqrt{n}} \Bigg| 
 \leq  \frac{c}{\sqrt{n}} R_{n,\delta}(x).
\end{align}
2. There exists $c >0$ such that for any  $n \geq 1$, $u\geq 0$ and $x\in \bb R$,   
\begin{align} \label{UniformCondInt-001}
\Bigg| \bb{P} \Bigg(  \frac{x+S_n }{ \sigma \sqrt{n}} \leq u, \tau_x >n\Bigg) 
      - \frac{V_{n}(x) }{\sigma \sqrt{n}}     
       \int_{0}^{u}  \ell \left( \frac{x}{\sigma \sqrt{n}}, z   \right)  dz  \Bigg| 
       \leq \frac{c}{\sqrt{n}} R_{n,\delta}(x).
\end{align}
\end{theorem}

Note that the bounds \eqref{UniformCondtau-001} and \eqref{UniformCondInt-001}  
hold uniformly for all starting points $x\in \mathbb{R}$.  
This fact will be one of the crucial points for establishing the main results of the paper.   

We will use below the following equivalence, 
which is obtained from this theorem and the first inequality in \eqref{lower bound of VL-001}:  
for any $x_0 \in \supp V$, uniformly in $x \in [x_0, \infty)$, as $n\to\infty$,  
\begin{align*} 
\mathbb{P} \left( \tau_{x}>n\right) \sim \frac{V_{n}(x) }{\sigma \sqrt{n}}.
\end{align*}  
From \eqref{UniformCondtau-001}, we derive the following upper bound, 
which will also be used in the proofs: there exists $c>0$ such that, for any $n \geq 1$ and $x \in \mathbb{R}$,  
\begin{align}\label{upper-bound-probab-tau-x}
\mathbb{P} \left( \tau_{x}>n\right)
\leq  c \frac{ n^{- \frac{\delta}{4}} + V_{n}(x) }{ \sqrt{n}}.
\end{align}  
Similar asymptotics and bounds apply to the reversed walk $\check S_n$ 
and the stopping time $\check \tau_x$.

To state our local limit theorems, we introduce the function $(x,y)\mapsto p(x,y)$ which will henceforth be referred to as the two-side normalized heat kernel: 
for any $x,y\in \bb R$, 
\begin{align} \label{def of func ell_H-002}
p (x, y) 
= \frac{\psi(x,y)}{H(x)H(y)} = \frac{\phi(x-y)-\phi(x+y)}{(2\Phi(x)-1)(2\Phi(y)-1)}. 
\end{align}
The function $p$ is well defined on $\bb R \times \bb R$. Indeed, 
in view of \eqref{def of func ell_H-001}, for any $y\in \bb R$, 
\begin{align} \label{approx by normal density-001b} 
p(0,y) = \lim_{x\to 0} p(x, y) =  \frac{\ell(0, y)}{H(y)}  
= \frac{\phi^+(y)}{H(y)} = \phi_L(y) :=  \frac{e^{-y^{2}/2}}{ L( y )}, 
\end{align}
where 
\begin{align*} 
p(0, 0) = \phi_L(0)= \lim_{y\to 0} \frac{e^{-y^{2}/2}}{L( y )} = \frac{1}{L(0)}=\frac{\sqrt{2\pi}}{2}.
\end{align*}
The function $p$ is positive  and symmetric on $\bb R \times \bb R$, and, for any $x\in \bb R$, 
the map $y\mapsto p(x, y)$ is pair on $\bb R$: 
 $p(x, y) > 0,$  $p(x, y) = p(y, x),$  $p(x,-y) = p(x, y),$ for any  $x, y \in \bb R$.
The plot of $p$ is given in Fig.\ \ref{plotpHH}.
\begin{figure}
  \includegraphics[width=10cm]{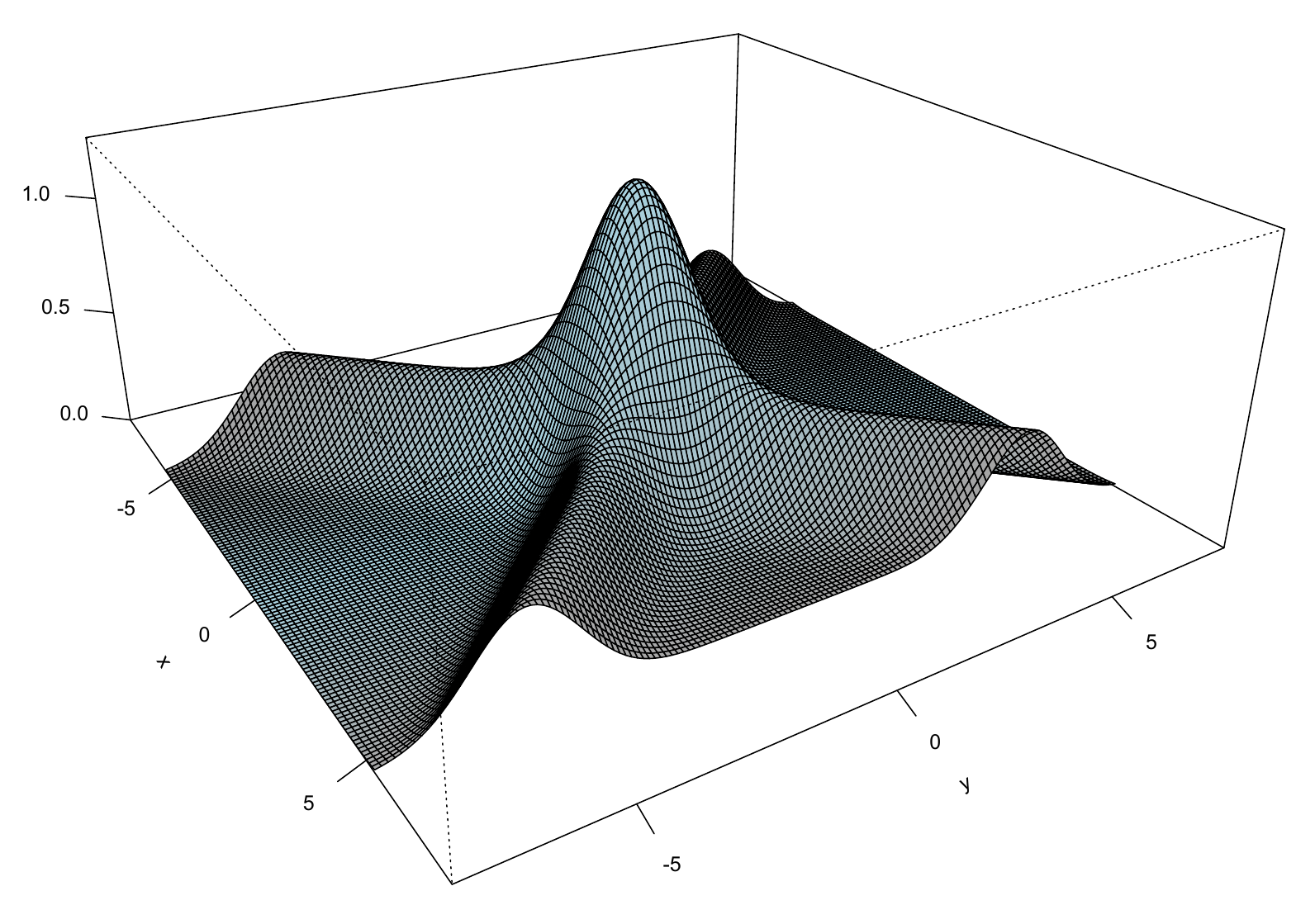}
  \caption{Plot of the function $(x,y)\mapsto p(x,y)$.}
  \label{plotpHH}
\end{figure}

\section{Conditioned local limit theorems in the lattice case} \label{Sec-CondLocLimTh-lattice-001}
\subsection{Main result} \label{sec:Main result and consequences}
Let $\hbar > 0$ and $a \in [0, \hbar)$.  
The random variable $X_1$ is said to be $(\hbar, a)$-lattice if its law is supported by the   
lattice $\hbar \mathbb{Z} + a$ and it is the coarsest possible one supporting $X_1$. 
Formally, $\mathbb{P}(X_1 \in \hbar \mathbb{Z} + a) = 1$ and for any integer $m > 1$, it holds
 $\mathbb{P}(X_1 \in m  \hbar \mathbb{Z} + a) < 1$.

The main result of the paper in the lattice case is the following conditioned local limit theorem with rate $n^{-3/2}$, 
which is an extension of Gnedenko's local limit theorem \cite{Gned48} to conditioned random walks. 

\begin{theorem} \label{CLLT-lattice-n3/2 main result} 
Assume that $X_1$ is $(\hbar, a)$-lattice, $\bb E (X_1) = 0$, $\bb E (X^2_1)= \sigma^2$,  
and that there exists $\delta > 0$ such that $\bb E (|X_1|^{2+\delta})  < \infty.$
Then there exists a constant $c>0$ such that, 
uniformly for $n\geq 2$ and $x, y \in \bb R$ satisfying $ y-x \in \hbar \bb Z+na$, 
\begin{align}\label{latt-theorem-n3/2 001}
& \bigg|  \bb{P}  \Big(x+ S_n = y,  \tau_x > n - 1 \Big)
-  \hbar \frac{V_n(x) \check V_n(y) }{ \sigma^3 n^{3/2} }   p \left( \frac{x}{\sigma\sqrt{n}} , \frac{y}{\sigma\sqrt{n} } \right) \bigg|
 \notag\\
& \qquad\qquad\qquad\qquad  
\leq  c \frac{\left(n^{-\frac{\delta}{8} }  + V_{n}(x) \right)
  \left( n^{-\frac{\delta}{8}} + \check V_n(y)  \right)}{n^{3/2  }}  n^{ -\frac{\delta}{8(3+\delta)} }\log n. 
 \end{align}
\end{theorem}

The proof of this theorem is provided in Section \ref{SecProofUnifballotLLTheor-lattice-001} 
as a consequence of a more general result stated in the form of a two-sided bound in 
Theorem \ref{T-Caravenatype-lattice 001}. 

Theorem $\ref{CLLT-lattice-n3/2 main result}$ provides meaningful asymptotic in large ranges for $x$ and $y$.
 In this regard, it offers a new perspective on various previous results from
\cite{Carav05, VatWacht09, Don12} and \cite{GX-2024-AIHP}. 

As an example, we demonstrate how the expansion \eqref{latt-theorem-n3/2 001}, which leverages the Gaussian heat kernel approximation, 
is related to the Caravenna type expansion \eqref{eq-into-carav-001}. 
Applying Theorem \ref{CLLT-lattice-n3/2 main result} with $x=0$ and using the identities \eqref{def of Q-001} and \eqref{approx by normal density-001b}, 
we obtain, for $n\geq 2$ and $ y \in \hbar \bb Z+na$,
\begin{align}\label{new extention Carav-001}
& \left|  \bb{P}  \Big(S_n = y,  \tau_0 > n - 1 \Big)
-  \hbar \frac{2V(0) \check V(y) }{\sigma^3 n^{3/2} }  \phi\left(\frac{y}{\sigma\sqrt{n}} \right) \right|
\leq  c \frac{ n^{-\frac{\delta}{8}} + \check V_n(y) }{n^{3/2  }}  n^{ -\frac{\delta}{8(3+\delta)} }\log n, 
 \end{align}
where $V(0)=\bb E(-S_{\tau_0})$.
This bound refines the  asymptotic of  Caravenna \eqref{eq-into-carav-001} 
under the $2+\delta$ moment assumption  by establishing an explicit rate of convergence.
To reconcile the two expressions, observe that by Rogozin's estimate \eqref{Rogizin estim-001}, we have 
$\frac{\check V(y) }{\sigma \sqrt{n} }  \phi\left(\frac{y}{\sigma\sqrt{n}} \right) \sim \phi^+\left(\frac{y}{\sigma\sqrt{n}} \right)$
as $n\to\infty$, uniformly for $y\geq \beta_n$, for any sequence $(\beta_n)_{n\geq 1}$ tending to infinity. 
In particular, for $y\sim\sqrt{n}$, both \eqref{new extention Carav-001} and  \eqref{eq-into-carav-001} yield the same leading asymptotic.
The primary advancement of \eqref{new extention Carav-001} 
is that it provides a nontrivial asymptotic uniformly for all $y\in \bb R$, 
including the case of fixed $y$, where \eqref{eq-into-carav-001} becomes trivial.
Further extensions of the Caravenna-type result are given by Theorem \ref{T-Caravenatype-lattice 001} for the lattice case and Theorem \ref{t-B 002} for the non-lattice case.

\subsection{Equivalence results} \label{sec-equiv-results-000}
To establish an equivalence result we consider the superlevel sets of the kernel $(x,y)\mapsto p(x,y)$: 
for any $\alpha\geq 0$, set  
\begin{align} \label{superlevelset-001}
\mathcal Q(\alpha) = \left\{  (x,y) \in \bb R^2 :  p(x,y) \geq \alpha \right\}.
\end{align}
The plots of typical level sets of the function $p$ are given in Fig.\ \ref{plotlevelsp001}. 
\begin{theorem} \label{Th-lattice-general equiv-c001}
Assume that $X_1$ is $(\hbar, a)$-lattice, $\bb E (X_1) = 0$, $\bb E (X^2_1)= \sigma^2$,  
and that there exists $\delta > 0$ such that $\bb E (|X_1|^{2+\delta})  < \infty.$ 
Let $x_0\in \supp V$, $y_0\in \supp \check V$ and $q<\frac{\delta}{8(3+\delta)}$. 
Then, as $n\to\infty$, uniformly for $x\geq x_0$ and $y\geq y_0$ such that $(x,y)\in \mathcal Q(n^{-q}),$
\begin{align} \label{lattice-corrol-geneq00}
\bb{P}  \Big(x+ S_n = y,  \tau_x > n - 1 \Big)
\sim  \hbar \frac{V_n(x) \check V_n(y) }{ \sigma^3 n^{3/2} }   
p \left( \frac{x}{\sigma\sqrt{n}} , \frac{y}{\sigma\sqrt{n} } \right).
\end{align}
In particular, the equivalence \eqref{lattice-corrol-geneq00} holds uniformly for $x\geq x_0$ and $y\geq y_0$ 
such that $y-x\in \hbar \bb Z+na$ and $|y-x| \leq \sigma \sqrt{2q n \log n}$ or  $|y+x| \leq \sigma \sqrt{2q n \log n}$.
\end{theorem}
\begin{figure}
  \includegraphics[width=7cm]{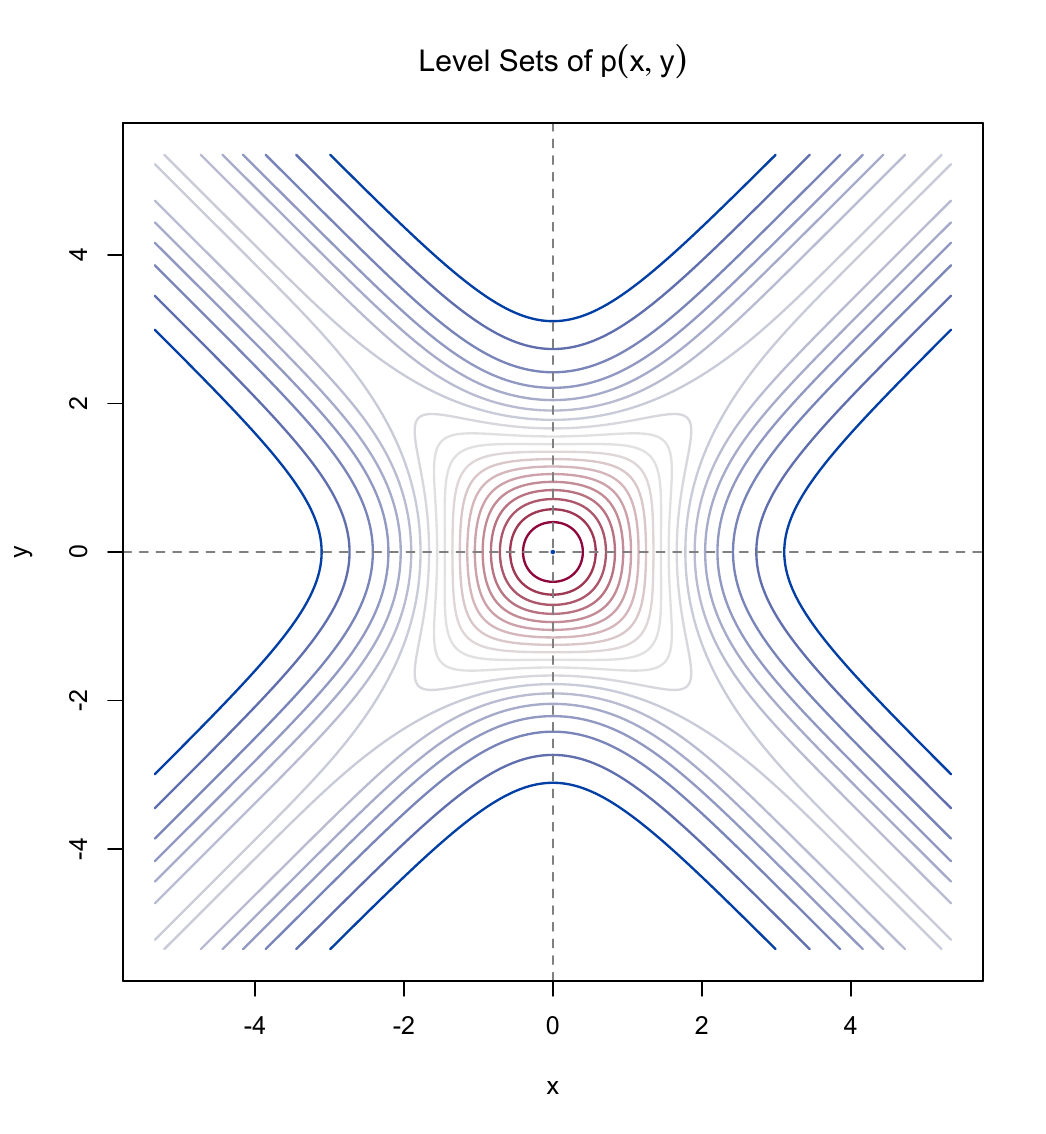}
  \caption{Level sets of the function $(x,y)\mapsto p(x,y)$, where levels range from  $0.025$ (blue) to $\frac{\sqrt{2\pi}}{2}$ (red).}
  \label{plotlevelsp001}
\end{figure}
\begin{proof}
The first part  is an immediate consequence of Theorem \ref{CLLT-lattice-n3/2 main result} 
and of the fact that by the definition, on the set $\mathcal Q(n^{-q})$ it holds $p(x,y)\geq n^{-q},$ so that the remainder term is 
$o$ compared with the leading term.
For the second part, we note that 
$\lim_{x\to\infty} p(x, x - h) = \phi(h)$, for any $h\in \bb R$, see Lemma \ref{Lem-p equiv x to inf}. 
Using this fact and the level sets of $p$ (see Fig.\ \ref{plotlevelsp001}),
it is straightforward to verify that there exists a constant $\alpha_0>0$ such that, for any $\alpha\in (0, \alpha_0]$,
\begin{align*} 
\scr D(\alpha) := \left\{ (x,y)\in \bb R^2: |x-y|\leq q_\alpha \ \text{or}\ |x+y|\leq q_{\alpha}  \right\} \subset \mathcal Q(\alpha),
\end{align*}
where $q_\alpha=\phi^{-1}(\alpha) =\sqrt{-2\ln (\sqrt{2\pi}\alpha)}$.
With $\alpha=\frac{n^{-q}}{\sqrt{2\pi}}$ we have $q_{\alpha} = \sqrt{2qn\log n}$, so that 
we obtain the second assertion of Theorem \ref{Th-lattice-general equiv-c001}. 
\end{proof}

Now we address several typical scenarios depending on the behavior of $x$ and $y$.
Under the assumptions of Theorem \ref{Th-lattice-general equiv-c001}, we deduce the following assertions,
where the sequence $(\alpha_n)_{n\geq 1}$ converges to zero and the sequence $(\beta_n)_{n\geq 1}$ converges to infinity.

\begin{itemize}
\item[a1:] 
As $n\to\infty$, uniformly in $x \geq x_0$ and $y\geq y_0$ such that $\frac{|x|}{\sqrt{n}} \leq \alpha_n$, $\frac{|xy|}{n} \leq\alpha_n$, 
$y \leq \sigma \sqrt{2q n \log n}$ 
and $y-x\in \hbar \bb Z+na$, 
\begin{align} \label{lattice-corrol-geneq01}
 \bb{P} \Big( x+S_n = y,  \tau_x >n-1 \Big) 
   \sim \hbar \frac{2V(x) \check V(y) }{\sigma^3 n^{3/2} }  \phi \left(\frac{y}{\sigma\sqrt{n}}\right). 
\end{align}

\item[a2:] 
As $n\to\infty$,  uniformly in $x \geq x_0$ and $y\geq y_0$ such that 
$\frac{|y|}{\sqrt{n}} \leq \alpha_n$, $\frac{|xy|}{n} \leq \alpha_n$,  $x \leq \sigma \sqrt{2q n \log n}$ 
and $y-x\in \hbar \bb Z+na$,
\begin{align} \label{lattice-corrol-geneq02}
 \bb{P} \Big( x+S_n = y,  \tau_x >n-1 \Big) 
   \sim \hbar \frac{2V(x) \check V(y) }{\sigma^3 n^{3/2} }  \phi \left(\frac{x}{\sigma\sqrt{n}}\right). 
\end{align}

\item[a3:] 
As $n\to\infty$, uniformly in $x \geq \beta_n$, $y \geq \beta_n$, such that  
$|y-x| \leq \sigma \sqrt{2q n\log n}$ and $y-x\in \hbar \bb Z+na$,
\begin{align} \label{lattice-corrol-eq03}
 \bb{P} \Big( x+S_n = y,  \tau_x >n-1 \Big) 
  \sim \frac{\hbar}{\sigma \sqrt{n} }\psi\left( \frac{x}{\sigma\sqrt{n}} , \frac{y}{\sigma\sqrt{n} } \right)
   =  \hbar \,  \psi_{\sigma^2n} \left( x , y \right). 
\end{align}
\end{itemize}

Claims a1-a2 are deduced from Theorem \ref{Th-lattice-general equiv-c001} using
 \eqref{def of Q-001}, \eqref{func L as integral of heat kern} and Lemmas \ref{lem-inequality for L}, \ref{Lem-p equiv-x to 0}.
The claim a3 is obtained using  Rogozin's estimate \eqref{Rogizin estim-001} and a similar one for $\check V$. 
The claims a1-a3 improve various statements which have been proved in the literature.
For example, from \eqref{lattice-corrol-geneq01} and \eqref{lattice-corrol-geneq02} we have:
\begin{itemize}
\item[a4:] 
As $n\to\infty$,  uniformly in $x \geq x_0$ and $y\geq y_0$ such that  $\frac{|x|}{\sqrt{n}}\leq \alpha_n$, $\frac{|y|}{\sqrt{n}} \leq \alpha_n$  
and $y-x\in \hbar \bb Z+na$,
\begin{align} \label{lattice-corrol-eq01}
 \bb{P} \Big( x+S_n = y,  \tau_x >n-1\big) 
   \sim \hbar \frac{2V(x) \check V(y) }{\sqrt{2\pi}\sigma^3 n^{3/2} }, 
\end{align}
\end{itemize}
which implies the asymptotic \eqref{eq-intro-000a}. 
As other examples, from \eqref{lattice-corrol-geneq01} it follows that: 
\begin{itemize}
\item[a5:] 
As $n\to\infty$, uniformly in $x \geq x_0$ and $y\geq y_0$  such that  $\frac{|x|}{\sqrt{n}} \leq \alpha_n$, $y\asymp\sqrt{n}$  
and $y-x\in \hbar \bb Z+na$, 
\begin{align} \label{lattice-corrol-eq02}
 \bb{P} \Big( x+S_n = y,  \tau_x >n-1\big) 
   \sim \hbar \frac{2V(x) }{\sqrt{2\pi}\sigma^2 n }  \phi^+ \left(\frac{y}{\sigma\sqrt{n}}\right). 
\end{align}
\item[a6:]
As $n\to\infty$, uniformly in $x \geq x_0$ and $y\geq y_0$  such that  $\frac{|x|}{\sqrt{n}} \leq \alpha_n$, $\frac{|xy|}{n}\leq \alpha_n$, $y \sim \sigma \sqrt{2q n\log n}$ and $y-x\in \hbar \bb Z+na$,
\begin{align} \label{lattice moder-corrol-eq03}
 \bb{P} \Big( x+S_n = y,  \tau_x >n-1\Big) 
   \sim \hbar \frac{2V(x)  }{\sqrt{2\pi}\sigma^2 }\frac{\sqrt{2q \log n}}{n^{1+q}}. 
\end{align}
\end{itemize}
Analogous results can be deduced from \eqref{lattice-corrol-geneq02} when $\frac{y}{\sqrt{n}} \to 0$ for various regimes of $x$.

Moreover, when $\frac{xy}{n}\to \infty$, from  \eqref{lattice-corrol-eq03} 
we deduce the following behavior which is similar to that of the ordinary local limit theorem. 
Indeed, using the identity $\psi(x,y)=\phi(x-y)(1-e^{-2xy})$, we have: 
\begin{itemize}
\item[a7:] 
As $n\to\infty$, uniformly in $x\geq \beta_n$, $y\geq \beta_n$ such that $\frac{xy}{n} \to \infty$,  
$|y-x| \leq \sigma \sqrt{2 q n\log n}$ and $y-x\in \hbar \bb Z+na$, 
\begin{align} \label{lattice-corrol-eq03bbb}
 \bb{P} \Big( S_n = y-x,  \tau_x >n-1 \Big) 
  \sim \frac{\hbar}{\sigma \sqrt{n} }\phi\left( \frac{y-x}{\sigma\sqrt{n} } \right). 
\end{align}
\end{itemize}
Note that the right-hand side is the expression appearing in the ordinary local limit theorem.
In particular, \eqref{lattice-corrol-eq03bbb} holds true uniformly for $\frac{x}{\sqrt{n}} \to \infty$, $y\geq \sqrt{n}$
(or $\frac{y}{\sqrt{n}} \to \infty$, $x\geq \sqrt{n}$), and
$|y-x| \leq \sigma \sqrt{2 q n\log n}$ such that $
y-x\in \hbar \bb Z+na$, as $n\to\infty$.

Results similar to \eqref{lattice-corrol-eq03}, \eqref{lattice-corrol-eq01}, \eqref{lattice-corrol-eq02} and \eqref{lattice moder-corrol-eq03}
have been obtained earlier in 
Caravenna \cite{Carav05}, Vatutin and Wachtel \cite{VatWacht09}, 
Doney \cite{Don12}, and \cite{GX-2024-AIHP} by other approaches.

\subsection{A local limit theorem for the exit time} \label{sec:loc Th for exit time-001}
Recall that in this section 
$X_1$ takes values on the non-centered lattice $\hbar \bb Z+a$,
where $\hbar>0$ and $a\in [0,\hbar)$. 
For any $n\geq 1$ and $x\in \bb R$, denote
\begin{align} \label{def of varkappa_n(x) abcd001}
\varkappa_n(x) :=\sum_{y\geq 0 : \, y-x \in \hbar \bb Z+na} \check V(y) \bb P(X_1<-y).   
\end{align}
The following theorem gives the asymptotic of the local behavior of the exit time  $\tau_x$, where the starting point 
$x$ belongs to the whole real line $\bb R$.

\begin{theorem} \label{theorem local for tau lattice}
Assume that $X_1$ is $(\hbar, a)$-lattice, $\bb E (X_1) = 0$, $\bb E (X^2_1)= \sigma^2$,  
and that there exists $\delta > 0$ such that $\bb E (|X_1|^{2+\delta})  < \infty.$ 
There exists a constant $c>0$ such that for any $n\geq 2$ and $x\in \bb R$, 
\begin{align} \label{locLT for time tau-002 lattice}
 \left| \bb{P} ( \tau_x =n+1 ) 
-  \hbar  \phi\left(\frac{x}{\sigma\sqrt{n}}\right) \frac{2 V(x) \varkappa_n(x) }{\sigma^3 n^{3/2}}    \right| 
  \leq  c \frac{n^{-\frac{\delta}{8} }  + V_{n}(x)}{n^{3/2  }}  n^{ -\frac{\delta}{8(3+\delta)} }\log n. 
\end{align}
\end{theorem}

An equivalence result can be obtained from Theorem \ref{theorem local for tau lattice} by restricting the range of $x$.  
Let $x_0\in \supp V$ and $q< \frac{\delta}{8(3+\delta)}$. 
Then, as $n\to \infty$, uniformly in $x_0\leq x \leq \sigma \sqrt{2qn\log n}$,
\begin{align} \label{equiv for prob tau=n 001}
 \bb{P} ( \tau_x =n+1 ) \sim  \hbar  \phi\left(\frac{x}{\sigma\sqrt{n}}\right) \frac{2 V(x) \varkappa_n(x) }{\sigma^3 n^{3/2}},
\end{align}
where hereafter we use the convention that $0\sim 0$.

Below we analyze two particular cases of \eqref{equiv for prob tau=n 001}.
Firstly, for any $x_0\in \supp V$, 
it holds that, 
as $n\to\infty$,
uniformly in $x \geq x_0$ such that $\frac{x}{\sqrt{n}} \to 0$ and $x\in \hbar \bb Z+na$,
\begin{align} \label{local tau_x-lattice001}
\bb{P} ( \tau_x =n+1 ) \sim \hbar \frac{2V(x) \varkappa_n(x) }{\sqrt{2\pi}\sigma^3 n^{3/2}}.
\end{align}
Secondly, for some sequence $\beta_n\to \infty$ as $n\to\infty$ 
depending on the rate of convergence in \eqref{locLT for time tau-002 lattice},
it holds that, as $n\to \infty$, uniformly in $x\in \hbar \bb Z+na$ satisfying 
$x \in (\beta_n^{-1} \sqrt{n}, \beta_n \sqrt{n})$,  
\begin{align} \label{local tau_x-lattice002}
\bb{P} ( \tau_x =n+1 ) \sim \hbar \frac{2 \phi^+\left(\frac{x}{\sigma\sqrt{n}}\right) \varkappa_n(x) }{\sqrt{2\pi}\sigma^2 n}.
\end{align}
In particular one can choose $\beta_n=\sigma \sqrt{2 q\log n}$, where $q < \frac{\delta}{8(3+\delta)}$.

There is an alternative formulation of the function $x\mapsto \varkappa_n(x)$ in terms 
of the function $u\mapsto \varkappa(u)$ defined as follows:  
for any $u\in [0, \hbar)$, 
\begin{align} \label{definition varkappa-aa001}
\varkappa(u) = \sum_{k =0}^{\infty} \check V(\hbar k +u) \bb P(X_1< -\hbar k-u).
\end{align}
For any $n\geq 1$ and $x\in \bb R$, denote by $\{na + x\}_{\hbar}$ the unique real number in $[0,\hbar)$ such that,  
$na+x=m\hbar + u$ for some $m\in \bb Z$.
With this notation, for any $n\geq 1$ and $x\in \bb R$,
\begin{align} \label{identity for varkappa(u) aa001}
\varkappa_n(x) = \varkappa(\left\{na + x\right\}_{\hbar}).   
\end{align}
Using Markov's inequality and the fact that $\check V(y) \leq c(1+y)$ for any $y\geq 0$, 
the function $\varkappa$ is bounded on $[0,\hbar)$, and therefore $\varkappa_n$ is also bounded on $\bb R$. 

The identity \eqref{identity for varkappa(u) aa001} implies the following straightforward properties of $\varkappa_n$.
When the random walk $S_n$ is supported on the lattice $\hbar \bb Z$ (which is equivalent to the condition $a=0$), the function $x\mapsto \varkappa_n(x)= \varkappa(\left\{x\right\}_\hbar)$ 
is periodic on $\bb R$ with period $\hbar$.
If, in addition to $a=0$, we assume that $x \in \hbar \bb Z$, then $\varkappa_n(x) = \varkappa(0)$
which is a constant.  
Similarly, if $a=0$ and the starting point $x$ takes values in the (generally non-centered) lattice $\hbar \bb Z+b$, 
with $b \in \bb R$, 
then we have $\varkappa_n(x) = \varkappa(\left\{b\right\}_\hbar)$ which is also a constant. 

The following lemmas give some more properties of the function $x\mapsto \varkappa_n(x)$. 
\begin{lemma} \label{lemma-varkappa-001}
Assume that $X_1$ is $(\hbar, a)$-lattice, $\bb E (X_1) = 0$ and $\bb E (X^2_1)= \sigma^2>0$. \\
1. There exists a constant $c>0$ such that, for any $x\in \bb R$ and $n\geq 1$,
\begin{align*} 
\bb P (\tau_x=n+1) \leq c \varkappa_n(x). 
\end{align*}
In particular, if $\varkappa_n(x)=0$, then $\bb P (\tau_x=n+1) = 0$. \\
 2. If $\varkappa_n(x)=0$ for some $x\in \bb R$ and $n\geq 1$, then $\bb P(X_1<a-\hbar)=0$.
\end{lemma}

The following lemma asserts that $\varkappa_n(x)$ is bounded from below, except the case when 
$\bb P(X_1< a-\hbar)=0$, where it is $0$ for $x$ and $n\geq 1$ such that $\{na + x\}_{\hbar} \in [\hbar - a,\hbar)$.

\begin{lemma} \label{lemma-varkappa-002}
Assume that $X_1$ is $(\hbar, a)$-lattice, $\bb E (X_1) = 0$ and $\bb E (X^2_1)= \sigma^2>0$.\\  
1. If $\bb P(X_1< a-\hbar)=0$, then $\bb P(X_1= a-\hbar)>0$, and for any $x\in \bb R$ and $n\geq 1$, 
\begin{align*} 
\varkappa_n(x)
= \left\{
\begin{array}{ll}
 \check V(\{na + x\}_{\hbar})\bb P(X_1=a-\hbar)       & \text{if } \{na + x\}_{\hbar} \in [0,\hbar-a), \\
 0       & \text{if } \{na + x\}_{\hbar} \in [\hbar - a,\hbar).
\end{array}
\right.   
\end{align*}
In particular, with $c=\check V(0)\bb P(X_1=a-\hbar)>0$, we have 
$\varkappa_n(x)\geq c$, for any  $x\in \bb R$ and $n\geq 1$
satisfying $\{na + x\}_{\hbar} \in [0,\hbar-a)$.\\ 
2. If $\bb P(X_1< a-\hbar)>0$, then, for any $x\in \bb R$ and $n\geq 1$, 
$$
\varkappa_n(x)\geq \check V(0) \bb P(X_1< a-\hbar)>0.
$$ 
\end{lemma}

The proof of Theorem \ref{theorem local for tau lattice} and Lemmas \ref{lemma-varkappa-001} and \ref{lemma-varkappa-002} are given in Section \ref{sec: proof LLT for tau}. 
 
\subsection{Upper and lower bounds}
From Theorem \ref{CLLT-lattice-n3/2 main result} we have the following upper bound: 
there exists a constant $c>0$ such that, for any $n \geq 1$ and $x,y\in \bb R$,
\begin{align} \label{Uniformbound-cc001}
\bb P_{n,x,y} := \bb{P}  \Big(x+ S_n = y,  \tau_x > n - 1 \Big)
\leq c \frac{\left(n^{-\frac{\delta}{8} }  + V_{n}(x) \right)  \left( n^{-\frac{\delta}{8}} + \check V_n(y)  \right)}{n^{3/2  }}. 
\end{align}
A more precise two-sided bound can be obtained under the assumptions of Theorem \ref{Th-lattice-general equiv-c001}. 
Let $x_0\in \supp V$, $z_0\in \supp \check V$ and $q<\frac{\delta}{8(3+\delta)}$. 
Then, there exists a constant $c>0$ such that, 
for any $n\geq 1$, $x\geq x_0$ and $y\geq y_0$ satisfying $z-x\in \hbar \bb Z+na$ and $|y-x| \leq \sigma \sqrt{q n \log n}$,
\begin{align} \label{lattice bound-corrol-geneq00}
c^{-1} \frac{V_n(x) \check V_n(z) }{ n^{3/2} }   
p \left( \frac{x}{\sigma\sqrt{n}} , \frac{y}{\sigma\sqrt{n} } \right)
\leq \bb P_{n,x,y}
\leq c \frac{V_n(x) \check V_n(y) }{ n^{3/2} }   
p \left( \frac{x}{\sigma\sqrt{n}} , \frac{y}{\sigma\sqrt{n} } \right).
\end{align}
Various scenarios depending of the behavior of $x$ and $y$ can be considered as in Section \ref{sec-equiv-results-000}.

An upper bound for the probability $\bb{P} ( \tau_x =n )$ follows
under the assumptions of Theorem \ref{theorem local for tau lattice}:
there exists a constant $c>0$ such that for any $n\geq 1$ and $x \in \bb R$, 
\begin{align*} 
 \bb{P} ( \tau_x =n +1 ) 
\leq  c \frac{n^{-\frac{\delta}{8} }  + V_{n}(x) }{n^{3/2  }}. 
\end{align*}
A precise two-sided bound can be obtained also from Theorem \ref{theorem local for tau lattice}: 
there exists a constant $c>0$ such that for any $n\geq 2$ and $x \leq \sigma \sqrt{2 q n \log n}$, 
\begin{align*} 
c^{-1} \phi\left(\frac{x}{\sigma\sqrt{n}}\right) \frac{V_n(x)  }{\sigma^3 n^{3/2}}
\leq  \bb{P} ( \tau_x =n+1 ) 
\leq c \phi\left(\frac{x}{\sigma\sqrt{n}}\right) \frac{V_n(x)  }{\sigma^3 n^{3/2}}. 
\end{align*}

\section{Conditioned local limit theorems in the non-lattice case}\label{sec: results non-lattice case}
\subsection{Main result and corollaries}
We say that the random variable $X_1$ is non-lattice 
if for any $\hbar >0$ and $a \in [0, \hbar)$ it holds that
$\bb P(X_1 \in  \hbar \bb Z +a ) \not=1$. 

Our main result in the non-lattice case is the following theorem. 

\begin{theorem} \label{main res non-lattice case-001}
Assume that $X_1$ is non-lattice, $\bb E (X_1) = 0$, $\bb E (X^2_1)= \sigma^2$,  and that there exists $\delta > 0$ 
such that $\bb E (|X_1|^{2+\delta})  < \infty.$
Let $v_0 >0$. Then, there exist a constant $c>0$ and a sequence $(\alpha_n)_{n \geq 1}$ of positive numbers
 satisfying $\lim_{n\to\infty}\alpha_n = 0$, 
such that  for any $n \geq 2$,  $x\in \bb R$, $y \in  \supp \check V$ and $v\geq v_0$, 
\begin{align*}
& \left| \bb{P}  \Big( x+S_n \in [y, y +  v),  \tau_x > n - 1 \Big) 
 - (1+\alpha_n)  \frac{V_n(x) }{ \sigma^3 n^{3/2} }  
  \int_{y}^{y+v}  \check V_n(z)  p \left( \frac{x}{\sigma\sqrt{n}} , \frac{z}{\sigma\sqrt{n} } \right) dz \right| \notag\\
 & \qquad\qquad\qquad  \leq c \frac{\left(n^{-\frac{\delta}{8} }  + V_{n}(x) \right) 
   \left(  \int_{y}^{y+v}  \check V_n(z) dz + v n^{-\frac{\delta}{4}} \right) }{  n^{3/2} } n^{-\frac{\delta}{8(3+\delta)} } \log n.
 \end{align*}
 \end{theorem}

Since $\check V(0)>0$, the previous theorem holds true uniformly for $y\in [0,\infty)$. 
The proof of Theorem \ref{main res non-lattice case-001} is given in Section \ref{sec-proof of equiv for intervals} 
based  on the more precise result of Theorem \ref{theorem-n3/2-upper-lower-bounds} 
which is stated in the form of two-sided bounds.  

An equivalence result can be obtained by restricting the range of  the variables $x, y$ 
in Theorem \ref{main res non-lattice case-001}, 
following an approach  similar to that used in Section \ref{sec-equiv-results-000}. 
\begin{theorem}\label{TH-non-latticeEQUIVALENCE-001} 
Assume that $X_1$ is non-lattice, $\bb E (X_1) = 0$, $\bb E (X^2_1)= \sigma^2$,  and that there exists $\delta > 0$ 
such that $\bb E (|X_1|^{2+\delta})  < \infty.$ 
Let $x_0\in \supp V$, $y_0\in \supp \check V$, $v_0>0$ and $q<\frac{\delta}{8(3+\delta)}$.
Then, as $n \to \infty$,  uniformly in $x\geq x_0$, $y\geq y_0$, $v\geq v_0$ 
satisfying $|y-x|+v\leq \sigma\sqrt{2q n\log n}$ or $|y+x|+v\leq \sigma\sqrt{2q n\log n}$,
\begin{align*} 
\bb{P} \left(x+S_n \in [y,y+v), \tau_x >n-1\right) 
\sim  \frac{V_{n}(x)}{\sigma^3 n^{3/2}} 
\int_{y}^{y+v} \check V_{n}(z) p \left( \frac{x}{\sigma\sqrt{n}}, \frac{z}{\sigma\sqrt{n}}   \right) dz. 
\end{align*}
\end{theorem}
\begin{proof}
Note that, for $y\geq y_0$, $v\geq v_0$ and $z\in [y,y+v]$ we have 
$ |x-z|\leq |x-y|+v\leq   \sigma\sqrt{2qn\log n}.$
Therefore, the assertion of the theorem follows from Theorem  \ref{main res non-lattice case-001},
in the same way as in the proof of Theorem \ref{Th-lattice-general equiv-c001}.
\end{proof}

The heat kernel approximations established in Theorems \ref{main res non-lattice case-001} and \ref{TH-non-latticeEQUIVALENCE-001}
are effective across various regimes dictated by the behavior of $x$, $y$ and $v$. 
These results provide a unified framework that generalizes and improves upon several previous findings. 
Specifically, from Theorems \ref{main res non-lattice case-001} and \ref{TH-non-latticeEQUIVALENCE-001}, we recover and enhance 
the main results of \cite{Carav05}, \cite{VatWacht09} and \cite{Don12} for the Gaussian case. 
Furthermore, by setting $v = u\sqrt{n}$ in Theorem \ref{main res non-lattice case-001}, 
one recovers the conditioned central limit theorem \eqref{UniformCondInt-001}, albeit with a slower convergence rate.

To provide further insight, we analyze several asymptotic regimes depending on the behavior of $x$  and $y$. 
Under the conditions of Theorem \ref{TH-non-latticeEQUIVALENCE-001}, 
the following equivalences hold, 
where the sequence $(\alpha_n)_{n\geq 1}$ converges to zero and the sequence $(\beta_n)_{n\geq 1}$ diverges to infinity.
\begin{itemize}
\item[b1:] 
As $n \to \infty$,  uniformly in $x\geq x_0$, $y\geq y_0$, $v\geq v_0$ satisfying $\frac{x}{\sqrt{n}} \leq \alpha_n$, 
$\frac{|xy|+|xv|}{n} \leq \alpha_n$,  and $|x-y|+v\leq \sigma\sqrt{2q n\log n}$, 
\begin{align} \label{non-lattice-corrol-geneq01}
\bb{P} \left(x+S_n \in [y,y+v], \tau_x >n-1\right) 
\sim  \frac{2V(x)}{\sigma^3 n^{3/2}} 
\int_{y}^{y+v} \check V(z)  \phi \left(\frac{z}{\sigma\sqrt{n}}   \right) dz. 
\end{align}
In particular, as $n\to\infty,$ uniformly in $x\geq x_0$, $y\geq y_0$, $v\geq v_0$ satisfying the conditions listed above and $\frac{|yv|+v^2}{n}\leq \alpha_n$, 
\begin{align} \label{non-lattice-corrol-geneq01b}
\bb{P} \left(x+S_n \in [y,y+v], \tau_x >n-1\right) 
\sim  \frac{2V(x) \int_{y}^{y+v} \check V(z)dz}{\sigma^3 n^{3/2}  } 
 \phi \left(\frac{y}{\sigma\sqrt{n}} \right). 
\end{align}

\item[b2:] 
As $n \to \infty$,  uniformly in $x\geq x_0$, $y\geq y_0$, $v\geq v_0$ satisfying $\frac{|y|+|v|}{\sqrt{n}} \leq\alpha_n$,
$\frac{|xy|+|xv|}{n}\leq \alpha_n$ 
and $|x-y|+v\leq \sigma\sqrt{2q n\log n}$, 
\begin{align} \label{non-lattice-corrol-geneq02}
\bb{P} \left(x+S_n \in [y,y+v], \tau_x >n-1\right) 
\sim  \frac{V(x) \int_{y}^{y+v} \check V(z) dz}{\sigma^3 n^{3/2}} 
  \phi \left(\frac{x}{\sigma\sqrt{n}} \right). 
\end{align}

\item[b3:] 
As $n \to \infty$,  uniformly in $x\geq \beta_n$, $y\geq \beta_n$, $v\geq v_0$ and $|y-x|+v\leq \sigma\sqrt{2q n\log n}$, 
\begin{align} \label{non-lattice-corrol-eq03}
\bb{P} \left(x+S_n \in [y,y+v], \tau_x >n-1\right) 
  \sim \frac{1}{\sigma \sqrt{n} }  \int_{y}^{y+v}  \psi\left( \frac{x}{\sigma\sqrt{n}} , \frac{z}{\sigma\sqrt{n} } \right) dz. 
\end{align}
\end{itemize}
Assertions b1-b2 follow from Theorem \ref{TH-non-latticeEQUIVALENCE-001} 
via the equivalences established in Lemmas \ref{Lem-p equiv-x to 0} and \ref{lem-inequality for L}. 
The claim b3 follows by Rogozin's estimate \eqref{Rogizin estim-001}. 
We note that because $V(0) > 0$ and $\check V(0) > 0$, the parameters $x_0$ and $z_0$ in b1-b2 may be set to zero: $x_0=0$ and $y_0=0$. 
These claims generalize results previously obtained in \cite{VatWacht09, Don12, GX-2024-AIHP}.

Let us finish this section by stating a result which is similar to the results in Section \ref{sec:Main result and consequences}, 
provided an additional condition on $v$ is assumed. 
Theorem \ref{main res non-lattice case-001} and Lemma \ref{lem-holder prop for ell-001} imply the following.
\begin{corollary} \label{non-lattice-equiv-xlarge-001}
Assume that $X_1$ is non-lattice, $\bb E (X_1) = 0$, $\bb E (X^2_1)= \sigma^2$,  and that there exists $\delta > 0$ 
such that $\bb E (|X_1|^{2+\delta})  < \infty.$
Let $v_0 >0$. 
Then, there exist a constant $c>0$ and a sequence $(\alpha_n)_{n \geq 1}$ of positive numbers with $\alpha_n \to 0$ as $n \to \infty$, 
such that  for any $n \geq 2$,  $x\in \bb R$, $y \in  \supp \check V$ and $v \geq v_0$,  
\begin{align*}
& \left| \bb{P}  \Big( x+S_n \in [y,y + v),  \tau_x > n - 1 \Big) 
 - \left(1+ \alpha_n \right)  \frac{V_n(x) \int_{y}^{y+v}  \check V_n(z)   dz }{ \sigma^3 n^{3/2} }  
   p \left( \frac{x}{\sigma\sqrt{n}} , \frac{y}{\sigma\sqrt{n} } \right) \right|  \notag\\
 & \leq  c \frac{\left(n^{-\frac{\delta}{8} }  + V_{n}(x)\right)  \left(  \int_{y}^{y+v}  \check V_n(z) dz + v n^{-\frac{\delta}{4}}  \right) }{  n^{3/2} }
 \left(n^{-\frac{\delta}{8(3+\delta)} } \log n + \frac{v}{\sqrt{n}}\right).
 \end{align*}
\end{corollary}

\begin{proof}[Proof of Corollary \ref{non-lattice-equiv-xlarge-001}]
By \eqref{lem-holder prop for ell-002} of Lemma \ref{lem-holder prop for ell-001}, there exists a constant $c>0$ such that, 
for any $x, y \in \bb R$, $v \in \bb R_+$ and $z\in [y,y+v]$,
\begin{align*} 
\left|p \left( \frac{x}{\sigma\sqrt{n}} , \frac{z}{\sigma\sqrt{n} } \right) - p \left( \frac{x}{\sigma\sqrt{n}} , \frac{y}{\sigma\sqrt{n} } \right)\right| 
 \leq c \frac{v}{\sigma\sqrt{n}},
\end{align*}
which together with Theorem \ref{main res non-lattice case-001} implies the assertion of the corollary.
\end{proof}
This corollary allows us to extend the results of  Section \ref{Sec-CondLocLimTh-lattice-001}
stated for lattice distributions to the case of non-lattice distributions, which because of the space limitations will not be pursued here.

\subsection{Extension for functions}
All of the above results can be extended from intervals to a more general case of functions. 
For instance,  the following theorem holds.  

\begin{theorem}\label{UnifballotLLTheor-001} 
Assume that $X_1$ is non-lattice, $\bb E (X_1) = 0$, $\bb E (X^2_1)= \sigma^2$,  
and that there exists $\delta > 0$ 
such that $\bb E (|X_1|^{2+\delta})  < \infty.$ 
Let $f:\bb R \to \bb R_+$ be a function such that
$y\mapsto f(y)(1+|y|)$ is directly Riemann integrable. 
Then, for any constant $v_0>0$, 
as $n\to\infty$, uniformly in $x,y \in \bb R$ and $v \geq  v_0$,
\begin{align*} 
 \bb{E} \left(  f \left(\frac{x+S_n-y}{v} \right); \tau_x >n-1\right) 
 &=  \frac{V_{n}(x)}{\sigma^3 n^{3/2}} 
\int_{\bb R}  f\left(\frac{z-y}{v}\right) \check V_{n}(z) p \left( \frac{x}{\sigma\sqrt{n}}, \frac{z}{\sigma\sqrt{n}}   \right) dz \\
&  \quad +   v \frac{\left(1  + V_{n}(x)\right) \left( 1 + v  + \check V_{n}(y)   \right)}{n^{3/2}} o(1).
\end{align*}
\end{theorem}

Theorem \ref{UnifballotLLTheor-001} can be obtained from Theorem \ref{theorem-n3/2-upper-lower-bounds} by standard method,
so the proof will not be detailed here.
By using a truncation argument and assuming convenient integrability assumptions on the target function $f$  
one can obtain equivalences similar to those given by Theorem \ref{TH-non-latticeEQUIVALENCE-001} 
and cases b1, b2 and b3. 

\subsection{A local limit theorem for the exit time}
The following assertion gives the asymptotic behaviour of the local behaviour of the exit time  $\tau_x$.
Denote 
\begin{align} \label{def kappanonlatt 001}
\varkappa = \int_{\bb R_-} \check V(y) dy = \int_{\bb R_+} \check V(y)  \bb P (\check X_1 >y ) dy.
\end{align}
Under conditions $\bb E (X_1) = 0$ and $\bb E (X^2_1)= \sigma^2  > 0$,   
it holds $0 < \varkappa < \infty$.

\begin{theorem} \label{theorem local for tau non-lattice }
Assume that $X_1$ is non-lattice, $\bb E (X_1) = 0$, $\bb E (X^2_1)= \sigma^2$,  
and that there exists $\delta > 0$ 
such that $\bb E (|X_1|^{2+\delta})  < \infty.$ 
Then, as $n\to\infty,$ uniformly in $x\in \bb R$,
\begin{align*} 
\bb{P} ( \tau_x =n ) 
   =  (1+o(1)) \phi\left(\frac{x}{\sigma\sqrt{n}}\right) \frac{2 V(x) \varkappa}{\sigma^3 n^{3/2}} 
 +  c  \frac{n^{-\frac{\delta}{6} }  + V_{n}(x) n^{-\frac{\delta}{8(3+\delta)} } \log n}{  n^{3/2} }.
\end{align*}
\end{theorem}

The proof of this theorem is deferred to Section \ref{sec proof of equiv local for tau}.

As a consequence we obtain the following equivalence result.
\begin{theorem} \label{theorem local equiv for tau non-lattice}
Under the conditions of Theorem \ref{theorem local for tau non-lattice }, 
for any $x_0\in \supp V$ and $0<q < \frac{\delta}{8(3+\delta)}$,
as $n \to \infty$, uniformly in $x \in \bb R$ such that $x_0\leq x \leq \sigma \sqrt{2q n \log n}$, 
\begin{align*} 
\bb{P} ( \tau_x =n ) 
   \sim \phi\left(\frac{x}{\sigma\sqrt{n}}\right) \frac{2V(x) \varkappa}{\sigma^3 n^{3/2}}.
\end{align*}
\end{theorem}
To our knowledge, Theorem \ref{theorem local equiv for tau non-lattice} is the first result in the literature providing 
an equivalence for $x$ beyond $\sqrt{n}$.

To compare this asymptotic with the results in \cite{GX-2024-AIHP},
let us analyze two particular cases. 
Firstly, from Theorem \ref{theorem local equiv for tau non-lattice} we have, for any $x_0\in \supp V$,  as $n\to\infty$, 
uniformly in $x\geq x_0$ satisfying $\frac{x}{\sqrt{n}}\to 0$, 
\begin{align*} 
\bb{P} ( \tau_x =n ) \sim \frac{2 V(x) \varkappa }{\sqrt{2\pi}\sigma^3 n^{3/2}}.  
\end{align*}
Secondly, uniformly in $x\asymp \sqrt{n}$ as $n\to\infty$, 
\begin{align*} 
\bb{P} ( \tau_x =n ) \sim \frac{2 \phi^+\left(\frac{x}{\sigma\sqrt{n}}\right) \varkappa }{\sqrt{2\pi}\sigma^2 n}. 
\end{align*}


\section{Auxiliary statements}\label{Sec-Auxiliary statements}
\subsection{Effective local limit theorems} \label{subsec-nonasymptoticLLT}

The local limit theorem, which is a refinement of the central limit theorem, has attracted significant interest since the foundational work of Gnedenko \cite{Gned48}, Shepp \cite{She64}, and Stone \cite{Sto65}.
In this section, we present effective local limit theorems for the random walk $(S_n)_{n \geq 1}$, 
involving target functions and explicit convergence rates.
These results will play a key role in deriving error terms in the conditioned local limit theorems.

For lattice valued random walks, the following local limit theorem is well known 
(see e.g. Ibgarimov and Linnik \cite{Ibrag-Linnik1971} and Petrov \cite{Petrov1971}). 
Recall that $\phi_v$  denotes the normal density of mean $0$ and variance $v$.
 \begin{theorem} \label{Th-lattice-LLT-general}
Assume that $X_1$ is $(\hbar, a)$-lattice, $\bb E (X_1) = 0$, $\bb E (X^2_1)= \sigma^2$,  and that there exists $\delta > 0$ 
such that $\bb E (|X_1|^{2 + \delta})  < \infty.$ 
Let $\delta_1 = \min\{ 1, \delta \}$. 
Then, there exists a constant $c >0$ such that, for any $n\geq 1$,
\begin{align*}
\sup_{z \in \hbar \bb Z+na}\left|  \bb{P}  (S_n = z)  -  \hbar \, \phi_{\sigma^2 n}\left(z \right) \right|
\leq  \frac{c}{n^{ (1 + \delta_1 )/2 }}. 
\end{align*}
\end{theorem}

For non-lattice valued random walks, we need a more sophisticated local limit theorem. 
Let $f, g:\mathbb R \mapsto \mathbb R_+$ 
be Borel measurable functions and $\alpha >0$.
We say that $g$ $\alpha$-dominates $f$ and we use notation $f \leq_{\alpha} g$ if
\begin{align} \label{def_upper_envelope_001} f(u)  \leq  g(u+v), \ \forall u\in \bb R,\ \forall |v| \leq \alpha.   
\end{align}
For any measurable function $y \mapsto g(y)$ on $\bb R$, 
we denote by  $\| g\|_1=\int_{\bb R} |g(y)|dy$ the $L_1$-norm of $g$.
The following theorem is from \cite[Theorem 2.7]{GX-2024-AIHP}. 

 \begin{theorem} \label{LLT-general}
Assume that $X_1$ is non-lattice, $\bb E (X_1) = 0$, $\bb E (X^2_1)= \sigma^2$,  and that there exists $\delta > 0$ 
such that $\bb E (|X_1|^{2 + \delta})  < \infty.$ 
Let $\delta_1 = \min\{ 1, \delta \}$. 
Then, there exists a constant $c >0$ with the following property: 
for any $\alpha \in (0, \frac{1}{2})$ there exists a constant $c_{\alpha} > 0 $ such that: 

\noindent 1. For any  integrable functions $f,g: \bb R\mapsto \bb R_+$ satisfying $f\leq_{\alpha} g$ and $n\geq 1$,
\begin{align}\label{LLT-general001}
  \bb{E} f (S_n) 
- (1 + c \alpha) \int_{\bb R } g (t) \phi_{\sigma^2 n}(t)  dt 
\leq  \frac{c_{\alpha}}{n^{ (1 + \delta_1 )/2 }} \left\Vert g\right\Vert_1. 
\end{align}

\noindent  2. For any integrable functions $f,h: \bb R\mapsto \bb R_+$ satisfying $f\geq_{\alpha} h$ and $n\geq 1$,
\begin{align}\label{LLT-general002}
  \bb{E} f (S_n) 
 -  \int_{\bb R }  \big[ h(t) -  c \alpha f(t) \big]  \phi_{\sigma^2 n}(t)  dt  
  \geq   -  \frac{c_{\alpha}}{n^{ (1 + \delta_1) /2 }}  \left\Vert f \right\Vert_1. 
\end{align}
\end{theorem}

Note that in Theorem \ref{LLT-general}, the constant $c_{\alpha}$ may diverge as $\alpha \to 0$.
Without loss of generality, we will assume in what follows that the map $\alpha\mapsto c_{\alpha}$ is increasing as $\alpha\to 0.$
It is well known that in the absence of further assumptions on the law of $X_1$, this rate of divergence cannot be controlled.
Another feature of Theorem \ref{LLT-general} is that it does not require the function $f$  to be directly Riemann integrable.
It is for this reason that the theorem is formulated as a two-sided bound.

\subsection{Properties of the heat kernel}
The following result is from \cite[Lemma 2.3]{GX-2024-CCLT}. 

\begin{lemma} \label{lem-inequality for L}
There exists a constant $c>0$ such that for any $x, x' \in \bb R$, 
\begin{align*} 
\left| \frac{L(x')}{L(x)} -1 \right| \leq c | x' - x |.
\end{align*}
 \end{lemma}

We continue with the important Lipschitz property of function $y\mapsto \ell(x,y)$,
which is crucial in the proof of Theorems \ref{UnifballotLLTheor-001}, \ref{T-Caravenatype-lattice 001} and \ref{t-B 002}.
 
\begin{lemma} \label{Lipschiz  property for ell_H in y} 
(1) There exists a constant $c > 0$ such that, for any $x, y \in \bb R$,
\begin{align}\label{derivative-ell-H-ineq}
\left| \frac{\partial}{\partial y}  \ell(x, y) \right| 
\leq  c \left( e^{- \frac{(x-y)^2}{3} } + e^{- \frac{(x+y)^2}{3}} \right). 
\end{align} 
(2) There exists a constant $c > 0$ such that, for any $x, y, a \in \bb R$, 
 \begin{align} \label{ell-intergr-bound-002}
\left| \ell(x, y+a)  -   \ell(x, y) \right|  \leq c |a|. 
\end{align}
(3) There exists a constant $c > 0$ such that, for any $x, y \in \bb R$ and $|a|\leq 1$, 
 \begin{align} \label{ell-intergr-bound-001}
\left| \ell(x, y+a)  -   \ell(x, y) \right| 
\leq  c \left( e^{- \frac{(x-y)^2}{4} } + e^{- \frac{(x+y)^2}{4}} \right) |a|. 
\end{align}
(4) There exists a constant $c >0$ such that for any $x, y \in \bb R$ with $|y| \leq 1$, 
\begin{align}\label{bound-ell-qwe}
|\ell(x, y)| \leq c e^{- \frac{x^2}{4}} |y|. 
\end{align}
\end{lemma}

\begin{proof} 
(1) We first prove \eqref{derivative-ell-H-ineq} in the case $|x|\leq 1$.
By \eqref{def of func ell_H-001}, 
we have, for any $x, y \in \bb R$, 
\begin{align} \label{derivative-ell-identity}
\frac{\partial}{\partial y} \ell(x, y) 
=  \frac{(x-y) \phi(x- y) + (x+y) \phi(x+y)}{H(x)} 
=  \frac{\phi(x- y) + \phi(x+ y)}{L(x)} - y \ell(x, y). 
\end{align}
Since $L(x)\geq L(1)$ for $|x|\leq 1$, we have that, for $|x| \leq 1$ and $y \in \bb R$, 
\begin{align*} 
\bigg| \frac{\phi(x- y) + \phi(x+ y)}{L(x)} \bigg|  
\leq c \left( e^{- \frac{(x-y)^2}{2} } + e^{- \frac{(x+y)^2}{2}} \right). 
\end{align*}
Since $|1 - e^{-t}| \leq |t| e^{|t|}$ for $t \geq 0$ and $|\frac{x}{H(x)}| \leq \frac{1}{H(1)}$ for $|x| \leq 1$, 
we get that, for $|x| \leq 1$ and $y \in \bb R$,  
\begin{align}\label{bound-ell-abc-01}
\big| y \ell(x, y) \big| 
& =  \frac{1}{\sqrt{2 \pi}} e^{- \frac{(x- y)^2}{2}}  \left|  \frac{ y \left( 1 - e^{- 2 xy}  \right) }{H(x)} \right|  \notag\\ 
& \leq \frac{1}{\sqrt{2 \pi}} e^{- \frac{(x-y)^2}{2}}   \left|  \frac{  2 x y^2 e^{|2 xy|}  }{H(x)} \right|   
 \leq  c  e^{- \frac{(x-y)^2}{2}}  y^2 e^{2 |y|}
 \leq  c'  e^{- \frac{(x-y)^2}{3}}. 
\end{align}
Therefore, we obtain \eqref{derivative-ell-H-ineq} for $|x|\leq 1$. 

For $|x|>1$, we have $H(x)\geq H(1)>0$, and therefore, by \eqref{derivative-ell-identity} and the inequality 
$|t| e^{- \frac{t^2}{2}} \leq c e^{- \frac{t^2}{3}}$ for $t \in \bb R$, 
\begin{align*}
\left| \frac{\partial}{\partial y} \ell(x, y)  \right|
= \left| \frac{(x-y) \phi(x- y) + (x+y) \phi(x+y)}{H(x)}  \right|
\leq  c \left( e^{- \frac{(x-y)^2}{3} } + e^{- \frac{(x+y)^2}{3}} \right),  
\end{align*}
which ends the proof of \eqref{derivative-ell-H-ineq}. 

(2) The inequality \eqref{ell-intergr-bound-002} is a consequence of \eqref{derivative-ell-H-ineq}. 

(3) Since $\ell(x, y+a) - \ell(x, y) = a \frac{\partial}{\partial y} \ell(x,\tilde y)$
with $\tilde y = y+\theta a$ and $\theta \in (0,1)$,
using \eqref{derivative-ell-H-ineq} together with the inequality $(t + b)^2\geq \frac{3}{4} t^2 - 3 b^2$ for $t, b \in \bb R$, 
we obtain  \eqref{ell-intergr-bound-001}.  

(4) Taking $y = 0$ in \eqref{ell-intergr-bound-001} and using the fact that $\ell(x, 0) = 0$, 
we get \eqref{bound-ell-qwe}. 
\end{proof}

We also need the following Lipschitz continuity of the function $x \mapsto \ell(x, y)$.  

\begin{lemma} \label{Holder prop for int ell} 
There exists a constant $c>0$ such that, for any $y \in \bb R$ and $x, x' \in \bb R$, 
 \begin{align}\label{Lipschitz-ell-x-001}
\left|  \ell(x', y)  -   \ell(x, y)  \right|  \leq c |x-x'|. 
\end{align}
In addition, there exists a constant $c>0$ such that $|\ell (x, y)| \leq c$ for any $x, y \in \bb R$. 
\end{lemma}

\begin{proof}
We refer to \cite[Lemma 2.5]{GX-2024-CCLT} for the proof of \eqref{Lipschitz-ell-x-001}. 
Now we show the boundedness of $\ell$. For $|x| \geq 1$ and $y \in \bb R$, it follows from the fact that $\psi$ is bounded and $H$ has a strictly positive lower bound.
For  $|x| < 1$ and $y \in \bb R$, using \eqref{Lipschitz-ell-x-001}, we have $|\ell (x, y)| \leq |\ell (x, y) - \ell (1, y)| + |\ell (1, y)| \leq c |x| + c' \leq c + c'$, as required. 
\end{proof}

We now establish Lipschitz continuity properties for the bivariate function $(x,y) \mapsto p(x,y)$ on $\bb R \times \bb R$.

\begin{lemma} \label{lem-holder prop for ell-001}
(1) There exists a constant $c > 0$ such that, for any $x, y \in \bb R$ with $|y| \leq 2$,
\begin{align}\label{deriv of ell bound aa001}
\left| \frac{\partial}{\partial y}  p(x, y) \right| 
\leq  c e^{- \frac{x^2}{4}}  |y|. 
\end{align} 
(2) There exists a constant $c > 0$ such that for any $x \in \bb R$, $|y| \leq 2$ and $|a|\leq 1$, 
\begin{align}\label{Lipschitz-continuity-pxy}
| p(x,y+a) - p(x, y) | \leq c e^{- \frac{x^2}{5}} |a|. 
\end{align}
(3) There exists a constant $c > 0$ such that $|p(x, y)| \leq c$ for any $x, y \in \bb R$.  \\
(4) There exists a constant $c > 0$ such that for any $x, y, a \in \bb R$, 
\begin{align}\label{lem-holder prop for ell-002}
| p(x,y+a) - p(x, y)| \leq c |a|. 
\end{align}
\end{lemma}

\begin{proof} 
(1) By \eqref{def of func ell_H-002}, we have 
\begin{align*} 
\frac{\partial}{\partial y} p(x,y) 
&=  \frac{x(\phi(x-y) + \phi(x+ y))}{H(x)H(y)} -\frac{y \psi(x,y)}{H(x)H(y)} - \frac{2\phi(y)\psi(x, y) }{H(x)H(y)^2}. 
\end{align*}
Since $H(x)=xL(x)$, 
we obtain
\begin{align} \label{derivative of p 001}
\frac{\partial}{\partial y} p(x,y) =  \frac{ Q_1(x,y) }{xy^2L(x)L(y)^2} - \frac{y \psi(x,y)}{H(x)H(y)}, 
\end{align}
where
\begin{align} 
&Q_1(x,y) = x y L(y) \Theta(x,y) -  2\phi(y) \psi(x,y), \label{def of Q(x,y)-qqqq002}\\
&\Theta(x,y) = \phi(x-y) + \phi(x+y). \label{def of Q(x,y)-qqqq004}
\end{align}
We first deal with the second term in \eqref{derivative of p 001}. 
Since $\psi(x,y) = \ell(x,y)H(x)$, 
by \eqref{bound-ell-qwe} of Lemma \ref{Lipschiz  property for ell_H in y}, 
there exists $c>0$ such that, for any $x\in \bb R$ and $|y|\leq 2$,
 \begin{align} \label{bound for Q_2 bb001}
\left|\frac{y \psi(x,y)}{H(x)H(y)}\right| = \left|\frac{\ell(x,y)}{L(y)}\right| 
 \leq  c e^{- \frac{x^2}{4}}  |y|.
\end{align}
Next we deal with the first term in \eqref{derivative of p 001}. 
Since
\begin{align*} 
\Theta(x,0) = 2\phi(x),\quad \Theta'_y(x,0) = 0,\quad  \Theta''_{y^2}(x,y) = \phi''(x-y) + \phi''(x+y), 
\end{align*}
by Taylor's formula, for any $x,y\in \bb R$, 
there exists $\tilde y \in \bb R$ with $|\tilde y|\leq |y|$ such that
\begin{align} \label{Theta-aa001}
\Theta(x,y) = 2\phi(x) + \left( \phi''(x-\tilde y) + \phi''(x+\tilde y) \right)\frac{y^2}{2},
\end{align}
where the function $\phi''$ is bounded on $\bb R$.
In the same way, by \eqref{Def-Levydens}, for any $x, y \in \bb R $, there exists $\tilde y \in \bb R$ with $|\tilde y|\leq |y|$ such that
\begin{align} \label{expan for func psi 001}
\psi(x,y) = 2 x \phi(x) y -  \left(\phi'''(x-\tilde y) + \phi'''(x+\tilde y)\right)\frac{y^3}{6}. 
\end{align}
Substituting \eqref{Theta-aa001} and \eqref{expan for func psi 001} into \eqref{def of Q(x,y)-qqqq002}, we derive that 
\begin{align} \label{formula-Qxy-aa001}
Q_1(x,y) 
&= 2xy\phi(x)L(y) +\frac{1}{2}xy^3L(y)\left( \phi''(x-\tilde y) + \phi''(x+\tilde y) \right) \notag\\
& \quad - 4xy\phi(x)\phi(y) + \frac{1}{3}y^3\phi(y)\left( \phi'''(x-\tilde y) + \phi'''(x+\tilde y) \right).
\end{align}
As $L(0) = \frac{2}{\sqrt{2\pi}}$ and $L'(0) = 0$, there exists $\tilde y \in \bb R$ with $|\tilde y|\leq |y|$ such that  
\begin{align*} 
L(y) = \frac{2}{\sqrt{2\pi}} +L''(\tilde y) \frac{y^2}{2}, \quad \phi(y) = \frac{1}{\sqrt{2\pi}} +\phi''(\tilde y) \frac{y^2}{2}. 
\end{align*}
Plugging this into \eqref{formula-Qxy-aa001}, we arrive at
\begin{align*} 
Q_1(x,y)
&= \frac{4xy}{\sqrt{2\pi}}\phi(x) + xy^3\phi(x)L''(\tilde y) + \frac{1}{2}xy^3L(y)\left( \phi''(x-\tilde y) + \phi''(x+\tilde y) \right)\\
& \quad - \frac{4xy}{\sqrt{2\pi}}\phi(x)  -2xy^3\phi(x) \phi''(\tilde y)    + \frac{1}{3}y^3\phi(y)\left( \phi'''(x-\tilde y) + \phi'''(x+\tilde y) \right)\\
&= xy^3 \bigg( \phi(x)L''(\tilde y) + \frac{1}{2}L(y)\left( \phi''(x-\tilde y) + \phi''(x+\tilde y) \right)\\
& \qquad\qquad  -2\phi(x) \phi''(\tilde y)    + \frac{1}{3}\phi(y)\frac{\phi'''(x-\tilde y) + \phi'''(x+\tilde y)}{x} \bigg).
\end{align*}
Therefore, 
\begin{align} \label{bound Q1 general case aa001}
\frac{Q_1(x,y)}{xy^2 L(x)L(y)^2} 
= \frac{y}{L(x)L(y)^2} 
\bigg( &\phi(x)L''(\tilde y) + \frac{1}{2}L(y)\left( \phi''(x-\tilde y) + \phi''(x+\tilde y) \right)  \notag\\
& -2\phi(x) \phi''(\tilde y)   +  \frac{1}{3}\phi(y)\frac{\phi'''(x-\tilde y) + \phi'''(x+\tilde y)}{x} \bigg).
\end{align}

We first consider the case when $|x|\leq 1$.
Since $\phi'''(t) = (3t - t^3) \phi(t)$, we get 
\begin{align} \label{third deriv of psi aa001}
 \phi'''(x-y) + \phi'''(x+y)  
& =  \left[ 3 (x-y) - (x-y)^3   \right]  \phi(x-y) +  \left[ 3 (x+y) - (x+y)^3  \right]  \phi(x+y)  \notag\\
& = x\Theta(x, y) \left(3-x^2 - 3y^2 \right) + y\psi(x, y) \left(3 - y^2 - 3x^2 \right).
\end{align}
By the Lipschitz property of the function $\ell$ (see Lemma \ref{Holder prop for int ell}) and the fact that $\ell(x, 0) = 0$, 
we have $|\ell(x, y)| = |\ell(x, y) - \ell(x, 0)|  \leq c|y|$. 
Since $H(x)\leq \min\{|x|,1\}$,  we obtain
\begin{align} \label{first bound for func psi 001}
\left|\psi(x,y)\right| = \left|H(x) \ell(x,y)\right|  \leq c \min\{|x|,1\}|y|.
\end{align}
From \eqref{third deriv of psi aa001}, \eqref{first bound for func psi 001} and the fact that $|\Theta(x,y)| \leq 1$ (cf.\ \eqref{def of Q(x,y)-qqqq004}), 
it follows that there exists a constant $c>0$ such that for any $|x|\leq 1$ and $|y|\leq 2$,
\begin{align} \label{bound 3d derivative-001}
\left| \frac{\phi'''(x-y) + \phi'''(x+ y)}{x} \right| \leq c.
\end{align}
Plugging this into \eqref{bound Q1 general case aa001} yields that there exists $c>0$ such that for any $|x|\leq 1$ and $|y|\leq 2$, 
\begin{align*} 
\left|\frac{Q_1(x,y)}{xy^2L(x)L(y)^2}\right| \leq c \left|\frac{y}{L(x)L(y)^2}\right| \leq c|y|. 
\end{align*}
Together with \eqref{derivative of p 001} and \eqref{bound for Q_2 bb001}, this implies that 
 \eqref{deriv of ell bound aa001} holds in the case when $|x|\leq 1$ and $|y|\leq 2$.

Next we consider the case when $|x| > 1$ and $|y|\leq 2$. 
Since $|\tilde y| \leq |y|$, 
there exists $c >0$ such that $|\phi''(x \pm \tilde y)| \leq c e^{- \frac{x^2}{3}}$ 
and $|\phi'''(x \pm \tilde y)| \leq c e^{- \frac{x^2}{3}}$. 
Implementing this into \eqref{bound Q1 general case aa001} and using the fact that $\frac{1}{L(x)} \leq \frac{|x|}{H(1)}$ for $|x| > 1$,
we get 
\begin{align*}
\left|\frac{Q_1(x,y)}{xy^2L(x)L(y)^2}\right| \leq  c e^{- \frac{x^2}{4}} |y|. 
\end{align*}
Combining this with \eqref{derivative of p 001} and \eqref{bound for Q_2 bb001}, 
we obtain that  \eqref{deriv of ell bound aa001} holds in the case when $|x| > 1$ and $|y|\leq 2$.
This ends the proof of \eqref{deriv of ell bound aa001}. 

(2) The inequality \eqref{Lipschitz-continuity-pxy} follows from \eqref{deriv of ell bound aa001}. 

(3) Since $p(x, y) = \frac{\ell(x, y)}{H(y)}$, by the boundedness of $\ell$ (Lemma \ref{Holder prop for int ell}) and the fact that $|H(y)| \geq H(1)$ for $|y| \geq 1$, 
we get that $|p(x, y)| \leq c$ for $x \in \bb R$ and $|y| \geq 1$. 
For $x \in \bb R$ and $|y| < 1$, using \eqref{Lipschitz-continuity-pxy}, we have $|p(x, y)| \leq |p(x, y) - p(x, 1)| + |p(x, 1)| \leq c$, as required. 

(4) Finally, we prove \eqref{lem-holder prop for ell-002}. 
It is enough to consider the case where $|a|\leq 1$, 
otherwise the inequality is trivial because of the boundedness of the function $p$. 
In this case, for $|y|\leq 2$, the assertion follows from \eqref{Lipschitz-continuity-pxy}. 
For $|y| > 2$ and $|a| \leq 1$, since $|H(y+a)|$ and $|H(y)|$ are bounded away from $0$, 
by telescoping, using \eqref{ell-intergr-bound-002} of Lemma \ref{Lipschiz  property for ell_H in y} and the boundedness of $\ell$ (Lemma \ref{Holder prop for int ell}), we deduce that
\begin{align*} 
&\left| p(x, y+a)  -   p(x, y) \right| \notag\\
& \leq \frac{1}{|H(y)|}\left| \ell(x, y+a)  -  \ell(x, y)\right|  + |\ell(x, y+a)| \left| \frac{H(y+a)-H(y)}{H(y+a)H(y)}  \right|  \notag\\
& \leq c |a|, 
\end{align*}
completing the proof of \eqref{lem-holder prop for ell-002}. 
\end{proof}

\begin{lemma} \label{Lem-p equiv x to inf}
let $(\beta_n)_{n\geq 1}$ be a sequence diverging to infinity.
Then, as $n\to\infty$, uniformly in $x\geq \beta_n$ and $y\geq\beta_n$, 
\begin{align*} 
p(x, y) \sim \phi(x-y).
\end{align*}
\end{lemma}
\begin{proof}
For any $x\geq\beta_n$ and $y\geq \beta_n$, by the identity $\psi(x,y)=\phi(x-y)(1-e^{-2xy})$,
\begin{align*} 
\left|\frac{p(x,y)}{\phi(x-y)}-1\right| 
= \left|\frac{1}{H(x)H(y)} \left(1-e^{-2xy}\right)-1\right|
\leq  \left| \frac{1-H(\beta_n)^2}{H(\beta_n)^2}\right| +  \frac{ e^{-2\beta_n^2}}{H(\beta_n)H(\beta_n)}. 
\end{align*}
Taking the limit as $n\to \infty$, we obtain the assertion of the Lemma.
\end{proof}

\begin{lemma} \label{Lem-p equiv-x to 0}

There exists a constant $c>0$ such that, for any $\ee\in (0,1]$ and $x,y\in \bb R$ satisfying $|x|\leq \ee$ and  $|yx|\leq  \ee$,
\begin{align*} 
\left|\frac{p(x, y)}{e^{-y^2/2}/L(y)} -1\right| \leq c \ee.
\end{align*}
\end{lemma}
\begin{proof}
By \eqref{expan for func psi 001},
for any $x, y \in \bb R $, there exists $\tilde x \in \bb R$ with $|\tilde x|\leq |x|$ such that
\begin{align*} 
\psi(x,y)
&= \psi(y,x) = 2 x y \phi(y) -  \left(\phi'''(y-\tilde x) + \phi'''(y+\tilde x)\right)\frac{x^3}{6}.
\end{align*}
Taking into account that $\phi'''(y) = \phi(y)(3y-y^3)$, we obtain
\begin{align*} 
\psi(x,y) 
=  2 x y \phi(y) - 2xy \phi(y) e^{-\frac{1}{2} \tilde x^2} B(\tilde x,y),
\end{align*}
where
\begin{align*} 
B(x,y)
&= \frac{x^2}{12y} \left(e^{y x } (y- x) (3 - (y- x)^2)+ e^{-y x} (y+ x)  (3- (y+ x)^2)\right) \notag\\
&= \frac{x^2}{6}\left( (3-y^2-3x^2) \ch(xy) + x^2(x^2+3y^2-3) \frac{\sh(xy)}{xy}   \right).              
\end{align*}
Therefore,
\begin{align} \label{develpsi-002}
\frac{p(x,y)}{e^{-y^2/2}/L(y)}
&= \frac{1}{H(x)H(y)} \frac{\psi(x,y)}{e^{-y^2/2}/L(y)} =  \frac{2}{\sqrt{2\pi}L(x)} \left(1-  e^{-\frac{1}{2}\tilde x^2} B(\tilde x,y)\right).
\end{align}
The functions $u\mapsto \cosh(u)$ and $u\mapsto \frac{\sinh(u)}{u}$ are bounded for $|u|\leq c$, for any fixed constant $c>0$. 
This implies that, for $|x|\leq \ee$ and $|xy|\leq \ee$,  we have that
$|B(\tilde x, y)| \leq c \ee^2$. The assertion of the lemma now follows from \eqref{develpsi-002} 
using Lemma  \ref{lem-inequality for L}. 
\end{proof}

\subsection{Convolution properties of the heat kernel}
Recall that the normal density of mean $0$ and variance $t > 0$ is denoted by $\phi_t$
and that $\Phi_t$ stands for the corresponding cumulative distribution function, i.e. 
$\phi_t(x) = \frac{1}{\sqrt{2\pi t}} \exp\left(-\frac{x^2}{2t}\right)$ and
$\Phi_t(x) = \int_{-\infty}^x \phi_t(u) du$,  $x \in \mathbb{R}$. 
The heat kernel $\psi_t$ with scale parameter $t$ is defined by \eqref{heat-kernel-001}.

By elementary calculations we have the following.
\begin{lemma}\label{Lem-integral-phi-s}
For any $s, t > 0$ and $x, y\in \bb R$, we have 
\begin{align*}
\int_{\bb R_+} \phi_s(z-x) \phi_t(z-y) dz
 = \phi_{s+t}(x-y) \Phi_{\frac{s t}{s+t}}\left( \frac{tx+sy}{s+t}  \right). 
\end{align*}
\end{lemma}


The following lemma is similar to \cite[Lemma 2.39, page 41]{Gyrya-Salof-Coste2011}. 

\begin{lemma} \label{lem: product heat kernels}
For any $s,t>0$ and $x,y\in \bb R$, we have 
\begin{align*} 
\int_{\bb R_+} \psi_{s}(x, z) \psi_t(z, y) dz
= \psi_{s+t}(x, y).
\end{align*}
\end{lemma}

\begin{proof} 
By \eqref{heat-kernel-001} and the symmetry of $\phi_t$, we have, for any $s, t>0$ and $x, y\in \bb R$,
\begin{align*}
\int_{\bb R_+} \psi_{s}(x,z) \psi_t(y,z)dz 
& =  \int_{\bb R_+} \phi_s(z-x) \phi_t(z-y) dz +   \int_{\bb R_+} \phi_s(z+x) \phi_t(z+y) dz  \notag\\
& \quad  -  \int_{\bb R_+} \phi_s(z+x) \phi_t(z-y) dz -  \int_{\bb R_+} \phi_s(z-x) \phi_t(z+y) dz. 
\end{align*}
By Lemma \ref{Lem-integral-phi-s}, we get 
\begin{align*}
& \int_{\bb R_+} \phi_s(z-x) \phi_t(z-y) dz +   \int_{\bb R_+} \phi_s(z+x) \phi_t(z+y) dz \notag\\
& =  \phi_{s+t}(x-y) \Phi_{\frac{s t}{s+t}}\left( \frac{tx+sy}{s+t}  \right)
  +  \phi_{s+t}(x-y) \left(1-\Phi_{\frac{s t}{s+t}}\left(\frac{tx+sy}{s+t}  \right)\right)  \notag\\
& = \phi_{s+t}(x-y), 
\end{align*}
and similarly,
\begin{align*}
 -  \int_{\bb R_+} \phi_s(z+x) \phi_t(z-y) dz -  \int_{\bb R_+} \phi_s(z-x) \phi_t(z+y) dz  
= -\phi_{s+t}(x+y). 
\end{align*}
The desired identity follows. 
\end{proof}

We denote, for any $v \in (0, 1]$,  
\begin{align} \label{new-density-scale-v}
\ell_{v}(x, y) 
= \frac{\psi_{v}(x, y)}{H(x)},  \quad x, y \in \bb R. 
\end{align}

\begin{lemma} \label{convol-phi-psi-001} 
For any $v \in (0,1)$ and $x, y \in \bb R$, we have
\begin{align} \label{ConvoNormalLevy02}
\phi_{v} * \psi_{1-v}(x, y) := \int_{-\infty}^{\infty} \phi_{v} (y-z) \psi_{1-v}(x, z) dz 
=  \psi (x, y) 
\end{align} 
and 
\begin{align}\label{Convo-normal-ell-01}
\phi_{v} * \ell_{1-v}(x, y) := \int_{-\infty}^{\infty} \phi_{v} (y-z) \ell_{1-v}(x, z) dz 
=  \ell (x, y). 
\end{align}
\end{lemma}

\begin{proof}
Using Lemma \ref{lem: product heat kernels}, we have, for any $v \in (0,1)$ and $x, y \in \bb R$, 
\begin{align*} 
  \int_{0}^{\infty} \phi_{v}(y-z)  \psi_{1-v}(x,z)   dz  
 = \psi(x,y) + \int_{0}^{\infty} \phi_{v}(y+z) \psi_{1-v}(x, z)  dz.  
\end{align*}
By a change of variable $z = - z'$ and the fact that $\psi_{1-v}(x, -z') = - \psi_{1-v}(x, z')$, 
\begin{align*}
\int_{0}^{\infty} \phi_{v}(y+z) \psi_{1-v}(x, z)  dz
& = \int_{-\infty}^0  \phi_{v}(y - z') \psi_{1-v}(x, -z')  dz' \notag\\
& = - \int_{-\infty}^0  \phi_{v}(y - z') \psi_{1-v}(x, z')  dz',
\end{align*}
which proves \eqref{ConvoNormalLevy02}. 
By \eqref{new-density-scale-v}, the identity \eqref{Convo-normal-ell-01} is a consequence of \eqref{ConvoNormalLevy02}. 
\end{proof}

\subsection{Conditioned central limit theorem and Rogozin estmate}
The conditioned integral limit theorem (Theorem \ref{introTheor-probtauUN-001}), 
implies the following integral form of the Gaussian heat kernel approximation 
with explicit rate of convergence, whose proof is straightforward.  

\begin{lemma}\label{New-ApplCondLT-002}
There exists $c > 0$ such that, for any $x\in \bb R$, $n \geq 1$ 
and any differentiable function $\varphi$ on $\bb R$ satisfying $\lim_{t \to \pm \infty} \varphi(t) = 0$, 
\begin{align*} 
\left| \bb E \left( \varphi \left(\frac{x+S_{n}}{\sigma \sqrt{n}}\right); \tau_{x}>n \right) 
 -   \frac{V_n(x)}{ \sigma \sqrt{n}}  \int_{\bb R _{+}}  \varphi(t) 
\ell \left(\frac{x}{\sigma\sqrt{n}}, t \right)dt \right| 
 \leq  c \frac{R_{n,\delta}(x)}{\sqrt{n}}   \int_{\bb R}  | \varphi'(t) | dt, 
\end{align*}
where $R_{n,\delta}(x)$ is defined in Theorem \ref{introTheor-probtauUN-001}. 
\end{lemma}

\begin{proof}
Since the function $t \mapsto \varphi(t)$ is differentiable on $\bb R$ and vanishes as $t \to \pm \infty$, 
using integration by parts,  we get that for any $x \in \bb R$,
\begin{align*} 
\bb E \left( \varphi \left(\frac{x+S_{n}}{\sigma \sqrt{n}}\right); \tau_{x}>n \right) 
 = \int_{\bb R}  \varphi'(t) \ \bb{P}\left(  \frac{x+S_{n}}{\sigma  \sqrt{n}} > t, \tau_{x}>n \right)  dt.  
 \end{align*}
By Theorem \ref{introTheor-probtauUN-001} (\eqref{UniformCondtau-001} and \eqref{UniformCondInt-001}) 
and the fact that $y\mapsto \ell(x,y)$ is a density function on $\bb R_+$, 
 there exists $c >0$ such that for any $n \geq 1$, $x \in \bb R$ and $t \in \bb R$,  
\begin{align*} 
 \left|  \bb{P} \left(  \frac{x+S_{n}}{\sigma \sqrt{n}} > t,  \tau_{x}>n\right) - 
\frac{V_n(x)}{\sigma \sqrt{n}}  \int_{ \max\{t, 0\} }^{\infty}  \ell \left( \frac{x}{\sigma \sqrt{n}}, y   \right)  dy  \right|
 \leq  c \frac{R_{n,\delta}(x)}{\sqrt{n}}. 
\end{align*}
Using integration by parts and the fact that $\ell ( \frac{x}{\sigma \sqrt{k}}, 0) = 0$, we have 
\begin{align*} 
\int_{\bb R}  \varphi'(t) \int_{ \max\{t, 0\} }^{\infty}  \ell \left( \frac{x}{\sigma \sqrt{k}},  y \right)  dy dt
=\int_{\bb R _{+}}  \varphi(t) \ell \left( \frac{x}{\sigma \sqrt{k}},  t \right) dt. 
\end{align*}
This concludes the proof of the lemma. 
\end{proof}

By Rogozin \cite[Corollary 4]{Rog77} the following holds: 
\begin{lemma} \label{lemma-Rogozin-001}
Assume that $\bb E (X_1) = 0$, $\bb E (X^2_1)= \sigma^2  > 0$  and that there exists $\delta > 0$ 
such that $\bb E (|X_1|^{2+\delta})  < \infty.$ Then
there exists a constant $c>0$ such that, for any $y \geq 1$, it holds 
\begin{align} \label{Rogizin estim-001}
0\leq  \frac{V(y)}{y} - 1  \leq c y^{-\delta}.
\end{align} 
\end{lemma}

\section{Proofs in the lattice case} \label{SecProof Theor-latticecase-001}

\subsection{Caravenna type local limit theorem in the lattice case}  \label{Sec Caravenna formulation latticecase-001}
In this subsection we formulate the following  Caravenna type local limit theorem in the lattice case. 

\begin{theorem} \label{T-Caravenatype-lattice 001} 
Assume that $X_1$ is $(\hbar, a)$-lattice, $\bb E (X_1) = 0$, $\bb E (X^2_1)= \sigma^2$,  and that there exists $\delta > 0$ 
such that $\bb E (|X_1|^{2+\delta})  < \infty.$ 
Then, one can find a constant $c >0$ such that, 
for any $n\geq 1$, $x \in \bb R$, $y\in \bb R_+$ such that $y-x \in \hbar \bb Z+na$, 
\begin{align} \label{Carav-type-bound-001}
\left|  \bb{P} \big( x+S_n = y, \tau_x >n \big)
   -  \hbar \frac{V_n(x)}{\sigma^2 n}  
   \ell\left(\frac{x}{\sigma\sqrt{n}} ,\frac{y}{\sigma\sqrt{n}}\right) \right| 
\leq  c \frac{ n^{-\frac{\delta}{6} }  + V_{n}(x) n^{-\frac{\delta}{8(3+\delta)} }\log n }{  n }. 
\end{align}
\end{theorem}

From the above result, it follows that, for any $x_0 \in \supp V$,
one can find a constant $c >0$ such that, 
for any $n\geq 1$, $x \geq x_0$, $y\in \bb R_+$ such that $y-x \in \hbar \bb Z+na $, 
\begin{align} \label{caravenna type cor-001}
\left|  \bb{P} \big( x+S_n = y \big| \tau_x >n \big)
   - \frac{\hbar}{\sigma \sqrt{n}}  
   \ell\left(\frac{x}{\sigma\sqrt{n}} ,\frac{y}{\sigma\sqrt{n}}\right) \right| 
\leq  c \frac{  n^{-\frac{\delta}{8(3+\delta)} }\log n }{  \sqrt{n} }. 
\end{align}
With $x=0$, the bound \eqref{caravenna type cor-001} improves the result of 
Caravenna \cite{Carav05}  under the additional moment assumption $\bb E (|X_1|^{2+\delta})  < \infty.$

For fixed values of $y \in \bb Z$, 
Theorem \ref{T-Caravenatype-lattice 001}  does not give a reasonable approximation, 
since the remainder term $o(n^{-1})$ is larger than the main term which is of order $n^{-3/2}$. 
To improve the remainder term we need to convolve Theorem \ref{T-Caravenatype-lattice 001} 
with the conditioned integral type theorem. 
This gives rise to another type of conditioned local limit theorem (also called {\it ballot theorem}), 
which ensures the rate $n^{-3/2}$ for the main term. 
This result is stated in Theorem \ref{CLLT-lattice-n3/2 main result} and will be proved in the next section. 

\subsection{Proof of the upper bound in the lattice case}
In this subsection we prove the upper bound in \eqref{Carav-type-bound-001}.

Let $\ee \in (0,\ee_0)$, where $\ee_0>0$ is a sufficiently small constant. 
Set $m = \floor{\ee n}$ and $k = n-m$, where $\floor{\cdot}$ denotes the integer part of a real number.  
We can assume $n>n_0(\ee) := \frac{4}{\ee^2}$, 
as otherwise, the desired bound holds trivially. 
By the Markov property,  for any $x \in \bb R$ and $y \in \bb R_+$ such that $y - x \in \hbar \bb Z+na$, we have
\begin{align} \label{lattice-JJJ-markov property}
\bb{P} \big( &x+S_n =y, \tau_x >n \big) \notag \\
&= \sum_{z\geq 0:\, z-x \in \hbar \bb Z + ka}  \bb{P} \big( z+S_{m}=y, \tau_z >m \big) \bb{P}\big( x+S_{k} = z,  \tau_x >k\big)
 \notag \\
&= \sum_{z\geq 0:\, z-x \in \hbar \bb Z + ka}  \bb{P} \big( z+S_{m}=y  \big) \bb{P} \big( x+S_{k} = z,  \tau_x >k \big) \notag\\
& \quad - \sum_{z\geq 0:\, z-x \in \hbar \bb Z + ka}  \bb{P} \big( z+S_{m}=y , \tau_z \leq m  \big) \bb{P} \big( x+S_{k} = z,  \tau_x >k \big) \notag\\
& = : I_{n,1}(x,y) - I_{n,2}(x,y),
\end{align}
where 
$z-x \in \hbar \bb Z + ka$ is equivalent to $y-z \in \hbar \bb Z + ma$, since $y - x \in \hbar \bb Z + na$. 

For the upper bound, it suffices to estimate $I_{n,1}(x,y)$, as $I_{n,2}(x,y)$ is non-negative.
The analysis for $I_{n,2}(x,y)$ is more involved and will be addressed in the next section, 
where we establish the lower bound for  $\bb{P} \big( x+S_n =y, \tau_x >n \big)$.

By the local limit theorem (Theorem \ref{Th-lattice-LLT-general}), 
there is a constant $c >0$ such that, for any $y, z \in \bb R$ with $z-y \in \hbar \bb Z + ma$,
\begin{align} \label{lattice-JJJJJ-1111-001}
\big| \bb{P} \big( y+S_{m}=z  \big) - \hbar \phi_{\sigma^2 m} \left( y-z \right) \big| 
 \leq    \frac{c }{ m^{(1 + \delta_1)/2}}  
 \leq    c\frac{ \ee^{-(1 + \delta_1)/2}}{ n^{(1 + \delta_1)/2}}, 
\end{align}
where  $\delta_1 = \min\{ 1, \delta \}$, $\delta>0$ is the exponent from the moment condition $\bb E (|X_1|^{2+\delta})  < \infty,$
and  $\phi_{v}(x) = \frac{1}{\sqrt{v}} \phi \left(\frac{x}{\sqrt{v}}\right)$ with $\phi$ being the standard Gaussian density. 
Since the function $L$ is decreasing on $\bb R_+$ and even on $\bb R$ (cf.\ \cite{GX-2024-CCLT}), 
it follows that $n \mapsto V_n(x)$ is increasing, i.e.,
for any $x \in \bb R$ and $1 \leq k \leq n$, 
\begin{align}\label{inequality-L-increase-n}
V_k(x) = V(x) L \left(\frac{x}{\sigma \sqrt{k}} \right) \leq V(x) L \left(\frac{x}{\sigma \sqrt{n}} \right) = V_n(x). 
\end{align}
By \eqref{upper-bound-probab-tau-x}  and \eqref{inequality-L-increase-n}, 
we have that for any $x \in \bb R$ and $k = n - \floor{\ee n}$, 
\begin{align}\label{lattice-upper-bound-probab-tau-x-002}
\bb{P} \left( \tau_{x} > k \right) \leq c \frac{ k^{- \frac{\delta}{4}} + V_k(x) }{\sqrt{k}} 
\leq c' \frac{ n^{- \frac{\delta}{4}} + V_n(x)}{\sqrt{n}}. 
\end{align}
Using this along with \eqref{lattice-JJJ-markov property} and \eqref{lattice-JJJJJ-1111-001}, 
we get that for any $x \in \bb R$ and $y \in \bb R_+$ with $y - x \in \hbar \bb Z + na$, 
\begin{align} \label{lattice-JJJ004}
| I_{n,1}(x,y)  -   J_n(x,y) | 
& \leq   c \frac{ \ee^{-(1 + \delta_1)/2}}{ n^{(1 + \delta_1)/2} }  \bb{P} \left( \tau_{x} > k \right) 
    \leq  c\ee^{- \frac{1 + \delta_1}{2}} 
\frac{ n^{- \frac{\delta}{4}} + V_n(x) }{ n^{ 1 + \frac{\delta_1}{2}} },  
\end{align}
where for brevity we denote
\begin{align} \label{lattice-JJJ006}
J_n(x,y) 
&:  = \hbar  \sum_{z\geq 0:\, z-x \in \hbar \bb Z + ka}   
    \phi_{\sigma^2  m} \left( y-z \right)
   \bb{P}\left( x+S_{k} =  y, \tau_{x}>k\right) \notag\\
 &  = \hbar \,   \bb E \Big( \phi_{\sigma^2  m} \left( x+S_{k}-z\right) ; \tau_{x}>k \Big) . 
\end{align}
Using Lemma \ref{New-ApplCondLT-002}, 
we get the following heat kernel approximation for $J_n(x, y)$ with a rate of convergence.

\begin{lemma}\label{lattice-New-ApplCondLT-002}
There exist constants $c, \ee_0 > 0$ such that, for any $n \geq 1$, 
$\ee \in (0, \ee_0)$, $x \in \bb R$ and $y \in \bb R_+$ with $y - x \in \hbar \bb Z + na$, 
\begin{align*}
 \left| J_{n}(x,y) - \hbar \frac{V_n(x)}{ \sigma^2 n }
   \ell\left(\frac{x}{\sigma\sqrt{n}} ,\frac{y}{\sigma\sqrt{n}}\right) \right| 
    \leq  c \ee^{-1/2} \frac{R_{n, \delta}(x)}{ n }  + c \ee^{1/2} \frac{V_n(x)}{ n }, 
\end{align*}
where 
 $\ell$ is defined by \eqref{def of func ell_H-001} and $R_{n,\delta}(x)$ is from Theorem \ref{introTheor-probtauUN-001}.
\end{lemma}

\begin{proof}
Using the change of variables $t = \sigma \sqrt{k} t'$ and $s = \sigma \sqrt{n} s'$, we get 
\begin{align}\label{lattice-def-Jn-x-abc}
J_n(x,y) 
&=  \frac{\hbar}{\sigma\sqrt{m}} \bb E \bigg( \phi \bigg( \frac{x+S_{k}-y}{\sigma\sqrt{m}}\Bigg) ; \tau_{x}>k \bigg) 
\notag\\
&=  \frac{ \hbar}{\sigma\sqrt{n}}  
 \bb E \bigg( \frac{1}{\sqrt{m/n}} \phi \bigg( \frac{x+S_{k}-y}{\sigma\sqrt{n}} \sqrt{n/m} \Bigg) ; \tau_{x}>k \bigg) \notag\\
&= \frac{ \hbar}{\sigma\sqrt{n}} \bb E \left( \varphi_{n, y} \left(\frac{x+S_{k}}{\sigma \sqrt{k}} \right); \tau_{x}>k \right), 
\end{align}
where, with $\eta_n = \frac{m}{n} = \frac{\floor{\ee n}}{n}$ and $\phi_{v}(x) = \frac{1}{\sqrt{v}}\phi \left(\frac{x}{\sqrt{v}}\right)$, 
\begin{align} \label{lattice-JJJ006b}
 \varphi_{n, y}(t) 
 =  \phi_{\eta_n} \left( t  \sqrt{k/n} -\frac{y}{\sigma\sqrt{n}}  \right),  \quad  t, y \in \bb R. 
\end{align}
Since the function $t \mapsto \varphi_{n,y}(t)$ is differentiable on $\bb R$ and vanishes as $t \to \pm \infty$, 
using Lemma \ref{New-ApplCondLT-002}, we get 
\begin{align} \label{lattice-ApplCondLT-002}
\left| J_{n}(x,y)  -   \hbar \frac{V_k(x)  }{\sigma \sqrt{k}}  
  \int_{\bb R _{+}}  \varphi_{n,y}(t) \ell \left( \frac{x}{\sigma \sqrt{k}},  t \right) dt  \right|  
 \leq  c \frac{R_{n,\delta}(x)}{\sqrt{n}}   \int_{\bb R}  | \varphi'_{n,y}(t) | dt,
\end{align}
where we used $R_{k, \delta}(x) \leq c' R_{n,\delta}(x)$ for some constant $c'>0$. 
From \eqref{lattice-JJJ006b} and a change of variable, there exists $c>0$ such that for any $y \in \bb R$, 
\begin{align} \label{lattice-JJJ-20001}
\int_{\bb R}  | \varphi'_{n,y}(t) | dt 
&  \leq \frac{1}{\sigma\sqrt{k}}  \int_{\bb R}  \left[   
   \left|  \frac{1}{\sqrt{m/k}}  \phi'\left( \frac{ t -\frac{y}{\sigma\sqrt{k}} }{\sqrt{m/k}}\right) \right|  \right]
\frac{dt}{\sqrt{m/k}} \notag\\
&  = \frac{1}{\sigma\sqrt{m}} \int_{\bb R}  \left[    \left| \phi'\left( t - \frac{y}{\sigma\sqrt{m}}\right) \right|  \right] dt   \notag\\
&  \leq  \frac{c}{\ee^{1/2} \sqrt{n}} . 
\end{align}
Now we transform the main term in \eqref{lattice-ApplCondLT-002}. 
By a change of variable $t = s \sqrt{n/k}$ and \eqref{new-density-scale-v}, we have 
\begin{align} \label{lattice-heat-kernel-esti-002}
& \frac{V_k(x)}{ \sigma \sqrt{k}}  \int_{\bb R _{+}}  \varphi_{n,y}(t) \ell \left(\frac{x}{\sigma\sqrt{k}}, t \right) dt  \notag\\
& =  \frac{V_k(x)}{ \sigma \sqrt{k}} 
  \int_{\bb R _{+}}  \varphi_{n,y} \left( s \sqrt{n/k} \right) \frac{1}{\sqrt{k/n}} 
   \ell \left(\frac{\frac{x}{\sigma\sqrt{n}}}{\sqrt{k/n}}, \frac{s}{\sqrt{k/n}} \right) ds  \notag \\
& = \frac{V_n(x)}{ \sigma \sqrt{n}}  \int_{\bb R _{+}}  \varphi_{n,y} \left( s \sqrt{n/k} \right) 
 \ell_{1-\eta_n}\left(\frac{x}{\sigma\sqrt{n}}, s \right) ds, 
\end{align}
where $\eta_n = \frac{m}{n}$ and $1 - \eta_n = \frac{k}{n}$,
and we used the identity 
$$
\frac{V_k(x)}{ \sigma \sqrt{k}} \frac{1}{H( \frac{x}{\sigma \sqrt{k}})} = \frac{V_n(x)}{ \sigma \sqrt{n}} \frac{1}{H( \frac{x}{\sigma \sqrt{n}})}.
$$ 
For $v \in (0, 1]$ and $x, y \in \bb R$, let 
\begin{align} \label{def-ell-H-plus-minus}
\ell_{v}^+(x,y)= \ell_{v}(x,y) \mathds 1_{\{ y > 0 \}}. 
\end{align}
The convolution $\phi_{v} * \ell^+_{1-v}$ is understood as an ordinary convolution of $\phi_{v}$ 
with respect to the second variable of the function $\ell^+_{1-v}$: for any $v \in (0, 1]$ and $x, y \in \bb R$, 
\begin{align} \label{def-convo-ell-H-plus}
\phi_{v} * \ell^+_{1-v}(x, y) := \int_{-\infty}^{\infty} \phi_{v} (y-z) \ell^+_{1-v}(x, z) dz
= \int_{\bb R_+} \phi_{v} (y-z) \ell_{1-v}(x, z) dz. 
\end{align}
By \eqref{lattice-JJJ006b} and \eqref{def-convo-ell-H-plus},  
we get
\begin{align*}
 \int_{\bb R _{+}}  \varphi_{n,y} \left( s \sqrt{n/k} \right) \ell_{1-\eta_n}\left(\frac{x}{\sigma\sqrt{k}}, s \right) ds  
&= \frac{1}{\sigma\sqrt{n}} \int_{\bb R _{+}}   \phi_{\eta_n} \left( \frac{y}{\sigma\sqrt{n}} - s   \right) 
    \ell_{1-\eta_n}\left(\frac{x}{\sigma\sqrt{n}} , s \right) ds  \notag\\
& = \frac{1}{\sigma\sqrt{n}} \phi_{\eta_n}*\ell^+_{1-\eta_n}\left(\frac{x}{\sigma\sqrt{n}} , \frac{y}{\sigma\sqrt{n}} \right).
\end{align*}
Combining this with \eqref{lattice-ApplCondLT-002}, \eqref{lattice-JJJ-20001} and \eqref{lattice-heat-kernel-esti-002}, 
we get 
\begin{align} \label{lattice-heat-kernel-esti-003}
\left| J_n(x,y) 
 -   \hbar \frac{V_n(x)}{ \sigma^2 n }
  \phi_{\eta_n}*\ell^+_{1-\eta_n}\left(\frac{x}{\sigma\sqrt{n}} ,\frac{y}{\sigma\sqrt{n}}\right) \right| 
 \leq  c \ee^{-1/2} \frac{ R_{n, \delta}(x)}{ n }. 
\end{align}
By Lemma \ref{convol-phi-psi-001} and \eqref{def-convo-ell-H-plus}, we have that, for any $x \in \bb R$ and $y \in \bb R_+$, 
\begin{align}\label{lattice-conv001}
 \phi_{\eta_n}*\ell^+_{1-\eta_n}\left(\frac{x}{\sigma\sqrt{n}} ,\frac{y}{\sigma\sqrt{n}}\right)   
 = \ell \left(\frac{x}{\sigma\sqrt{n}} ,\frac{y}{\sigma\sqrt{n}}\right)  
 -  \int_{\frac{y}{\sigma\sqrt{n}}}^{\infty}  
\phi_{\eta_n} (t) \ell_{1-\eta_n} \left( \frac{x}{\sigma \sqrt{n}}, \frac{y}{\sigma\sqrt{n}} - t \right) dt. 
\end{align}
Since $\frac{y}{\sigma \sqrt{n}} \in [0, t]$, using \eqref{ell-intergr-bound-002} 
of Lemma \ref{Lipschiz  property for ell_H in y}
and the fact that $\ell_{1-\eta_n} ( \frac{x}{\sigma \sqrt{n}}, 0 ) = 0$,
\begin{align*} 
\left| \ell_{1-\eta_n} \left( \frac{x}{\sigma \sqrt{n}}, \frac{y}{\sigma\sqrt{n}} - t \right)  \right|
\leq c \left( t-\frac{y}{\sigma \sqrt{n}}\right) \leq c| t |. 
\end{align*}
Therefore, 
\begin{align} \label{lattice-bound-ell-psi-a}
\left| \int_{\frac{y}{\sigma\sqrt{n}}}^{\infty}  
\phi_{\eta_n} (t) \ell_{1-\eta_n} \left( \frac{x}{\sigma \sqrt{n}}, \frac{y}{\sigma\sqrt{n}} - t \right) dt  \right|
\leq  c \int_{\frac{y}{\sigma\sqrt{n}}}^{\infty}  \phi_{\eta_n} (t)  | t | dt \leq c \sqrt{\eta_n} \leq c \ee^{1/2}. 
\end{align}
Together with \eqref{lattice-heat-kernel-esti-003} and 
\eqref{lattice-conv001}, this concludes the proof of the lemma. 
\end{proof}

Combining Lemma \ref{lattice-New-ApplCondLT-002} with \eqref{lattice-JJJ-markov property} and \eqref{lattice-JJJ004},  
we obtain that 
there exist constants $c, \ee_0 > 0$ such that, for any $n \geq 1$, 
$\ee \in (0, \ee_0)$, $x \in \bb R$ and $y \in \bb R_+$ with $y - x \in \hbar \bb Z + na$, 
\begin{align}\label{lattice-last-inequality-001}
&  \bb{P} \big( x+S_n = y, \tau_x >n \big)
  -  \hbar  \frac{V_n(x)}{\sigma^2 n} 
   \ell\left(\frac{x}{\sigma\sqrt{n}} ,\frac{y}{\sigma\sqrt{n}}\right)   \notag\\
 & \leq  c  \ee^{- \frac{1 + \delta_1}{2} } \frac{ n^{- \frac{\delta}{4}} + V_n(x) }{  n^{ 1 + \frac{\delta_1}{2}} }     
+ c \ee^{-1/2} \frac{ n^{-\frac{\delta}{4} }  + V_{n}(x) n^{-\frac{\delta}{4(3+\delta)} }\log n }{  n } 
+  c  \ee^{1/2}  \frac{V_n(x)}{ n }. 
\end{align}
Taking $\ee = n^{-\frac{\delta}{4(3+\delta)} }$, 
the last term in \eqref{lattice-last-inequality-001} is bounded by the second one. 
This concludes the proof of the upper bound in Theorem  \ref{T-Caravenatype-lattice 001}. 

\subsection{Proof of the lower bound in the lattice case}
In this subsection we establish the lower bound in Theorem  \ref{T-Caravenatype-lattice 001}. 
In the proof we shall make use of the following Fuk-Nagaev inequality for random walks, which can be found in \cite{FN71, Nag79}. 
\begin{lemma}\label{Lem_FukNagaev}
Let  $\bb E X_1 = 0$ and  $\bb E (X_1^{2}) = \sigma^2  < \infty.$
Then, for any $u, v >0$, 
\begin{align*}
\bb P \left(\max_{1 \leq k \leq n} |S_k| > u \right) 
\leq  2 \exp \left[ \frac{u}{v}  \left( 1 + \log \frac{n}{uv} \right)  \right]   + n \bb P \left(|X_1| > v \right). 
\end{align*}
\end{lemma}

By Lemma \ref{lattice-New-ApplCondLT-002}, \eqref{lattice-JJJ-markov property} and \eqref{lattice-JJJ004}, 
it remains to provide an upper bound for $I_{n,2}(x, y)$ defined in \eqref{lattice-JJJ-markov property}. 
Since 
\begin{align}\label{lattice-bound-Li-norm-M}
\sum_{z\geq 0:\, z-x \in \hbar \bb Z + ka}  \bb{P} \big( z+S_{m}=y , \tau_z \leq m  \big) \leq 1, 
\end{align}
applying the upper bound in Theorem \ref{T-Caravenatype-lattice 001} established in the previous section,   
we obtain that, for any $n\geq 1$, $x \in \bb R$, $y \in \bb R_+$ such that $y -x \in \hbar \bb Z+na$, 
\begin{align}\label{lattice-eqt-A 001_Lower}
I_{n,2}(x, y) 
  \leq   \frac{V_k(x)}{\sigma^2 k}  K_n(x,y) 
  + c  \frac{n^{-\frac{\delta }{6} }  + V_{n}(x) n^{-\frac{\delta}{8(3+\delta)} } \log n }{  n }.  
\end{align}
where we use the fact that $z-x \in \hbar \bb Z + ka$ is equivalent to $y-z \in \hbar \bb Z + ma$, and for brevity we denote
\begin{align*}
K_n(x,y) =  \sum_{z\geq 0:\, z- y \in \hbar \bb Z + ma}
   \bb{P} \big( z+S_{m}=y , \tau_z \leq m  \big)   \ell \left(\frac{x}{\sigma\sqrt{k}} ,\frac{z}{\sigma\sqrt{k}}\right). 
\end{align*}
We write 
\begin{align}\label{lattice-decom-integral-gm-ell}
K_n(x,y) 
&= \sum_{0 \leq  z \leq \sqrt{\ee n \log n}:  \, z-y \in \hbar \bb Z + ma}  
 \bb{P} \big( z+S_{m}=y , \tau_z \leq m  \big)   \ell \left(\frac{x}{\sigma\sqrt{k}} ,\frac{z}{\sigma\sqrt{k}}\right)  \notag\\ 
& \quad + 
\sum_{z >  \sqrt{\ee n \log n}: \,  z-y \in \hbar \bb Z + ma}  
 \bb{P} \big( z+S_{m}=y , \tau_z \leq m  \big)   \ell \left(\frac{x}{\sigma\sqrt{k}} ,\frac{z}{\sigma\sqrt{k}}\right)  \notag \\
&=: K_{n,1}(x,y) + K_{n,2}(x,y).
\end{align}
By the local limit theorem (see Theorem \ref{Th-lattice-LLT-general}), 
there exists $c>0$ such that for any $y,z \in \bb R$ such that $z-y\in \hbar \bb Z + ma$, 
\begin{align*}
\bb{P} \big( z+S_{m}=y , \tau_z \leq m  \big)
\leq   \frac{c}{\sqrt{m}}. 
\end{align*}
Using this and the fact that $m = \floor{\ee n}$ and $k = n-m$ (so that $\sqrt{n/k} \leq 2$ for some small enough $\ee>0$), we get
\begin{align}\label{latice-first-term-gm-y-inte}
 \left| K_{n,1}(x,y)   \right|  
& \leq  \frac{c}{\sqrt{m}} 
\sum_{0 \leq  z \leq \sqrt{\ee n \log n}:  \, z-y \in \hbar \bb Z + ma} 
  \ell \left(\frac{x}{\sigma\sqrt{k}}, \frac{z}{\sigma\sqrt{k}}\right)  \notag\\
& \leq  c \ee^{- 1/2}    \int_{0}^{ 2 \sqrt{\ee \log n} }  \ell \left(\frac{x}{\sigma\sqrt{k}} , z \right)   dz.
\end{align}
Note that, by Lemma \ref{Lipschiz  property for ell_H in y} and the fact that $\ell(x, 0) = 0$, 
there exists a constant $c >0$ such that $|\ell(x, z)| = |\ell(x, z) - \ell(x, 0)| \leq c |z|$ 
for any $x, z \in \bb R$.
Hence, from \eqref{latice-first-term-gm-y-inte}
we obtain 
\begin{align}\label{latice-integral-gm-ell-small}
 \left| K_{n,1}(x,y)   \right|  
 \leq  c \ee^{- 1/2} \int_{0}^{ 2 \sqrt{\ee \log n} }  z dz 
 \leq  c  \ee^{1/2} \log n. 
\end{align}

Next we handle the second term on the right-hand side of \eqref{lattice-decom-integral-gm-ell}. 
Since the function $(x, z) \mapsto |\ell(x, z)|$ is bounded on $\bb R \times \bb R$, we get
\begin{align}\label{latice-integral-My-Fk-xy}
 \left| K_{n,2}(x,y)   \right|  
& \leq  c  \sum_{z > \sqrt{\ee n \log n}: \,  z-y \in \hbar \bb Z + ma}   \bb{P} \big( z+S_{m}=y , \tau_z \leq m  \big)  \notag\\
& = c \,  \bb P \left( \tau_{y - S_m }  \leq m,  \,  y - S_m \geq  \sqrt{\ee n \log n}   \right). 
\end{align}
Observe that the event $\{ \tau_{y - S_m }  \leq m \}$
means that there exists an integer $1 \leq j \leq m$ such that 
\begin{align*}
y - S_m + S_j  < 0, 
\end{align*}
so that on the event $\{ \tau_{y - S_m}  \leq m,  y - S_m \geq  \sqrt{\ee n \log n}  \}$, 
there exists an integer $1 \leq j \leq m$ such that $\sqrt{\ee n \log n} + S_j <   0$. 
This implies that 
\begin{align}\label{latice-large-deviation-aa}
\bb P \left( \tau_{y - S_m }  \leq m,  \,  y - S_m \geq  \sqrt{\ee n \log n}   \right)
& \leq  \bb P \left(  \sqrt{\ee  n \log n} + \min_{1 \leq j \leq m} S_j   < 0   \right)  \notag\\
& \leq  \bb P \left(  \max_{1 \leq j \leq m}  |S_j| \geq  \sqrt{\ee  n \log n}  \right). 
\end{align}
Now we apply the Fuk-Nagaev inequality (Lemma \ref{Lem_FukNagaev}) 
with $n$ replaced by $m= \floor{\ee n}$ and 
$u = \sqrt{\ee n \log n}$, $v = \sqrt{\ee n}$,
and the moment condition $\bb E (|X_1|^{2 + \delta}) < \infty$ to get
\begin{align}\label{latice-Proba_002}
\bb P \left(  \max_{1 \leq j \leq m}  |S_j| \geq   \sqrt{\ee  n\log n}  \right)
& \leq  c  e^{ \sqrt{\log n}  (1+ \log \frac{1}{\sqrt{\log n}}) }   
+ \ee n \bb P(|X_1| > \sqrt{\ee n})   \notag \\
& \leq  c  n^{-\delta}  +  c \frac{ \ee n }{ (\sqrt{\ee n})^{2 + \delta} } 
 \leq    c (\ee n)^{-\delta/2}. 
\end{align}
This, together with \eqref{lattice-decom-integral-gm-ell}, \eqref{latice-integral-gm-ell-small} 
and \eqref{latice-integral-My-Fk-xy}, implies that 
\begin{align}\label{latice-Bound_J1}
K_n(x,y) 
  \leq   c  \left(   \ee^{1/2}\log n  +  (\ee n)^{-\delta/2} \right).
\end{align}
Combining this with \eqref{lattice-eqt-A 001_Lower} and taking $\ee = n^{-\frac{\delta}{4(3+\delta)} }$,
we conclude the proof of the lower bound in Theorem  \ref{T-Caravenatype-lattice 001}.

\subsection{Proof of the ballot type local limit theorem in the lattice case} 
\label{SecProofUnifballotLLTheor-lattice-001}

The goal of this subsection is to establish Theorem \ref{CLLT-lattice-n3/2 main result}. 
We will make use of the following duality identities. 
Note that in the lattice case under consideration $S_n \in \hbar \bb Z+na$ and $\check S_n \in \hbar \bb Z - na$. 
Moreover, for any $x,y\in \bb R$ we have that $y-x\in \hbar \bb Z+na$ is equivalent to $x-y\in \hbar \bb Z - na$.  
\begin{lemma} \label{Lemma-lattice duality-001}
For any $n \geq 1$ 
and $x, y \in \bb R$ with $y-x \in \hbar \bb Z+na$, we have
\begin{align} \label{lattice-duality-001}
\bb P \big( x+S_n = y, \tau_x >n-1 \big) = \bb P \big( y+\check S_n = x, \check \tau_y > n-1 \big). 
\end{align}
Moreover, for any function $f: \bb R_+ \to \bb R_+$ and any $y\in \bb R$,
\begin{align} \label{lattice-duality-002}
\sum_{x\geq 0: \, y-x\in \hbar \bb Z+na} f(x) \bb P ( x+S_n = y, \tau_x >n-1)  
= \bb E \left( f(y+\check S_n);  \check \tau_y  > n \right).
\end{align}
 \end{lemma}
\begin{proof}
The first identity \eqref{lattice-duality-001} is obvious. 
For the second one, 
 for any function $f: \bb Z \to \bb R_+$ and $y\in \bb R$, we have
\begin{align*} 
&\sum_{x\geq 0: \,  y-x\in \hbar \bb Z+na} f(x) \bb P ( x+S_n = y, \tau_x >n-1) \notag \\
&\qquad\qquad 
= \sum_{x\geq 0: \, x-y\in \hbar \bb Z-na} f(x) \bb P ( y+\check S_n = x, \check \tau_y >n-1) \notag\\
&\qquad\qquad = \bb E \left( f(y+\check S_n); \check \tau_y >n \right),
\end{align*}
which proves \eqref{lattice-duality-002}.
\end{proof}

The following lemma 
 is an easy consequence of the duality formula \eqref{lattice-duality-002} 
 and bounds of the exit time for the dual random walk. 

\begin{lemma}\label{latt-lemma Qng integr}
Assume that $\bb E (X_1) = 0$, $\bb E (X^2_1)= \sigma^2$,  and that there exists $\delta > 0$ 
such that $\bb E (|X_1|^{2+\delta})  < \infty.$ 
Then, there exists a constant $c>0$ such that for any $n\geq 1$ and $y\in \bb R$, 
\begin{align*}
\sum_{x\geq 0: \, y-x\in \hbar \bb Z+na}   \bb P (x+S_n=y, \tau_x >n - 1) 
\leq  \frac{c}{\sqrt{n}}   \left( n^{ -\delta/4} +  \check V_n(y) \right). 
\end{align*}
\end{lemma}
\begin{proof}
By  the duality identity \eqref{lattice-duality-002} with $f=1$, we have, for any $y\in \bb R$, 
\begin{align*}
\sum_{x\geq 0: \, y-x \in \hbar \bb Z + na}  \bb P ( x+S_n = y, \tau_x >n-1)
 =    \mathbb{P} \big( \check \tau_{y} > n \big).
\end{align*}
The conclusion follows using the bound \eqref{upper-bound-probab-tau-x} applied to the 
reversed walk $\check S_n$. 
\end{proof}

Let us now proceed with the proof of Theorem \ref{CLLT-lattice-n3/2 main result}. 
Set $m=\floor{n/2}$ and $k = n-m.$  
By the Markov property, we have for any $x, y \in \bb R$ such that $y - x \in \hbar \bb Z+na$, 
\begin{align} \label{latt-JJJ-markov property_aa}
&\bb{P}  \Big(  x+S_n = y, \tau_x > n - 1 \Big)  \notag \\
& = \sum_{z\geq 0: \, z-x\in \hbar \bb Z + ka}  \bb P  \Big( z+S_{m}= y,  \tau_{z} > m - 1 \Big) \bb{P}\left( x+S_{k} = z,  \tau_x > k \right),
\end{align}
Applying  Theorem \ref{T-Caravenatype-lattice 001}, we obtain the following bound: 
for any $n \geq 1$ and $x, y \in \bb R$ with $y - x \in \hbar \bb Z+na$, 
\begin{align}\label{latt-UpperBoundhhn32}
& \bigg| \bb{P}  \Big( x+S_n = y, \tau_x > n - 1 \Big)    \notag\\
&  \quad  - \hbar \frac{V_k(x)}{\sigma^2 k} 
 \sum_{z\geq 0: \, y-z\in  \hbar \bb Z +ma}  \bb P  \Big( z+S_{m}= y,  \tau_{z} > m - 1 \Big)
  \ell \left(\frac{x}{\sigma\sqrt{k}}, \frac{z}{\sigma\sqrt{k}}\right) \bigg| \notag\\
 & \leq  c  \frac{n^{-\frac{\delta}{6} }  + V_{n}(x) n^{-\frac{\delta}{8(3+\delta)} } \log n }{  n }  
 \sum_{z\geq 0: \,  y-z\in  \hbar \bb Z +ma}  \bb P  \Big( z+S_{m}= y,  \tau_{z} > m - 1 \Big), 
\end{align}
where we used the fact that $k = n - \floor{n/2}$, $V_k(x) \leq V_n(x)$ (cf.\ \eqref{inequality-L-increase-n})
and $z-x\in \hbar \bb Z + ka$ if and only if $y-z\in  \hbar \bb Z +ma$. 
For the last term on the right-hand side of \eqref{latt-UpperBoundhhn32}, by Lemma \ref{latt-lemma Qng integr}
and again the fact that $m = \floor{n/2}$ and $\check V_m(x) \leq \check V_n(x)$,
\begin{align} \label{latt-BoundJ2Order32}
\sum_{z\geq 0: \,  y-z\in  \hbar \bb Z +ma}  \bb P  \Big( z+S_{m}= y,  \tau_{z} > m - 1 \Big) 
 \leq c  \left( n^{ -\delta/4} +  \check V_n(y) \right).
\end{align}
For the main term in \eqref{latt-UpperBoundhhn32}, 
using the duality identity \eqref{lattice-duality-002}, we get
\begin{align}\label{latt-BoundJ1rr}
&  \sum_{z\geq 0: \,  y-z\in  \hbar \bb Z +ma}  \bb P  \Big( z+S_{m}= y,  \tau_{z} > m - 1 \Big)
  \ell \left(\frac{x}{\sigma\sqrt{k}}, \frac{z}{\sigma\sqrt{k}}\right)  \notag\\
  & \qquad =   \bb E \left[  \ell\left(\frac{x}{\sigma\sqrt{k}}, \frac{y + \check S_m}{\sigma \sqrt{k}} \right); 
  \check \tau_{y} > m   \right],
\end{align}
where $\check \tau_{y}$ is defined by \eqref{definition of tau x star-001}. 
Using the conditioned central limit theorem (cf.\ Lemma \ref{New-ApplCondLT-002}), 
we derive the following effective
Gaussian heat kernel approximation for the expectation in \eqref{latt-BoundJ1rr}. 
The crucial step in proving this result involves the convolution operation \eqref{psi-convolution-lem}.  

\begin{lemma} \label{Lem-ballot-001}
There exists a constant $c>0$ such that, for any $n \geq 2$ and $x, y \in \bb R$ 
with $m= \floor{n/2}$ and $k = n - m$, 
\begin{align*}
 \Bigg|   
  \frac{V_k(x)}{\sigma^2 k } 
 \bb E \left[ \ell \left( \frac{x}{\sigma \sqrt{k}},  \frac{y + \check S_m}{\sigma \sqrt{k}} \right);  \check \tau_y > m  \right] 
 - \frac{ V_n(x) \check V_n(y) }{ \sigma^3 n^{3/2}}   p \left( \frac{x}{\sigma\sqrt{n}} , \frac{y}{\sigma\sqrt{n} } \right)
\Bigg|    
  \leq  c \frac{V_n(x) \check R_{n,\delta}(y) }{ n^{3/2} }, 
\end{align*}
where the function $p$ is defined by \eqref{def of func ell_H-002} 
and $\check R_{n,\delta}(y) = n^{-\frac{\delta}{4} } + \check V_n(y)  n^{-\frac{\delta}{4(3+\delta)} }\log n$.
\end{lemma}

\begin{proof} 
Since $\lim_{y \to \pm \infty} \ell ( \frac{x}{\sigma \sqrt{k}},  y) = 0$ (cf.\ \eqref{def of func ell_H-001}), 
applying Lemma \ref{New-ApplCondLT-002} to the reversed random walk $\check S_n$ 
and the target function $\varphi(\cdot) = \ell ( \frac{x}{\sigma \sqrt{k}},   \sqrt{\frac{m}{k}} \cdot)$, 
we derive that 
\begin{align}\label{Heat-kernel-approx-00a}
& \left| \bb E \left[ \ell \left( \frac{x}{\sigma \sqrt{k}},  \frac{y + \check S_m}{\sigma \sqrt{k}} \right);  \check \tau_y > m  \right]  
 -  \frac{\check V_m(y) }{\sigma \sqrt{m}}  
  \int_{\bb R_+}  \ell \left( \frac{x}{\sigma \sqrt{k}},   \sqrt{\frac{m}{k}}  \, t \right)  \ell \left( \frac{y}{\sigma \sqrt{m}},  t  \right)  dt  \right|
  \notag\\
&  \leq  c \frac{\check R_{m,\delta}(y)}{ \sqrt{m} }   
 \int_{\bb R} \left|  \frac{\partial}{\partial t}  \ell \left( \frac{x}{\sigma \sqrt{k}}, \sqrt{\frac{m}{k}} \, t \right)  \right| dt  \notag\\
& \leq  c \frac{\check R_{n,\delta}(y)}{ \sqrt{n} }, 
\end{align}
where in the last line we used the fact that $\check R_{m,\delta}(y) \leq c \check R_{n,\delta}(y)$ (cf.\ \eqref{inequality-L-increase-n}) 
and the inequality \eqref{derivative-ell-H-ineq} of Lemma \ref{Lipschiz  property for ell_H in y}. 
For the main term in \eqref{Heat-kernel-approx-00a}, we have 
\begin{align*}
A_{k,m}(x, y) : & =  \frac{V_k(x) \check V_m(y) }{ \sigma^3 k \sqrt{m}}   
 \int_{\bb R_+}  \ell \left( \frac{x}{\sigma \sqrt{k}},   \sqrt{\frac{m}{k}}  \, t \right)  \ell \left( \frac{y}{\sigma \sqrt{m}},  t  \right)  dt  \notag\\
 & = \frac{V_k(x) \check V_m(y) }{\sigma^4  k m}  \int_{\bb R_+}  \ell \left( \frac{x}{\sigma \sqrt{k}},\frac{u}{\sigma\sqrt{k}} \right) 
\ell \left( \frac{y}{\sigma\sqrt{m}},  \frac{u}{\sigma\sqrt{m} }   \right)  du, 
\end{align*}
where we used a change of variable $t = \frac{u}{\sigma \sqrt{m}}$. 
By \eqref{heat-kernel-001} and Lemma \ref{lem: product heat kernels}, we get 
\begin{align*}
 \int_{\bb R_+}  \psi \left( \frac{x}{\sigma \sqrt{k}},\frac{u}{\sigma\sqrt{k}} \right) 
\psi \left( \frac{y}{\sigma\sqrt{m}},  \frac{u}{\sigma\sqrt{m} }   \right)   du  
& = \sigma^2 \sqrt{km}  \int_{\bb R_+} \psi_{\sigma^2k} (x, u)  \psi_{\sigma^2m}  (y, u)  du  \notag\\
& = \sigma^2 \sqrt{km}  \,  \psi_{\sigma^2 n} (x, y),  
\end{align*}
so that, by \eqref{func L as integral of heat kern}, \eqref{def of func ell_H-001}, \eqref{def of Q-001} and \eqref{def of func ell_H-002}, 
\begin{align}\label{psi-convolution-lem}
A_{k,m}(x, y)  = \frac{V_k(x) \check V_m(y) }{ \sigma^4  km } 
\frac{ \sigma^2 \sqrt{km}  \,  \psi_{\sigma^2 n} (x, y) }{ H \left( \frac{x}{\sigma \sqrt{k}} \right) H \left( \frac{y}{\sigma\sqrt{m}}  \right) }   
 =  \frac{V_n(x) \check V_n(y)}{ \sigma^3 n^{3/2} }   p \left( \frac{x}{\sigma\sqrt{n}} , \frac{y}{\sigma\sqrt{n} } \right). 
\end{align}
Combining \eqref{Heat-kernel-approx-00a} and \eqref{psi-convolution-lem}, 
and using the fact that $V_k(x)  \leq  V_n(x)$ (cf.\ \eqref{inequality-L-increase-n}),
we obtain the assertion of the lemma. 
\end{proof}

Combining \eqref{latt-UpperBoundhhn32}, \eqref{latt-BoundJ2Order32}, \eqref{latt-BoundJ1rr} 
and Lemma \ref{Lem-ballot-001}, we get 
\begin{align*}
& \bigg|  \bb{P}  \Big( x+S_n = y, \tau_x > n - 1 \Big)
 -   \hbar \frac{V_n(x) \check V_n(y) }{ \sigma^3 n^{3/2} }   p \left( \frac{x}{\sigma\sqrt{n}} , \frac{y}{\sigma\sqrt{n} } \right) \bigg|  \notag\\
&   \leq  c  \frac{n^{-\frac{\delta}{6} }  + V_{n}(x) n^{-\frac{\delta}{8(3+\delta)} }  \log n }{  n^{3/2} }
 \left( n^{-\frac{\delta}{4}} + \check V_n(y)  \right)  
 + c \frac{V_n(x) }{ n^{3/2} }  \check R_{n, \delta}(y)  \notag\\
& \leq  c \frac{\left(n^{-\frac{\delta}{8} }  + V_{n}(x) \right)
  \left( n^{-\frac{\delta}{8}} + \check V_n(y)  \right)}{n^{3/2  }}  n^{ -\frac{\delta}{8(3+\delta)} }\log n.
 \end{align*}
This finishes the proof of  Theorem \ref{CLLT-lattice-n3/2 main result}.

\subsection{Proof of the local theorem for the exit time} \label{sec: proof LLT for tau}
The goal of this section is to prove Theorem \ref{theorem local for tau lattice}, by using Theorem \ref{CLLT-lattice-n3/2 main result}.

Denote, for any $n \geq 1$ and $x\in \bb R$,
\begin{align} \label{definition of varkappa-aa001}
K_n(x)  = \sum_{y \geq 0: \,  y-x \in \hbar \bb Z+na} \check V_{n}(y) 
  \bb P(X_1 < -y) \,  p\left( \frac{x}{\sigma\sqrt{n}}, \frac{y}{\sigma\sqrt{n}} \right).
\end{align}
Since $p$ is bounded on $\bb R \times \bb R$ and $\check V_{n}(y) \leq \check V(y)$, 
it follows that $K_n$ is also bounded on $\bb R$:
\begin{align*} 
K_n(x) \leq c \varkappa_n( x ) <\infty.
\end{align*}
where $\varkappa_n$ is defined by \eqref{def of varkappa_n(x) abcd001}. 
The following lemma establishes the relation between the functions $K_n$ and $\varkappa_n$,
where we recall that the function $\phi_L$ is defined by \eqref{approx by normal density-001b}. 

\begin{lemma} \label{lem-appox for varkappa_n-001}
There exists a constant $c>0$ such that, for any $n\geq 1$ and $x \in \bb R$,
\begin{align*} 
\left| K_n(x) - \phi_L\left(\frac{x}{\sigma\sqrt{n}}\right)  L(0)  \varkappa_n(x) \right| 
\leq c n^{- \frac{\delta}{2(1 + \delta)}}.
\end{align*}
 \end{lemma}
\begin{proof}
Let $(\beta_n)_{n \geq 1}$ be a sequence of positive numbers such that $\beta_n\to 0$ as $n\to\infty$. 
Note that $\check V_n(y) = \check  V(y)  L \left(\frac{y}{\sigma\sqrt{n}}\right)$.  
Since $p$ is positive and bounded on $\bb R \times \bb R$,
and $L$ is bounded on $\bb R$,  the quantity
$K_n(x)$ can be bounded from above and below as follows: 
\begin{align} \label{two sided bound K_n-001}
I_1(x) \leq  K_n(x) \leq  I_1(x) + c I_2(x), 
\end{align}
where, for $n\geq 1$ and $x\in \bb R$, 
\begin{align*} 
I_1(x) &= \sum_{0 \leq y \leq \beta_n \sqrt{n}: \,  y - x \in \hbar \bb Z+na} \check V(y)  \bb P(X_1 < -y) \, 
p\left( \frac{x}{\sigma\sqrt{n}}, \frac{y}{\sigma\sqrt{n}} \right) L\left(\frac{y}{\sigma\sqrt{n}}\right),  \notag\\
I_2(x) &= \sum_{y > \beta_n \sqrt{n}: \,  y - x \in \hbar \bb Z+na} \check V(y) \bb P(X_1 < -y). 
\end{align*}
We first give an upper bound of $I_1(x)$. 
By the  Lipschitz property of the function $p$ (see \eqref{lem-holder prop for ell-002}), 
there exists $c>0$ such that for any $n \geq 1$, $y \in [0, \beta_n \sqrt{n}]$ and $x \in \bb R$, 
\begin{align*} 
\left|p\left( \frac{x}{\sigma\sqrt{n}}, \frac{y}{\sigma\sqrt{n}} \right) - p\left( \frac{x}{\sigma\sqrt{n}}, 0 \right)\right|
\leq c  \frac{y}{\sigma\sqrt{n}} \leq c \beta_n.
\end{align*}
By Lemma 2.3 in \cite{GX-2024-AIHP}, 
we have  
\begin{align*} 
\left| L\left(\frac{y}{\sigma\sqrt{n}}\right) - L(0) \right| \leq c \frac{y}{\sigma\sqrt{n}} \leq c \beta_n. 
\end{align*}
Therefore, using the symmetry of the function $p$
 and the fact that $\phi_L$ is bounded on $\bb R$ (cf.\  \eqref{approx by normal density-001b}), we obtain
\begin{align} \label{bound I_1aadd001}
I_1(x)
& \leq  \left(\phi_L\left(\frac{x}{\sigma\sqrt{n}}\right) + c \beta_n\right) \left(L(0)+c\beta_n\right)
   \sum_{0 \leq y \leq \beta_n \sqrt{n}: \,  y-x \in \hbar \bb Z+na} \check V(y) 
 \bb P(X_1 < -y)  \notag \\
&\leq  \left(\phi_L\left(\frac{x}{\sigma\sqrt{n}}\right) + c \beta_n\right) \left(L(0)+c\beta_n\right)   \varkappa_n(x) \notag \\
&\leq   \phi_L\left(\frac{x}{\sigma\sqrt{n}}\right) L(0) \varkappa_n(x) +c\beta_n.
\end{align}
For $I_2(x)$, since $\check V(y) \leq c(1+y)$ and $\bb P(X_1 < -y) \leq \frac{1}{ y^{2+\delta} } \bb E(|X_1|^{2+\delta})$,
 we derive that 
\begin{align*} 
I_2(x) 
\leq  \sum_{y > \beta_n \sqrt{n}: \,  y-x \in \hbar \bb Z+na} \frac{1+y}{y^{2+\delta}} \bb E(|X_1|^{2+\delta}) \leq c(\beta_n\sqrt{n})^{-\delta}. 
\end{align*}
Combining this with \eqref{two sided bound K_n-001} and \eqref{bound I_1aadd001}, we get the following upper bound:
\begin{align*} 
K_n(x) 
\leq L(0) \phi_L\left(\frac{x}{\sigma\sqrt{n}}\right) \varkappa_n(x)  + c \beta_n + c(\beta_n\sqrt{n})^{-\delta}.  
\end{align*}
A lower bound is proved in the same way. Indeed, proceeding as in \eqref{bound I_1aadd001}, we have
\begin{align*} 
I_1(x)
& \geq  \left(\phi_L\left(\frac{x}{\sigma\sqrt{n}}\right) - c \beta_n\right)  \left( L(0) - c \beta_n \right) 
 \sum_{0 \leq y \leq \beta_n \sqrt{n}: \,  y-x \in \hbar \bb Z+na} \check V(y) 
\bb P(X_1 < -y)  \notag \\ 
&\geq  \left(\phi_L\left(\frac{x}{\sigma\sqrt{n}}\right) - c \beta_n\right)  \left( L(0) - c \beta_n \right)  \varkappa_n(x) 
   -c I_2(x)\\
&\geq  \phi_L\left(\frac{x}{\sigma\sqrt{n}}\right) L(0) \varkappa_n(x) - c \beta_n - c(\beta_n\sqrt{n})^{-\delta}.
\end{align*}
Choosing $\beta_{n} = n^{- \frac{\delta}{2(1 + \delta)}}$, we obtain the assertion of the lemma. 
\end{proof}

\begin{proof}[Proof of Theorem \ref{theorem local for tau lattice}]
By the Markov property, we have, for any $n \geq 1$ and $x\in \bb R$,
\begin{align} \label{obvious identity-tau002}
\bb P\left( \tau_x = n+1 \right) 
= \sum_{y \geq 0: \, y-x \in \hbar \bb Z+na}  \bb P\left( y + X_1 <0 \right)  \bb P\left( x + S_n = y, \tau_x > n-1 \right).
\end{align}
Summing up over $y \in \bb R_+$ in the lattice $\hbar \bb Z + na +x$ in  \eqref{obvious identity-tau002},
using Theorem \ref{CLLT-lattice-n3/2 main result} and the definition of $K_n(x)$ (cf.\ \eqref{definition of varkappa-aa001}), we get
\begin{align*} 
 & \left| \bb{P} ( \tau_x =n+1 ) 
   - \hbar   \frac{V_n(x) }{ \sigma^3 n^{3/2} } K_n(x) \right| \\
&\leq c \frac{\left(n^{-\frac{\delta}{6} }  + V_{n}(x) \right)}{n^{3/2  }}  n^{ -\frac{\delta}{8(3+\delta)} }\log n 
\sum_{y \geq 0: \, y-x \in \hbar \bb Z+na}  \left( n^{-\frac{\delta}{6}} + \check V_n(y)  \right)  \bb P\left( X_1 < -y \right). 
\end{align*}
Since the function $\varkappa_n$ is bounded on $\bb R$, 
the last sum is also bounded.
By Lemma \ref{lem-appox for varkappa_n-001}
and using the fact that $L(0) = \frac{2}{\sqrt{2 \pi}}$
and $V_n(x) \phi_L\left(\frac{x}{\sigma\sqrt{n}}\right) = \sqrt{2 \pi} V(x) \phi \left(\frac{x}{\sigma\sqrt{n}}\right)$,
this ends the proof of the theorem.
\end{proof}

\begin{proof}[Proof of Lemma \ref{lemma-varkappa-001}]
For $x\in \bb R$ and $n\geq 1$, set for brevity $u_{n,x}=\{na + x\}_{\hbar}$.
By \eqref{definition varkappa-aa001} and \eqref{identity for varkappa(u) aa001}, we have 
\begin{align} \label{latt-loctau-bbb001}
\varkappa_n(x) = \varkappa(u_{n,x}) 
& = \sum_{k =0}^{\infty} \check V(\hbar k +u_{n,x}) \bb P(X_1< -\hbar k-u_{n,x})  \notag\\
& \geq \check V(0) \sum_{k =0}^{\infty}  \bb P(X_1< -\hbar k-u_{n,x}),
\end{align}
where we used the fact that $\check V(\hbar k +u_{n,x})\geq \check V(0)$. 
On the other hand, since $na+x = m\hbar + u_{n,x}$ for some $m\in \bb Z$, we have
\begin{align*} 
\bb P(\tau_x=n+1) 
&\leq \bb P(x+ S_{n}\geq 0,  x+ S_{n}+ X_{n+1} < 0) \\ 
&= \sum_{k \in \bb Z: \,  k \hbar + na+x \geq 0}  \bb P \left( X_{n+1} < -k \hbar - na - x, S_n=k \hbar + na \right)  \notag\\
&\leq \sum_{k \in \bb Z: \,  (k+m) \hbar + \left\{na + x\right\}_{\hbar} \geq 0} \bb P(X_1< - (k+m) \hbar - u_{n,x})\\
&= \sum_{k' \in \bb Z: \,  k' \hbar + u_{n,x} \geq 0} \bb P(X_1< - k' \hbar - u_{n,x})\\
&= \sum_{k=0}^{\infty} \bb P(X_1< - k \hbar - u_{n,x}).
\end{align*}
The first assertion follows using \eqref{latt-loctau-bbb001}. 

For the second, if $\varkappa_n(x)=0$, then, by \eqref{latt-loctau-bbb001} and the fact that $X_1$ is supported on the lattice $\hbar \mathbb{Z} + a$, 
we get $\bb P(X_1< a- \hbar) \leq  \bb P(X_1< -u_{n,x}) =0.$
\end{proof}

\begin{proof}[Proof of Lemma \ref{lemma-varkappa-002}]
We use the notation from the proof of Lemma \ref{lemma-varkappa-001}. 
If $\bb P(X_1<a-\hbar)=0$, 
then, by \eqref{latt-loctau-bbb001}, we get $\varkappa_n(x)  = \check V(u_{n,x})  \bb P(X_1< -u_{n,x})$, where
\begin{align*} 
\bb P(X_1< -u_{n,x}) 
= \left\{
\begin{array}{ll}
 \bb P(X_1=a-\hbar)         & \text{if }    u_{n,x} \in [0,\hbar-a), \\
0     & \text{if }   u_{n,x} \in [\hbar - a,\hbar).
\end{array}
\right.   
\end{align*}
This proves the first assertion. 

For the second assertion, it follows from \eqref{latt-loctau-bbb001} that, for $x\in \bb R$ and $n\geq 1$,
\begin{align*} 
\varkappa_n(x)  \geq \check V(0)  \bb P(X_1<-u_{n,x}) \geq \check V(0)  \bb P(X_1<a-\hbar)>0.
\end{align*}
\end{proof}

\section{Proofs in the non-lattice case} \label{SecProof-non-lattice Theor-probtau-001}
\subsection{Effective two-sided bound version of the Caravenna type result}

We first prove an effective local limit theorem of Caravenna type for target functions 
with unbounded support and no a priori regularity assumptions. 
Furthermore, the result will be formulated as two distinct bounds--an upper and a lower estimate.

\begin{theorem} \label{t-B 002} 
Assume that $X_1$ is non-lattice, $\bb E (X_1) = 0$, $\bb E (X^2_1)= \sigma^2$,  and that there exists $\delta > 0$ 
such that $\bb E (|X_1|^{2+\delta})  < \infty.$ 
Then, one can find a constant $c >0$ and a sequence $(\alpha_n)_{n \geq 1}$ of positive numbers with the property that 
$\lim_{n\to\infty}\alpha_n = 0$ such that: 

\noindent 1. 
For any $n\geq 1$, $x \in \bb R$ and 
any integrable functions $f, g: \bb R \mapsto \bb R_+$ 
satisfying $f \leq_{\alpha_n} g$, $f(t)=0$ for $t<0$, and $g(t)=0$ for $t<-\alpha_n$,       
\begin{align}\label{eqt-B 001} 
  \bb{E} \Big( f(x+S_n ); \tau_x >n \Big)
& \leq  (1 +  c \alpha_n)  \frac{V_n(x)}{\sigma^2 n} 
  \int_{\bb R_+} g(y) \ell\left(\frac{x}{\sigma\sqrt{n}} ,\frac{y}{\sigma\sqrt{n}}\right) dy   \notag\\
& \quad   
+ c  \frac{n^{-\frac{\delta}{6} }  + V_{n}(x) n^{-\frac{\delta}{8(3+\delta)} } \log n}{  n }  \left\|  g \right\|_1. 
\end{align}

\noindent 2. 
For any $n\geq 1$, $x \in \bb R$ and 
any integrable functions $f, g, h : \bb R \mapsto \bb R_+$
satisfying $h \leq_{\alpha_n} f \leq_{\alpha_n} g$, $f(t)=0$ for $t<0$, $g(t)=0$ for $t<-\alpha_n$, 
and $h(t)=0$ for $t<\alpha_n$, 
\begin{align} \label{eqt-B 002} 
  \bb{E} \Big( f(x+S_n ); \tau_x >n \Big)
&  \geq   
 \frac{V_n(x)}{\sigma^2 n} 
    \int_{\bb R_+} \Big[  h(y)  - c \alpha_n f (y)  \Big]  
    \ell\left(\frac{x}{\sigma\sqrt{n}} ,\frac{y}{\sigma\sqrt{n}}\right) dy  
    \notag\\
 & \quad 
 -  c  \frac{n^{-\frac{\delta }{6} }  + V_{n}(x) n^{-\frac{\delta}{8(3+\delta)} } \log n }{  n }  \left\|  g \right\|_1. 
\end{align}
\end{theorem}

Theorem \ref{t-B 002} will play a key role in the proof of Theorem \ref{theorem-n3/2-upper-lower-bounds}, 
where it is applied to a target function of the form
\begin{align*} 
f_m(y) = \mathds 1_{\{y \geq 0\}}  \bb E  \Big( f \left( y+S_{m} \right);  \tau_{y} > m - 1 \Big), \quad y \in \bb R. 
\end{align*}

\subsection{Proof of the upper bound}
In this section we will prove the upper bound \eqref{eqt-B 001} in Theorem \ref{t-B 002}. 

Let $\ee \in (0,\ee_0)$, where $\ee_0>0$ is a sufficiently small constant. 
Set $m = \floor{\ee n}$ and $k = n-m.$ 
We can assume that $n>n_0(\ee) := \frac{4}{\ee^2}$, 
otherwise the bound that we want to prove becomes trivial. 
By the Markov property,  for any starting point $x \in \bb R$, we have
\begin{align} \label{JJJ-markov property}
I_n(x):= \bb{E} \Big( f(x+S_n ); \tau_x >n \Big) 
&= \int_{\bb R_+}  \bb{E} \Big( f(t+S_{m}); \tau_t >m \Big) \bb{P}\left( x+S_{k}\in dt,  \tau_x >k\right)
 \notag \\
&\leq \int_{\bb R_+}  \bb{E} \Big[ f(t+S_{m})  
\Big]
\bb{P} \left( x+S_{k}\in dt,  \tau_x >k \right).
\end{align}
By the upper bound \eqref{LLT-general001} of Theorem \ref{LLT-general}, 
one can find a constant $c >0$ with the property that 
for any $\alpha \in (0, \frac{1}{2})$, there is a constant $c_{\alpha} >0$ such that, for any $t \in \bb R$
and any integrable function $g: \bb R\mapsto \bb R_+$ satisfying $f \leq_{\alpha} g$, 
$f(t) = 0$ for $t < 0$, and $g(t) = 0$ for $t < -\alpha$,       
\begin{align} \label{JJJJJ-1111-001}
\bb{E} \Big[ f(t+S_{m})  \Big] 
& \leq  (1 + c \alpha)  \int_{\bb R} g(s)   \phi_{\sigma^2 m} \left( t-s \right) ds 
    +   \frac{c_{\alpha} \ee^{-(1 + \delta_1)/2}}{ n^{(1 + \delta_1)/2}} \|g\|_1, 
\end{align}
where $\phi_{v}(x) = \frac{1}{\sqrt{v}}\phi \left(\frac{x}{\sqrt{v}}\right)$, 
$\delta_1 = \min\{ 1, \delta \}$ and $\delta>0$ is from the moment assumption $\bb E (|X_1|^{2+\delta})  < \infty.$ 
Substituting \eqref{JJJJJ-1111-001} into \eqref{JJJ-markov property} and using \eqref{lattice-upper-bound-probab-tau-x-002}, 
we get that, for $x \in \bb R$,
\begin{align} \label{JJJ004}
I_{n}(x)    
& \leq  (1+ c \alpha) J_n(x) 
   +   \frac{c_{\alpha} \ee^{-(1 + \delta_1)/2}}{ n^{(1 + \delta_1)/2} }   \|g\|_1 \bb{P} \left( \tau_{x} > k \right)  \notag\\
&  \leq  (1+ c \alpha)  J_n(x) +  c_{\alpha} \ee^{- \frac{1 + \delta_1}{2}} 
\frac{ n^{- \frac{\delta}{4}} + V_n(x) }{ n^{ 1 + \frac{\delta_1}{2}} }  \|g\|_1,  
\end{align}
where for brevity we set
\begin{align} \label{JJJ006}
J_n(x) 
:  = \int_{\bb R _{+}}   
\left( \int_{\bb R}  g(s)  
   \phi_{\sigma^2  m} \left( t-s \right) ds  \right)
   \bb{P}\left( x+S_{k}\in dt, \tau_{x}>k\right). 
\end{align}
Using the change of variables $t = \sigma \sqrt{k} t'$ and $s = \sigma \sqrt{n} s'$, we have 
\begin{align}\label{def-Jn-x-abc}
J_n(x) 
&=  \bb E \left( \varphi_n\left(\frac{x+S_{k}}{\sigma \sqrt{k}}\right); \tau_{x}>k \right), 
\end{align}
where, with the notation $\eta_n=\frac{m}{n}$, 
\begin{align} \label{JJJ006b}  
 \varphi_n(t) 
:= \int_{\bb R} g(\sigma \sqrt{n} s) \phi_{\eta_n} \left( t \sqrt{k/n} - s \right) ds,  \quad  t \in \bb R. 
\end{align}
Since the function $t \mapsto \varphi_n(t)$ is differentiable on $\bb R$ and vanishes as $t \to \pm \infty$,  
by Lemma \ref{New-ApplCondLT-002}, there exist $c, \ee_0 > 0$ such that, for any $\ee \in (0, \ee_0)$, $x\in \bb R$ 
and $n \geq 1$, 
\begin{align}\label{LLT-nonlatice-thm01}
\left| J_n(x) 
 -   \frac{V_k(x)}{ \sigma \sqrt{k}}  \int_{\bb R _{+}}  \varphi_n(t) 
\ell \left(\frac{x}{\sigma\sqrt{k}}, t \right)dt \right| 
& \leq  c \frac{R_{n,\delta}(x)}{\sqrt{n}}   \int_{\bb R}  | \varphi_n'(t) | dt, 
\end{align}
where $R_{n,\delta}(x)$ is defined in Theorem \ref{introTheor-probtauUN-001}. 
By \eqref{JJJ006b} and a change of variable, 
\begin{align} \label{JJJ-20001}
\int_{\bb R}  | \varphi'_n(t) | dt 
&  \leq  \int_{\bb R}  \left[ \int_{\bb R}  
g (\sigma \sqrt{k} s)  
   \left| \phi'\left( \frac{ t -s }{\sqrt{m/k}}\right) \right|  \frac{ds}{\sqrt{m/k}}  \right]
\frac{dt}{\sqrt{m/k}} \notag\\
&  = \int_{\bb R}  \left[ \int_{\bb R}  
g (\sigma \sqrt{m} s)    \left| \phi'\left( t -s\right) \right| ds \right] dt   \notag\\
&  \leq \frac{c}{\sqrt{m}} \|g \|_1 = \frac{c}{\ee^{1/2} \sqrt{n}} \|g \|_1. 
\end{align}
Concerning the main term in \eqref{LLT-nonlatice-thm01}, 
substituting $t' = t \sqrt{n/k}$ and using \eqref{new-density-scale-v}, 
\begin{align} \label{heat-kernel-esti-002}
& \frac{V_k(x)}{ \sigma \sqrt{k}}  \int_{\bb R _{+}}  \varphi_n(t') \ell \left(\frac{x}{\sigma\sqrt{k}}, t' \right)dt'  \notag\\
& =  \frac{V_k(x)}{ \sigma \sqrt{k}} 
  \int_{\bb R _{+}}  \varphi_n \left( t \sqrt{n/k} \right) \frac{1}{\sqrt{k/n}}  \ell \left(\frac{\frac{x}{\sigma\sqrt{n}}}{\sqrt{k/n}}, \frac{t}{\sqrt{k/n}} \right) dt  \notag \\
  & =  \frac{V_k(x)}{ \sigma \sqrt{k}} 
  \int_{\bb R _{+}}  \varphi_n \left( t \sqrt{n/k} \right)  \frac{\psi_{1-\eta_n}\left(\frac{x}{\sigma\sqrt{n}}, t \right)}{H\left( \frac{x}{\sigma \sqrt{k}} \right)}  dt  \notag \\
& = \frac{V_n(x)}{ \sigma \sqrt{n}}  \int_{\bb R _{+}}  \varphi_n \left( t \sqrt{n/k} \right) \ell_{1-\eta_n}\left(\frac{x}{\sigma\sqrt{n}}, t \right) dt, 
\end{align}
where $\eta_n = \frac{m}{n}$ and $1 - \eta_n = \frac{k}{n}$,
and we also made use of the fact that $\frac{V_k(x)}{ \sigma \sqrt{k}} \frac{1}{H( \frac{x}{\sigma \sqrt{k}})} = \frac{V_n(x)}{ \sigma \sqrt{n}} \frac{1}{H( \frac{x}{\sigma \sqrt{n}})}$. 
By \eqref{JJJ006b}, \eqref{def-convo-ell-H-plus} and Fubini's theorem, we get
\begin{align*}
&  \int_{\bb R _{+}}  \varphi_n \left( t \sqrt{n/k} \right) \ell_{1-\eta_n}\left(\frac{x}{\sigma\sqrt{k}}, t \right)dt \\
&= \int_{\bb R _{+}}  \left[ \int_{\bb R} 
g(\sigma \sqrt{n}s )  \phi_{\eta_n} \left(s-t\right) ds  \right]
    \ell_{1-\eta_n}\left(\frac{x}{\sigma\sqrt{n}} , t \right) dt \\
&=  \int_{\bb R} g(\sigma \sqrt{n}s ) 
\Bigg[ \int_{\bb R _{+}}  \phi_{\eta_n} \left(s-t\right)  \ell_{1-\eta_n}\left(\frac{x}{\sigma\sqrt{n}} , t \right) dt\Bigg] ds \\
& =  \int_{\bb R} g(\sigma \sqrt{n} s)  
\phi_{\eta_n}*\ell^+_{1-\eta_n}\left(\frac{x}{\sigma\sqrt{n}} ,s\right)  ds  \notag\\ 
& = \frac{1}{\sigma\sqrt{n}}  
\int_{\bb R} g(y)  \phi_{\eta_n}*\ell^+_{1-\eta_n}\left(\frac{x}{\sigma\sqrt{n}} ,\frac{y}{\sigma\sqrt{n}}\right)  dy. 
\end{align*}
Combining this with \eqref{heat-kernel-esti-002} and using the fact that $g(t) = 0$ for $t < -\alpha$,  we have
\begin{align}\label{heat-kernel-esti-003}
 \frac{V_k(x)}{ \sigma \sqrt{k}}  \int_{\bb R _{+}}  \varphi_n(t') 
\ell \left(\frac{x}{\sigma\sqrt{k}}, t' \right)dt'  
& =   \frac{V_n(x)}{ \sigma^2 n }
\int_{\bb R_+} g(y)  \phi_{\eta_n}*\ell^+_{1-\eta_n}\left(\frac{x}{\sigma\sqrt{n}} ,\frac{y}{\sigma\sqrt{n}}\right)  dy  \notag\\
& \quad +  \frac{V_n(x)}{ \sigma^2 n }
\int_{-\alpha}^0 g(y)  \phi_{\eta_n}*\ell^+_{1-\eta_n}\left(\frac{x}{\sigma\sqrt{n}} ,\frac{y}{\sigma\sqrt{n}}\right)  dy. 
\end{align}
For the first term in \eqref{heat-kernel-esti-003}, 
 it follows from \eqref{lattice-conv001} and \eqref{lattice-bound-ell-psi-a} that  
\begin{align}\label{bound-ell-psi-a}
\int_{\bb R_+} g(y)  \phi_{\eta_n}*\ell^+_{1-\eta_n}\left(\frac{x}{\sigma\sqrt{n}} ,\frac{y}{\sigma\sqrt{n}}\right)  dy
\leq  \int_{\bb R_+} g(y)   \ell \left(\frac{x}{\sigma\sqrt{n}} ,\frac{y}{\sigma\sqrt{n}}\right)   dy + c \ee^{1/2} \|g\|_1. 
\end{align}
By \eqref{ell-intergr-bound-002} of Lemma \ref{Lipschiz  property for ell_H in y} and the fact that $\ell_{1-\eta_n} ( \frac{x}{\sigma \sqrt{n}}, 0 ) = 0$,
 we have 
$| \ell_{1-\eta_n}  ( \frac{x}{\sigma \sqrt{n}}, \frac{y}{\sigma\sqrt{n}} - z )| \leq c ( \frac{1}{\sqrt{n}} + |z| )$
for any $x, z \in \bb R$ and $y \in [-\alpha, 0]$ with $\alpha <1$, so that, by \eqref{def-ell-H-plus-minus}, 
\begin{align*}
\left| \phi_{\eta_n}*\ell^+_{1-\eta_n}\left(\frac{x}{\sigma\sqrt{n}} ,\frac{y}{\sigma\sqrt{n}}\right) \right|
& = \left|  \int_{\bb R}  \phi_{\eta_n} (z) \ell^+_{1-\eta_n} \left( \frac{x}{\sigma \sqrt{n}}, \frac{y}{\sigma\sqrt{n}} - z \right) dz \right| \notag\\
& \leq  c  \int_{\bb R}  \phi_{\eta_n} (z) \left(  \frac{1}{\sqrt{n}}  + |z| \right) dz \notag\\
& \leq  c \left( \frac{1}{\sqrt{n}} + \ee^{1/2} \right). 
\end{align*}
Hence 
\begin{align}\label{bound-ell-psi-b}
\left| \int_{-\alpha}^0 g(y)  \phi_{\eta_n}*\ell^+_{1-\eta_n}\left(\frac{x}{\sigma\sqrt{n}} ,\frac{y}{\sigma\sqrt{n}}\right)  dy  \right|
\leq  c \left( \frac{1}{\sqrt{n}} + \ee^{1/2} \right)  \int_{-\alpha}^0 g(y) dy. 
\end{align}
Substituting \eqref{bound-ell-psi-a} and \eqref{bound-ell-psi-b} into \eqref{heat-kernel-esti-003}, we get
\begin{align}\label{heat-kernel-esti-004}
& \frac{V_k(x)}{ \sigma \sqrt{k}}  \int_{\bb R _{+}}  \varphi_n(t') 
\ell \left(\frac{x}{\sigma\sqrt{k}}, t' \right)dt'  \notag\\
& \leq  \frac{V_n(x)}{ \sigma^2 n }
\int_{\bb R_+} g(y) \ell\left(\frac{x}{\sigma\sqrt{n}} ,\frac{y}{\sigma\sqrt{n}}\right)  dy 
 +  c \left( \frac{1}{\sqrt{n}} + \ee^{1/2} \right) \frac{V_n(x)}{ n } \|g\|_1.  
\end{align}
From Lemma \ref{New-ApplCondLT-002} and \eqref{heat-kernel-esti-004},  it follows that
\begin{align}\label{Bound-Jn-x-abc}
& \left| J_{n}(x) -  \frac{V_n(x)}{ \sigma^2 n }
\int_{\bb R_+} g(y)   \ell\left(\frac{x}{\sigma\sqrt{n}} ,\frac{y}{\sigma\sqrt{n}}\right)  dy \right|   \notag\\
& \leq  c \ee^{-1/2} \frac{R_{n, \delta}(x)}{ n } \left\|  g \right\|_1 + c \left( \frac{1}{\sqrt{n}} + \ee^{1/2} \right) \frac{V_n(x)}{ n } \|g\|_1.  
\end{align}
Combining this with \eqref{JJJ004}, we obtain 
\begin{align}\label{last-inequality-001}
&  \bb{E} \Big( f(x+S_n ); \tau_x >n \Big)
 \leq  (1 +  c \alpha)  \frac{V_n(x)}{\sigma^2 n} 
  \int_{\bb R_+} g(y) \ell\left(\frac{x}{\sigma\sqrt{n}} ,\frac{y}{\sigma\sqrt{n}}\right) dy   \notag\\
& \qquad\qquad\qquad   
+ c \ee^{-1/2} \frac{ n^{-\frac{\delta}{4} }  + V_{n}(x) n^{-\frac{\delta}{4(3+\delta)} }\log n }{  n }  \left\|  g \right\|_1
+  c \left( \frac{1}{\sqrt{n}} + \ee^{1/2} \right) \frac{V_n(x)}{ n } \|g\|_1  \notag\\
& \qquad\qquad\qquad  + c_{\alpha}  \ee^{- \frac{1 + \delta_1}{2} } \frac{ n^{- \frac{\delta}{4}} + V_n(x) }{  n^{ 1 + \frac{\delta_1}{2}} }  \|g\|_1. 
\end{align}
Now we take $\ee = n^{-\frac{\delta}{4(3+\delta)} }$ and choose a sequence $(\alpha_n)_{n \geq 1}$ of positive numbers such that $c_{\alpha_n} \leq n^{\frac{\delta}{8(3+\delta)}}$. 
Then the last term on the right-hand side of \eqref{last-inequality-001} is bounded by the second one. 
This concludes the proof of the upper bound \eqref{eqt-B 001}.

\subsection{Proof of the lower bound}
In this section we establish the lower bound \eqref{eqt-B 002} in Theorem \ref{t-B 002}. 

Let us keep the notation used in the proof of the upper bound \eqref{eqt-B 001}, notably the same $\ee_0, \ee, \alpha, k, m, \delta, \delta_1$. 
In view of \eqref{JJJ-markov property}, by the Markov property, we have
\begin{align} \label{staring point for the lower-bound-001} 
 \bb{E} \Big( f(x+S_n ); \tau_x >n \Big)  
 =  I_{n,1}(x) - I_{n,2}(x), 
\end{align}
where 
\begin{align}
I_{n,1}(x) & = \int_{\bb R_+}  \bb E  \Big( f ( t + S_m)  \Big) \bb{P} \left( x+S_{k} \in dt,  \tau_x >k\right),  \label{def-In-x-1} \\
I_{n,2}(x) & =  \int_{\bb R_+}  \bb{E} \Big( f(t + S_m);  \tau_t \leq m \Big)  \bb{P} \left( x+S_{k}\in dt, \tau_x >k\right).  \label{def-In-x-2}
\end{align}
While the lower bound of $I_{n,1}(x)$ is very similar to that in the upper bound, the estimation of 
$I_{n,2}(x)$ is much more involved. 

\textit{Lower bound of $I_{n,1}(x)$. }
By  the local limit theorem (see \eqref{LLT-general002} in Theorem \ref{LLT-general}), 
one can find a constant $c >0$ with the property that 
for any $\alpha \in (0, \frac{1}{2})$, there is a constant $c_{\alpha} >0$ such that, for any $t \in \bb R$
and any integrable function $h: \bb R\mapsto \bb R_+$ satisfying $h \leq_{\alpha} f$, 
\begin{align} \label{lower-LLT-Cara-001}
\bb{E} \Big[ f(t+S_{m})  \Big] 
 \geq  \int_{\bb R}  (h(s) - c \alpha f(s))   \phi_{\sigma^2 m} \left( s-t \right) ds 
    -  \frac{c_{\alpha} \ee^{-(1 + \delta_1)/2}}{ n^{(1 + \delta_1)/2}} \|f\|_1. 
\end{align}
Substituting \eqref{lower-LLT-Cara-001} into \eqref{def-In-x-1} and using the bound \eqref{lattice-upper-bound-probab-tau-x-002},
we get, 
 for any $n \geq 1$ and $x \in \bb R$,
\begin{align}\label{bound-In1-lower001}
I_{n,1}(x) \geq J_{n,1}(x) 
-  c_{\alpha} \frac{ \ee^{-(1 + \delta_1)/2} (n^{- \frac{\delta}{4}} + V_n(x)) }{ n^{1 + \frac{\delta_1}{2} }} \left\Vert f  \right\Vert_1, 
\end{align}
where 
\begin{align*}
J_{n,1}(x): =   \int_{\bb R_+}  \left( \int_{\bb R }  \Big[  h(y)  - c \alpha f (y)  \Big]   \phi_{\sigma^2  m} \left( s - t \right) ds  \right)
   \bb{P}\left( x+S_{k}\in dt,  \tau_x >k \right).
\end{align*}
For $J_{n,1}(x)$, proceeding in the same way as in the proof of the upper bound of $J_n(x)$ (cf.\ \eqref{JJJ006} and \eqref{Bound-Jn-x-abc}),
one has, uniformly in $x \in \bb R$,  
\begin{align}\label{lower-bound-Jn-x-prime}
J_{n,1}(x)
& \geq  \frac{V_n(x)}{ \sigma^2 n }  \int_{\bb R_+}  \Big[  h(y)  - c \alpha f (y)  \Big] 
 \ell\left(\frac{x}{\sigma\sqrt{n}} ,\frac{y}{\sigma\sqrt{n}}\right) dy  \notag\\
 &\quad    - c \ee^{-1/2} \frac{R_{n, \delta}(x)}{ n } \left\|  f \right\|_1 - c \left( \frac{1}{\sqrt{n}} + \ee^{1/2} \right) \frac{V_n(x)}{ n } \|f\|_1. 
\end{align}
Substituting \eqref{lower-bound-Jn-x-prime} into \eqref{bound-In1-lower001}, we obtain 
the lower bound for $I_{n,1}(x)$: 
\begin{align*}
 I_{n,1}(x)  
 & \geq  \frac{V_n(x)}{ \sigma^2 n }  \int_{\bb R_+}  \Big[  h(y)  - c \alpha f (y)  \Big] 
 \ell\left(\frac{x}{\sigma\sqrt{n}} ,\frac{y}{\sigma\sqrt{n}}\right) dy  \notag\\
 &\quad    - c \ee^{-1/2} \frac{R_{n, \delta}(x)}{ n } \left\|  f \right\|_1 - c \left( \frac{1}{\sqrt{n}} + \ee^{1/2} \right) \frac{V_n(x)}{ n } \|f\|_1
 \notag\\
 & \quad -  c_{\alpha} \frac{ \ee^{-(1 + \delta_1)/2} (n^{- \frac{\delta}{4}} + V_n(x)) }{ n^{1 + \frac{\delta_1}{2} }} \left\Vert f  \right\Vert_1. 
\end{align*}
As in the proof of the upper bound, taking $\ee = n^{-\frac{\delta}{4(3+\delta)} }$ and choosing a sequence $(\alpha_n)_{n \geq 1}$ of positive numbers such that $c_{\alpha_n} \leq n^{\frac{\delta}{8(3+\delta)}}$,
we get 
\begin{align}\label{Lower_In_lll}
 I_{n,1}(x)  
 & \geq  \frac{V_n(x)}{ \sigma^2 n }  \int_{\bb R_+}  \Big[  h(y)  - c \alpha_n f (y)  \Big] 
 \ell\left(\frac{x}{\sigma\sqrt{n}} ,\frac{y}{\sigma\sqrt{n}}\right) dy  \notag\\
 &\quad    - c  \frac{n^{-\frac{\delta}{6} }  + V_{n}(x) n^{-\frac{\delta}{8(3+\delta)} }\log n }{  n }  \left\|  f \right\|_1. 
\end{align}

\textit{Upper bound of $I_{n,2}(x)$.}
We proceed to give an upper bound for $I_{n,2}(x)$ defined in \eqref{def-In-x-2}, which can be rewritten as
\begin{align} 
I_{n,2}(x) = \int_{\bb R_+}  f_m(y)  \bb{P}\left( x+S_{k}\in dy, \tau_x > k  \right),  \label{K2-b01c-001}
\end{align}
where, for $y \in \bb R$, 
\begin{align} \label{Bytheduality-001}
f_m(y) : =  \mathds 1_{\{y \geq 0\}} \bb{E}  \Big( f(y + S_{m});  \tau_{y} \leq m  \Big). 
\end{align}
The function $y \mapsto f_m(y)$ is integrable on $\bb R$ since  $f$ is integrable on $\bb R$.
Denote 
\begin{align} \label{DefM88}
g_m(y) :=  \mathds 1_{\{y + \alpha \geq 0\}}
\bb{E}  \Big( g(y + S_{m});   \tau_{y - \alpha} \leq m  \Big),  \quad  y \in \bb R.   
\end{align}
Then $f_m \leq_{\alpha} g_m$ since $f \leq_{\alpha} g$. 
By \eqref{DefM88} and a change of variable $t + S_m =u$, we get
\begin{align}\label{bound-Li-norm-M}
\| g_m \|_1 \leq \left\|  g  \right\|_1. 
\end{align}
Applying the upper bound \eqref{eqt-B 001} of  Theorem \ref{t-B 002} to the function $f_m$, 
we obtain that, 
uniformly in $x\in \bb R$,  
\begin{align}\label{eqt-A 001_Lower}
I_{n,2}(x) 
&  \leq   (1 +  c \alpha_n)  \frac{V_k(x)}{\sigma^2 k}  
 \int_{\bb R_+}  g_m(y)  \ell \left(\frac{x}{\sigma\sqrt{k}} ,\frac{y}{\sigma\sqrt{k}}\right) dy   \notag\\
& \quad  + c  \frac{n^{-\frac{\delta }{6} }  + V_{n}(x) n^{-\frac{\delta}{8(3+\delta)} } \log n}{  n }  \left\|  g \right\|_1.  
\end{align}
For the first term in \eqref{eqt-A 001_Lower}, we write 
\begin{align}\label{decom-integral-gm-ell}
 \int_{\bb R_+}  g_m(y)  \ell \left(\frac{x}{\sigma\sqrt{k}} ,\frac{y}{\sigma\sqrt{k}}\right) dy  
& = \int_{0}^{\sqrt{\ee n \log n}}  g_m(y)  \ell \left(\frac{x}{\sigma\sqrt{k}} ,\frac{y}{\sigma\sqrt{k}}\right) dy
 \notag\\
& \quad + \int_{\sqrt{\ee n \log n}}^{\infty}  g_m(y) \ell \left(\frac{x}{\sigma\sqrt{k}} ,\frac{y}{\sigma\sqrt{k}}\right) dy. 
\end{align}
By \eqref{DefM88} and the local limit theorem (see \eqref{LLT-general001} in Theorem \ref{LLT-general}), 
there exists $c>0$ such that for any $y \in \bb R$, 
\begin{align*}
g_m(y) =  
\mathds 1_{\{y + \alpha \geq 0\}}  \bb{E}  \Big(  g(y + S_{m});   \tau_{y - \alpha} \leq m  \Big)
\leq  \bb{E}  \Big( \tilde g(y + S_{m})  \Big)  \leq \frac{c}{\sqrt{m}} \| g\|_1. 
\end{align*}
Here we choose $\tilde g$ such that $f\leq_{\alpha/2} \tilde g \leq_{\alpha/2}  g$.
Using this and the fact that $m = \floor{\ee n}$ and $k = n-m$ (so that $\sqrt{n/k} \leq 2$ for $\ee$ small), we get
\begin{align}\label{first-term-gm-y-inte}
 \left| \int_{0}^{\sqrt{\ee n \log n}}  g_m(y)  
 \ell \left(\frac{x}{\sigma\sqrt{k}} ,\frac{y}{\sigma\sqrt{k}}\right) dy  \right|  
& \leq  c \ee^{- 1/2}  \| g\|_1  \int_{0}^{ 2 \sqrt{\ee \log n} }  \ell \left(\frac{x}{\sigma\sqrt{k}} , y \right)   dy \notag\\
& \leq   c  \| g\|_1  \ee^{1/2} \log n, 
\end{align}
where in the last line we used the second inequality in \eqref{latice-integral-gm-ell-small}.

Next we handle the second term in \eqref{decom-integral-gm-ell}. 
Since the function $\ell$ is bounded on $\bb R \times \bb R$, by \eqref{DefM88}, we get
\begin{align}\label{integral-My-Fk-xy}
&  \int_{\sqrt{\ee n \log n}}^{\infty} 
 g_m(y)  \ell \left(\frac{x}{\sigma\sqrt{k}} ,\frac{y}{\sigma\sqrt{k}}\right) dy  \notag\\
& \leq  c  \int_{\bb R}  
\bb{E}  \Big( g(y + S_{m});   \tau_{y - \alpha} \leq m    \Big)
  \mathds 1_{\{ y  \geq \sqrt{\ee n \log n}  \}}  dy   \notag\\
& = c  \int_{\bb R}  g(t) 
\bb P \left( \tau_{t - S_m - \alpha}  \leq m,  t - S_m \geq  \sqrt{\ee n \log n}  \right) dt,
\end{align}
where in the last equality we used a change of variable $y + S_m = t$.
Following the same proof as in \eqref{latice-large-deviation-aa} and \eqref{latice-Proba_002}, there exists $c>0$ such that for any $t \in \bb R$, 
\begin{align*}
\bb P \left( \tau_{t - S_m - \alpha}  \leq m,  t - S_m \geq  \sqrt{\ee n \log n}  \right)
\leq c (\ee n)^{-\delta/2}. 
\end{align*}
This, together with \eqref{decom-integral-gm-ell}, \eqref{first-term-gm-y-inte} and \eqref{integral-My-Fk-xy}, implies that 
\begin{align}\label{Bound_J1}
\int_{\bb R_+}  g_m(y)  \ell \left(\frac{x}{\sigma\sqrt{k}} ,\frac{y}{\sigma\sqrt{k}}\right) dy
  \leq   c  \left(   \ee^{1/2} \log n  +  (\ee n)^{-\delta/2} \right) \left\|  g  \right\|_1.
\end{align}
Combining this with \eqref{staring point for the lower-bound-001}, \eqref{Lower_In_lll} and \eqref{eqt-A 001_Lower}, 
taking $\ee = n^{-\frac{\delta}{4(3+\delta)} }$ 
and using the fact that $\|f\|_1 \leq \|g\|_1$,
we conclude the proof of the lower bound \eqref{eqt-B 002}.

\subsection{Two-sided bound ballot type theorem} 
\label{SecProof Theor-probtau-001}

For functions $f$ with fixed support, Theorem \ref{t-B 002} does not provide a satisfactory approximation, 
as the remainder term $o(n^{-1})$ dominates the main term, which is of order $n^{-3/2}$. 
To refine the approximation and improve the remainder, we convolve Theorem \ref{t-B 002} 
with a conditioned integral-type local limit theorem. 
This leads to a different form of conditioned local limit theorem--often referred to as the ballot theorem--which guarantees a main term of order $n^{-3/2}$.
To prepare for a sufficiently general formulation of the local limit theorem, we begin by presenting it in the form of a two-sided bound.

\begin{theorem} \label{theorem-n3/2-upper-lower-bounds}
Assume that $X_1$ is non-lattice, $\bb E (X_1) = 0$, $\bb E (X^2_1)= \sigma^2$,  and that there exists $\delta > 0$ 
such that $\bb E (|X_1|^{2+\delta})  < \infty.$
Then, one can find a constant $c>0$ and a sequence $(\alpha_n)_{n \geq 1}$ of positive numbers
satisfying $\lim_{n \to \infty} \alpha_n = 0$ such that:  

\noindent 1. For any $n\geq 1$ and $x\in \bb R$ and 
any measurable functions $f, g: \bb R \mapsto \bb R_+$ satisfying $f \leq_{\alpha_n} g$  and $\int_{\bb R} g(y) dy < \infty$, 
\begin{align}\label{theorem-n3/2 001}
&  \bb{E}  \Big( f (x+S_n );  \tau_x > n - 1 \Big)
 \leq  (1 + c \alpha_n)  \frac{V_n(x) }{ \sigma^3 n^{3/2} }  
  \int_{\bb R}  g (y)  \check V_n(y)  p \left( \frac{x}{\sigma\sqrt{n}} , \frac{y}{\sigma\sqrt{n} } \right) dy  \notag\\
&   \qquad\qquad\qquad\qquad  +  c  \frac{n^{-\frac{\delta}{6} }  + V_{n}(x) n^{-\frac{\delta}{8(3+\delta)} } \log n}{  n^{3/2} }
 \int_{\bb R} g(y) \left(  \check V_n(y) + \frac{1}{n^{\delta/4}} \right) dy. 
 \end{align}

\noindent 2. For any $n\geq 1$ and $x\in \bb R$ and 
any measurable functions $f, g, h: \bb R \mapsto \bb R_+$  satisfying $h \leq_{\alpha_n} f \leq_{\alpha_n} g$ 
and $\int_{\bb R} g(y) dy < \infty$,  
\begin{align}\label{theorem-n3/2 002}
&  \bb{E}  \Big( f (x+S_n );  \tau_x > n - 1 \Big)
   \geq   \frac{V_n(x) }{ \sigma^3 n^{3/2} }  
 \int_{\bb R}   \left(  h(y) - c \alpha_n f(y) \right)   \check V_n(y)  p \left( \frac{x}{\sigma\sqrt{n}} , \frac{y}{\sigma\sqrt{n} } \right) dy   \notag\\
 & \qquad\qquad\qquad\qquad  -  c  \frac{n^{-\frac{\delta}{6} }  + V_{n}(x) n^{-\frac{\delta}{8(3+\delta)} } \log n }{  n^{3/2} }
 \int_{\bb R} g(y) \left(  \check V_n(y) + \frac{1}{n^{\delta/4}} \right) dy.   
 \end{align}
\end{theorem}

Note that, under condition $ \int_{\bb R} g(y) dy <\infty$, the right-hand sides of \eqref{theorem-n3/2 001} 
and \eqref{theorem-n3/2 002} are finite, since $\frac{1}{\sqrt{n}}\check V_n(y) \leq c$ for any $y\in \bb R$, 
by Lemma 2.4 in \cite{GX-2024-CCLT}.

\subsection{Proof of Theorem \ref{theorem-n3/2-upper-lower-bounds}}

We recall a duality identity for the random walk $x + S_n$ jointly with the exit time $\tau_x$.
This identity is a simple consequence of the fact that the Lebesgue measure is translation invariant. 
Recall that $S_0 = 0$ and $S_n=\sum_{i=1}^{n} X_i$ for $n \geq 1$, and its dual random walk is 
given by $\check S_n= -S_n = - \sum_{i=1}^{n} X_i$ for $n \geq 1$. 
The following lemma is from \cite[Lemma 3.2]{GX-2024-AIHP}.

\begin{lemma} \label{lemma-duality-lemma-2_Cor}
For any $n\geq 1$ and any bounded measurable functions $g, h: \bb R \mapsto \bb R$, 
we have
\begin{align} \label{eq-duality-lemma-002_Cor}
\int_{\bb R_+}  h(x) \bb E \left(   g(x+S_n) \mathds 1_{\{ \tau_x > n-1  \}}  \right) dx
=\int_{\bb R}  g(y) \bb E \left(  h(y + \check S_n)  \mathds 1_{\{ \check \tau_y > n \}}  \right) dy. 
\end{align}
\end{lemma}

The lemma below is an easy consequence of the duality formula 
and bounds of the exit time for the dual random walk. 

\begin{lemma}\label{lemma Qng integr}
Assume that $\bb E (X_1) = 0$, $\bb E (X^2_1)= \sigma^2$,  and that there exists $\delta > 0$ 
such that $\bb E (|X_1|^{2+\delta})  < \infty.$ 
Let $g: \bb R \to  \bb R_{+}$ be a measurable function
satisfying $ \int_{\bb R} g(y) dy <\infty$.
Then, there exists a constant $c>0$ such that for any $n\geq 1$, 
\begin{align} \label{eq in lemma Qng integr}
\int_{\bb R _{+}}  \bb E ( g(x+S_n); \tau_x >n - 1)  dx 
\leq  \frac{c}{\sqrt{n}}  \int_{\bb R} g(y) \left(  \check V_n(y) + \frac{1}{n^{ \delta/4}} \right)dy. 
\end{align}
\end{lemma}

\begin{proof}
Using \eqref{eq-duality-lemma-002_Cor} of Lemma \ref{lemma-duality-lemma-2_Cor} with $h = 1$, we get that for any $n \geq 1$, 
\begin{align*}
\int_{\bb R_{+}}  \bb E ( g(x+S_n) ; \tau_x >n-1)  dx
 = \int_{\bb R} g(y)   \mathbb{P} \big( \check \tau_{y} > n \big) dy.
\end{align*}
The conclusion follows from \eqref{UniformCondtau-001} of Theorem \ref{introTheor-probtauUN-001}.   
\end{proof}

The following effective
Gaussian heat kernel approximation result is an analogue of Lemma \ref{Lem-ballot-001}. 

\begin{lemma} \label{Lem-ballot-nonlattice-001}
Let $(\alpha_n)_{n\geq 1}$ be a sequence of non-negative numbers such that $\alpha_n\to 0$ as $n\to\infty$.
There exists a constant $c>0$ such that, for any $n \geq 2$ and  $x, y \in \bb R$,
with $m= \floor{n/2}$ and $k = n - m$, 
\begin{align*}
& \Bigg|   \frac{V_k(x)}{\sigma^2 k } 
 \bb E \left[ \ell \left( \frac{x}{\sigma \sqrt{k}},  \frac{y + \check S_m \pm \alpha_n}{\sigma \sqrt{k}} \right);  \check \tau_y > m  \right] 
 - \frac{ V_n(x) \check V_n(y) }{ \sigma^3 n^{3/2}}   p \left( \frac{x}{\sigma\sqrt{n}} , \frac{y}{\sigma\sqrt{n} } \right)
\Bigg|    \notag\\
&  \leq  c \frac{V_n(x) \check R_{n, \delta}(y) }{ n^{3/2} }. 
\end{align*}
\end{lemma}

\begin{proof}
By Lemma \ref{Lipschiz  property for ell_H in y}, there exists a constant $c>0$ such that for any $x, y \in \bb R$, 
\begin{align*}
 \left| \ell \left( \frac{x}{\sigma \sqrt{k}},  \frac{y + \check S_m \pm \alpha_n}{\sigma \sqrt{k}} \right) 
-  \ell \left( \frac{x}{\sigma \sqrt{k}},  \frac{y + \check S_m}{\sigma \sqrt{k}} \right)   \right|  
 \leq  \frac{c \alpha_n}{\sqrt{k}}.  
\end{align*}
The assertion follows from Lemma \ref{Lem-ballot-001}, using \eqref{upper-bound-probab-tau-x}. 
\end{proof}

Let us now proceed with the proof of Theorem \ref{theorem-n3/2-upper-lower-bounds}. 
We first establish the upper bound \eqref{theorem-n3/2 001}.   
Without loss of generality we assume that there is a function $\widetilde g: \bb R \to \bb R_+$ such that 
$f \leq_{\alpha_n} \widetilde g \leq_{\alpha_n} g$. 
Set $m=\floor{n/2}$ and $k = n-m.$  
By the Markov property, we have for any $x \in \bb R$, 
\begin{align} \label{JJJ-markov property_aa}
\bb{E}  \Big( f (x+S_n ); \tau_x > n - 1 \Big)  
 = \int_{\bb R_+}  f_m(y) \bb{P}\left( x+S_{k}\in dy,  \tau_x > k \right),
\end{align}
where 
\begin{align} \label{JJJ-markov property_aa-002}
f_m(y) = \mathds 1_{\{y \geq 0\}}  \bb E  \Big( f \left( y+S_{m} \right);  \tau_{y} > m - 1 \Big). 
\end{align}
Denote
\begin{align} \label{DefHmt}  
g_m(y) := \mathds 1_{\{y + \alpha_n \geq 0\}} 
\bb E  \Big( \widetilde g \left( y+S_{m} \right);  \tau _{y + \alpha_n} > m - 1 \Big),  \quad  y \in \bb R. 
\end{align} 
Since $f\leq_{\alpha_n}  \widetilde g$, we see that $f_m \leq_{\alpha_n} g_m$ and that both functions $f_m$ and $g_m$ are integrable on $\bb R$. 
Applying the upper bound \eqref{eqt-B 001} of Theorem \ref{t-B 002},  for any $n \geq 1$ and $x \in \bb R$, 
\begin{align}\label{UpperBoundhhn32}
& \bb{E}  \Big( f (x+S_n ); \tau_x > n - 1 \Big)   = \bb E \left( f_m(x + S_k); \tau_x > k \right)  \notag\\
&  \leq   (1 + c \alpha_n)    \frac{V_k(x)}{\sigma^2 k} 
  \int_{\bb R_+}  g_m(y)  \ell \left(\frac{x}{\sigma\sqrt{k}}, \frac{y}{\sigma\sqrt{k}}\right) dy 
   +  c  \frac{n^{-\frac{\delta}{6} }  + V_{n}(x) n^{-\frac{\delta}{8(3+\delta)} } \log n}{  n }  \left\|  g_m \right\|_1, 
\end{align}
where we used \eqref{inequality-L-increase-n} and the fact that $m= \floor{n/2}$.
For the last term on the right-hand side of \eqref{UpperBoundhhn32}, 
by \eqref{DefHmt}, \eqref{inequality-L-increase-n} and Lemma \ref{lemma Qng integr}, we get 
\begin{align} \label{BoundJ2Order32}
\int_{\bb R} g_{m}(y) dy
& = \int_{\bb R_+} \mathbb E  \left( \widetilde g \left( y- \alpha_n +S_{m} \right) ;\tau_y > m - 1 \right) dy  \notag\\
& \leq  \int_{\bb R_+} \mathbb E  \left( g \left( y +S_{m} \right) ;\tau_y > m - 1 \right) dy  \notag\\
& \leq \frac{c}{\sqrt{n}}  \int_{\bb R} g(y) \left(  \check V_n(y) + \frac{1}{n^{ \delta/4 }} \right) dy. 
\end{align}
For the first term on the right-hand side of \eqref{UpperBoundhhn32}, 
we use \eqref{DefHmt}, a change of variable  
and the duality formula (\eqref{eq-duality-lemma-002_Cor} of Lemma \ref{lemma-duality-lemma-2_Cor}) to get
\begin{align}\label{BoundJ1rr}
A_{k, m}(x) :& =  \frac{V_k(x)}{ \sigma^2 k} 
 \int_{\bb R_+}  g_m(y)  \ell \left(\frac{x}{\sigma\sqrt{k}}, \frac{y}{\sigma\sqrt{k}}\right) dy \notag\\
&=  \frac{V_k(x)}{\sigma^2 k }  
\int_{\bb R_+} \bb E \Big[ \widetilde g(y + S_m); \tau_{y + \alpha_n} > m - 1 \Big]  \mathds 1_{\{y + \alpha_n \geq 0\}} 
 \ell\left(\frac{x}{\sigma\sqrt{k}}, \frac{y}{\sigma\sqrt{k}}\right) dy  \notag\\
&=  \frac{V_k(x)}{\sigma^2 k} \int_{\bb R_+} \bb E  \Big[  \widetilde g(y + S_m - \alpha_n); \tau_{y} > m - 1 \Big]  
     \ell\left(\frac{x}{\sigma\sqrt{k}}, \frac{y - \alpha_n}{\sigma\sqrt{k}}\right) dy  \notag\\
& =   \int_{\bb R}   \widetilde g(y - \alpha_n)   \frac{V_k(x)}{\sigma^2 k }  
 \bb E \left[  \ell\left(\frac{x}{\sigma\sqrt{k}}, \frac{y + \check S_m - \alpha_n }{\sigma \sqrt{k}} \right); 
  \check \tau_{y} > m   \right]  dy  \notag\\
  & \leq  \int_{\bb R}  g(y)   \frac{V_k(x)}{\sigma^2 k }  
 \bb E \left[  \ell\left(\frac{x}{\sigma\sqrt{k}}, \frac{y + \check S_m - \alpha_n }{\sigma \sqrt{k}} \right); 
  \check \tau_{y} > m   \right]  dy. 
\end{align}
By Lemma \ref{Lem-ballot-nonlattice-001}, it follows that 
\begin{align}\label{non-lattice gauss approx-001}
A_{k, m}(x)
\leq  \frac{V_n(x) }{ \sigma^3 n^{3/2} }  
  \int_{\bb R}  g (y)  \check V_n(y)  p \left( \frac{x}{\sigma\sqrt{n}} , \frac{y}{\sigma\sqrt{n} } \right) dy
  +  c \frac{V_n(x) }{ n^{3/2} }   \int_{\bb R}  g (y)   R_{n, \delta}(y) dy. 
\end{align}
Combining \eqref{UpperBoundhhn32}, \eqref{BoundJ2Order32}, \eqref{BoundJ1rr} and 
\eqref{non-lattice gauss approx-001} 
finishes the proof of the upper bound \eqref{theorem-n3/2 001}.

Now we prove the lower bound \eqref{theorem-n3/2 002}.
Denote
\begin{align} \label{Def-funct-hm}  
h_m(y) := \mathds 1_{\{y - \alpha_n \geq 0\}} 
\bb E  \Big( \widetilde h \left( y+S_{m} \right);  \tau _{y - \alpha_n} > m - 1 \Big),  \quad  y \in \bb R,  
\end{align}
where $m=\floor{n/2}$, $k = n-m$ and the function $\widetilde h: \bb R \to \bb R_+$ satisfies $h \leq_{\alpha_n} \widetilde h \leq_{\alpha_n} f$. 
Recall that $f_m$ and $g_m$ are defined by \eqref{JJJ-markov property_aa-002} and \eqref{DefHmt}, respectively. 
Applying the lower bound \eqref{eqt-B 002} of Theorem \ref{t-B 002}, 
we get 
\begin{align}\label{lower-bound-Ballot-001}
 \bb{E}  \Big( f (x+S_n ); \tau_x > n - 1 \Big)  
 &  = \bb E \left( f_m(x + S_k); \tau_x > k \right)  \notag\\
&  \geq    \frac{V_k(x)}{\sigma^2 k} 
    \int_{\bb R_+} \Big[  h_m(y)  - c \alpha_n  f_m(y)  \Big]  
    \ell\left(\frac{x}{\sigma\sqrt{k}} ,\frac{y}{\sigma\sqrt{k}}\right) dy   \notag\\
 & \quad  -  c  \frac{n^{-\frac{\delta }{6} }  + V_{n}(x) n^{-\frac{\delta}{8(3+\delta)} } \log n}{  n }  \left\|  g_m \right\|_1,  
\end{align}
where we used \eqref{inequality-L-increase-n} and the fact that $m= \floor{n/2}$. 
By the duality formula (\eqref{eq-duality-lemma-002_Cor} of Lemma \ref{lemma-duality-lemma-2_Cor}), we have
\begin{align*}
 \int_{\bb R_+}   h_m(y)  \ell\left(\frac{x}{\sigma\sqrt{k}} ,\frac{y}{\sigma\sqrt{k}}\right) dy 
&=   \int_{\bb R_+} \bb E  \Big[  \widetilde h(y + S_m + \alpha_n); \tau_{y} > m - 1 \Big]  
     \ell\left(\frac{x}{\sigma\sqrt{k}},  \frac{y + \alpha_n}{\sigma\sqrt{k}}\right) dy  \notag\\
& =   \int_{\bb R}   \widetilde h(y + \alpha_n)  
 \bb E \left[  \ell\left(\frac{x}{\sigma\sqrt{k}}, \frac{y + \check S_m + \alpha_n }{\sigma \sqrt{k}} \right); 
  \check \tau_{y} > m   \right]  dy  \notag\\
  & \geq  \int_{\bb R}  h(y)   \bb E \left[  \ell\left(\frac{x}{\sigma\sqrt{k}}, \frac{y + \check S_m + \alpha_n }{\sigma \sqrt{k}} \right); 
  \check \tau_{y} > m   \right]  dy
\end{align*}
and 
\begin{align*}
  \int_{\bb R_+}   f_m(y)  \ell\left(\frac{x}{\sigma\sqrt{k}} ,\frac{y}{\sigma\sqrt{k}}\right) dy  
& =  \int_{\bb R_+} \bb E \Big[ f(y + S_m); \tau_{y} > m - 1 \Big]  
 \ell\left(\frac{x}{\sigma\sqrt{k}},  \frac{y}{\sigma\sqrt{k}}\right) dy  \notag\\
& =   \int_{\bb R}   f(y)  
 \bb E \left[  \ell\left(\frac{x}{\sigma\sqrt{k}}, \frac{y + \check S_m}{\sigma \sqrt{k}} \right); 
  \check \tau_{y} > m   \right]  dy. 
\end{align*}
From Lemma \ref{Lem-ballot-nonlattice-001}, 
it follows that 
\begin{align*}
&   \frac{V_k(x)}{\sigma^2 k} 
    \int_{\bb R_+} \Big[  h_m(y)  - c \alpha_n  f_m(y)  \Big]  
    \ell\left(\frac{x}{\sigma\sqrt{k}} ,\frac{y}{\sigma\sqrt{k}}\right) dy   \notag\\
& \geq   \frac{V_n(x) }{ \sigma^3 n^{3/2} }  
  \int_{\bb R}  \left(  h(y) - c \alpha_n f(y) \right) \check V_n(y)  p \left( \frac{x}{\sigma\sqrt{n}} , \frac{y}{\sigma\sqrt{n} } \right) dy 
 - c \frac{V_n(x) }{ n^{3/2} }  \int_{\bb R}   f(y) \check R_{n, \delta}(y) dy.
\end{align*}
Substituting this into \eqref{lower-bound-Ballot-001} and using \eqref{BoundJ2Order32} finishes the proof 
of the lower bound \eqref{theorem-n3/2 002}.

\subsection{Proof of Theorem \ref{main res non-lattice case-001}}  \label{sec-proof of equiv for intervals}
By Theorem \ref{theorem-n3/2-upper-lower-bounds}, 
there exists a sequence of positive numbers $(\alpha_n)_{n \geq 1}$ satisfying $\lim_{n \to \infty} \alpha_n = 0$ such that 
\begin{align*}
  \bb{P}  \Big( x+S_n \in [y,y+v), & \tau_x > n - 1 \Big) \leq  (1 + c \alpha_n)  \frac{V_n(x) }{ \sigma^3 n^{3/2} } 
 \int_{y-\alpha_n}^{y+v+\alpha_n}  \check V_n(z)  p \left( \frac{x}{\sigma\sqrt{n}} , \frac{z}{\sigma\sqrt{n} } \right) dz 
\notag\\
&  +  c  \frac{n^{-\frac{\delta}{6} }  + V_{n}(x) n^{-\frac{\delta}{8(3+\delta)} } \log n}{  n^{3/2} }
 \int_{y-\alpha_n}^{y+v+\alpha_n}  \left(  \check V_n(z) + \frac{1}{n^{\delta/4}} \right) dz. 
 \end{align*}
 Note that 
 \begin{align}\label{3 intergrals}
 & \int_{y-\alpha_n}^{y+v+\alpha_n}  \check V_n(z)  p \left( \frac{x}{\sigma\sqrt{n}} , \frac{z}{\sigma\sqrt{n} } \right) dz  \notag\\
 & = \left( \int_{y}^{y+v}  + \int_{y-\alpha_n}^{y}  + \int_{y+v}^{y+v+\alpha_n}  \right) \check V_n(z)  p \left( \frac{x}{\sigma\sqrt{n}} , \frac{z}{\sigma\sqrt{n} } \right) dz. 
\end{align}
We first deal with the second integral in the right-hand side of \eqref{3 intergrals}.
Since $\check V$ is increasing on $\bb R$, by Lemma \ref{lem-inequality for L}, for any $n\geq 1$ and $z\in \bb R$,
\begin{align*} 
\check V_n(z-\alpha_n) 
&= \check V_n(z-\alpha_n) L\left(\frac{z-\alpha_n}{\sigma\sqrt{n}}\right) \notag\\
&\leq  \check V(z) L\left(\frac{z}{\sigma\sqrt{n}}\right) \left(1+c\frac{\alpha_n}{\sqrt{n}}\right)
=  \check V_n(z) \left(1+c\frac{\alpha_n}{\sqrt{n}}\right). 
\end{align*}
Using \eqref{lem-holder prop for ell-002}, this implies that, for any $n\geq 1$ and $x,z\in \bb R$,
\begin{align*}
& \int_{y-\alpha_n}^y \check V_n(z)  p \left( \frac{x}{\sigma\sqrt{n}} , \frac{z}{\sigma\sqrt{n} } \right) dz =  \int_{y}^{y+ \alpha_n} \check V_n(z -\alpha_n)  p \left( \frac{x}{\sigma\sqrt{n}} , \frac{z -\alpha_n}{\sigma\sqrt{n} } \right) dz \notag\\
& \leq \left( 1 +  c \frac{\alpha_n}{\sqrt{n}} \right) 
\int_{y}^{y+ \alpha_n} \check V_n(z)  p \left( \frac{x}{\sigma\sqrt{n}} , \frac{z}{\sigma\sqrt{n} } \right) dz 
+  c \frac{\alpha_n}{\sqrt{n}}   \int_{y}^{y+ v} \check V_n(z) dz. 
\end{align*}
For the third integral in the right-hand side of \eqref{3 intergrals},
using the bound 
\begin{align*} 
\sup_{z \geq v_0+ \inf (\supp \check V)} \frac{\check V(z + \alpha_n)}{\check V(z)} \leq c
\end{align*}
and Lemma \ref{lem-inequality for L}, we have, for any $z \geq v_0+\inf(\supp \check V)$, 
\begin{align*}
\check V_n(z + \alpha_n) 
&=   \check V(z + \alpha_n) L \left( \frac{z + \alpha_n}{\sigma \sqrt{n}} \right) \notag\\
 &\leq c \check V(z) L \left( \frac{z}{\sigma \sqrt{n}} \right) \left( 1 + c \frac{\alpha_n}{\sqrt{n}}   \right)
 =  c \check V_n(z) \left( 1 + c \frac{\alpha_n}{\sqrt{n}}   \right). 
\end{align*}
As above, using \eqref{lem-holder prop for ell-002}, we get, for $n\geq 1$ 
and any $x \in \bb R$, $y \in\supp \check V$ and $v \geq v_0$, 
\begin{align*}
& \int_{y+v}^{y +v+\alpha_n}  \check V_n(z)  p \left( \frac{x}{\sigma\sqrt{n}} , \frac{z}{\sigma\sqrt{n} } \right) dz 
 = \int_{y+v- \alpha_n}^{y+v}  \check V_n(z + \alpha_n)  p \left( \frac{x}{\sigma\sqrt{n}} , \frac{z + \alpha_n}{\sigma\sqrt{n} } \right) dz  \notag\\
&  \leq  c \int_{y+v- \alpha_n}^{y+v}  \check V_n(z)  p \left( \frac{x}{\sigma\sqrt{n}} , \frac{z}{\sigma\sqrt{n} } \right) dz
 +   c \frac{\alpha_n}{\sqrt{n}}   \int_{y}^{y+v}  \check V_n(z) dz. 
\end{align*}
Since $\lim_{n \to \infty} \alpha_n = 0$ and $v\geq v_0$, we have, as $n \to \infty$,
uniformly in $x\in \bb R$, $y \in  \supp \check V$ and $v\geq v_0$,
\begin{align*}
 \left( \int_{y}^{y+ \alpha_n}  +  \int_{y+v-\alpha_n}^{y+v}  \right) \check V_n(z)  p \left( \frac{x}{\sigma\sqrt{n}} , \frac{z}{\sigma\sqrt{n} } \right) dz 
= o(1) \int_{y}^{y+v}  \check V_n(z)  p \left( \frac{x}{\sigma\sqrt{n}} , \frac{z}{\sigma\sqrt{n} } \right) dz. 
\end{align*}
Therefore, as $n\to\infty$, uniformly in $x \in \bb R$, $y \in\supp \check V$ and $v \geq v_0$, 
\begin{align*}
 &\left( \int_{y}^{y+ \alpha_n}  +  \int_{y+v-\alpha_n}^{y+v}  \right)
  \check V_n(z)  p \left( \frac{x}{\sigma\sqrt{n}} , \frac{z}{\sigma\sqrt{n} } \right) dz \notag\\
& \qquad =  o(1)  \int_{y}^{y+v}  \check V_n(z)  p \left( \frac{x}{\sigma\sqrt{n}} , \frac{z}{\sigma\sqrt{n} } \right) dz +   
O\left( \frac{\alpha_n}{\sqrt{n}}\right)   \int_{y}^{y+v}  \check V_n(z) dz. 
\end{align*}
Collecting the above expansions finishes the proof of the upper bound of the theorem. 

The lower bound is obtained in the same way.

\subsection{Proofs for local behavior of $\tau_x$ in the non-lattice case.}  \label{sec proof of equiv local for tau} 
In this section we prove  Theorem \ref{theorem local for tau non-lattice }.
Set $f(y)=\bb P\left( y+X_1 <0 \right)= \bb P(\check S_1 >y)$ for $y\in\bb R$. 
By the Markov property and Theorem \ref{theorem-n3/2-upper-lower-bounds}, 
one can find a constant $c>0$ and a sequence $(\alpha_n)_{n \geq 1}$ of positive numbers satisfying $\lim_{n \to \infty} \alpha_n = 0$  
such that for any $n\geq 1$ and $x\in \bb R$, 
\begin{align*}
\bb P\left( \tau_x = n+1 \right) 
& = \bb{E}  \Big( f (x+S_n ) \mathds 1_{\{ x+S_n \geq 0\} };  \tau_x > n - 1 \Big)  \notag\\
& \leq   (1 + c \alpha_n)  \frac{V_n(x) }{ \sigma^3 n^{3/2} }  \bar K_n(x)  \notag\\
& \quad  +  c  \frac{n^{-\frac{\delta}{6} }  + V_{n}(x) n^{-\frac{\delta}{8(3+\delta)} } \log n}{  n^{3/2} }
 \int_{-\alpha_n}^{\infty}  f (y - \alpha_n) \left(  \check V_n(y) + \frac{1}{n^{\delta/4}} \right) dy, 
\end{align*}
where, 
for $x\in \bb R$, we denoted for short,
\begin{align*} 
\bar K_n(x) & = \int_{-\alpha_n}^{\infty}  f (y - \alpha_n)  \check V_n(y)  p \left( \frac{x}{\sigma\sqrt{n}} , \frac{y}{\sigma\sqrt{n} } \right) dy  \notag\\
& = \int_{-2\alpha_n}^{\infty}  f (y)  \check V_n(y + \alpha_n)  p \left( \frac{x}{\sigma\sqrt{n}} , \frac{y + \alpha_n}{\sigma\sqrt{n} } \right) dy. 
\end{align*}
Since $p$ is bounded,  by finiteness of $\varkappa$ (see \eqref{def kappanonlatt 001}), we have, for $x\in \bb R$,
\begin{align*} 
\bar K_n(x) \leq  \int_{-\alpha_n}^{\infty} f (y - \alpha_n) \left(\check V_n(y)+\frac{1}{n^{\delta/4}} \right) dy \leq c  <\infty.
\end{align*}
To accomplish the proof of Theorem \ref{theorem local for tau non-lattice }, it is enough to establish the following lemma, 
where we recall that the function $\phi_L$ is defined by \eqref{approx by normal density-001b}. 
\begin{lemma} \label{lem-appox for varkappa-non-latt-001}
There exists a constant $c>0$ such that, for any $n\geq 1$ and $x \in \bb R$,
\begin{align*} 
\left| \bar K_n(x) -  \varkappa L(0) \phi_L\left(\frac{x}{\sigma\sqrt{n}}\right) \big( 1 + o(1) \big) \right| 
\leq c n^{-\frac{\delta}{2(1+\delta)}}.
\end{align*}
 \end{lemma}

\begin{proof} 
Lemma \ref{lem-appox for varkappa_n-001}
Let $(\beta_n)$ be a sequence of positive numbers such that $n^{-1/2}\leq \beta_n\to 0$ as $n\to\infty$. 
Then, for any $n\geq 1$ and $x\in \bb R$, the quantity
$\bar K_n(x)$ can be bounded from above and below as follows: 
\begin{align*} 
I_1(x) \leq  \bar K_n(x)  \leq  I_1(x) + c I_2(x),
\end{align*}
where
\begin{align*} 
I_1(x) 
= \int_{-2 \alpha_n}^{\beta_n \sqrt{n} }  f(y) \check V_{n}(y + \alpha_n) \, p\left( \frac{x}{\sigma\sqrt{n}}, \frac{y+ \alpha_n}{\sigma\sqrt{n}} \right) dy, 
\quad
I_2(x) = \int_{\beta_n \sqrt{n}}^{\infty} f(y) \check V_{n}(y + \alpha_n) dy.
\end{align*}
By the  Lipschitz property of the function $p$ (see \eqref{lem-holder prop for ell-002}), 
we have, for $y \in [-2\alpha_n, \beta_n \sqrt{n}]$, 
\begin{align*} 
\left|p\left( \frac{x}{\sigma\sqrt{n}}, \frac{y+\alpha_n}{\sigma\sqrt{n}} \right) - p\left( \frac{x}{\sigma\sqrt{n}}, 0 \right)\right|
\leq c  \frac{|y+\alpha_n|}{\sigma\sqrt{n}} \leq c \beta_n + c \frac{\alpha_n}{\sqrt{n}} \leq c\beta_n.
\end{align*}
Note that $\check V_n(y)= \check  V(y) L\left(\frac{y}{\sigma\sqrt{n}}\right)$. By Lemma 2.3 in \cite{GX-2024-AIHP}, 
we have  
\begin{align*} 
\left| L\left(\frac{y+\alpha_n}{\sigma\sqrt{n}}\right) - L(0) \right| \leq c \frac{|y+\alpha_n|}{\sigma\sqrt{n}} \leq c \beta_n 
+ c \frac{\alpha_n}{\sqrt{n}} \leq c \beta_n. 
\end{align*}
Since $\check V$ is increasing on $\bb R$, it has at most countably many discontinuity points.
By the dominated convergence theorem, it holds that $\int_{\bb R_+} f(y) |\check V(y+\alpha_n) -  \check V(y)| dy \to 0$ as $n \to \infty$. 
Therefore, combining \eqref{approx by normal density-001b} with the symmetry of the function $p$, we obtain
\begin{align} \label{bound I_1nonlatt-aadd001}
I_1(x)
&\leq  \int_{-2\alpha_n}^{\beta_n \sqrt{n}}  f(y) \check V(y+\alpha_n)  dy 
 \left(\phi_L\left(\frac{x}{\sigma\sqrt{n}}\right) + c \beta_n\right) \left(L(0)+c\beta_n\right) \notag \\
&\leq \left( \int_{0}^{\beta_n \sqrt{n}}  f(y)  \check V(y+\alpha_n)   dy + c \alpha_n\right) 
\left(\phi_L\left(\frac{x}{\sigma\sqrt{n}}\right) + c \beta_n\right) \left(L(0)+c\beta_n\right) \notag \\
&\leq \left( \int_{0}^{\infty}  f(y)  \check V(y)   dy + c \alpha_n + o(1) \right) 
\left(\phi_L\left(\frac{x}{\sigma\sqrt{n}}\right) + c \beta_n\right) \left(L(0)+c\beta_n\right) \notag \\
&\leq  \varkappa L(0) \phi_L\left(\frac{x}{\sigma\sqrt{n}}\right)   \big( 1 + o(1) \big) + c \beta_n.
\end{align}
For $I_2(x)$, we derive the bound
\begin{align*} 
I_2(x) & = \int_{\beta_n \sqrt{n}}^{\infty}  \bb P(\check S_1 >y)  \check V_{n}(y + \alpha_n) dy   
\leq  c \bb E(|X_1|^{2+\delta})  \int_{\beta_n \sqrt{n}}^{\infty} \frac{1+y}{y^{2+\delta}} dy  
\leq  c(\beta_n\sqrt{n})^{-\delta}. 
\end{align*}
Combining the above inequalities, we get the following upper bound:
\begin{align*} 
\bar K_n(x) 
\leq \varkappa L(0) \phi_L\left(\frac{x}{\sigma\sqrt{n}}\right) \big( 1 + o(1) \big) + c \beta_n + c(\beta_n\sqrt{n})^{-\delta}.
\end{align*}
A lower bound is proved in the same way. Indeed, proceeding as in \eqref{bound I_1nonlatt-aadd001}, we have
\begin{align*} 
I_1(x)
\geq  \varkappa L(0)  \phi_L\left(\frac{x}{\sigma\sqrt{n}}\right) \big( 1 - o(1) \big) - c \beta_n - c(\beta_n\sqrt{n})^{-\delta}.
\end{align*}
Choosing $\beta_{n}=n^{-\frac{\delta}{2(1+\delta)}}$, we obtain the assertion of the lemma. 
\end{proof}


\appendix

\section{Renewal function formulations}\label{sec: renewal func formulation}
Our results are expressed in terms of the harmonic functions $V$ and $\check V$. 
To facilitate comparison with results formulated using the corresponding renewal functions, 
we provide the necessary definitions below and establish their connections with $V$ and $\check V$.
Let 
\begin{align*} 
c_+  &= \sum_{k=1}^{\infty} \frac{1}{k}\left(\bb P(S_k>0) - \frac{1}{2}\right)=\sum_{k=1}^{\infty} \frac{1}{k}\left(\bb P(\check S_k<0) - \frac{1}{2}\right) =\check c_-,\\
c_- &= \sum_{k=1}^{\infty} \frac{1}{k}\left(\bb P(S_k<0) - \frac{1}{2}\right) = \sum_{k=1}^{\infty} \frac{1}{k}\left(\bb P(\check S_k>0) - \frac{1}{2}\right) =\check c_+,\\
c_0 &= \sum_{k=1}^{\infty} \frac{1}{k}\left(\bb P(S_k=0) - \frac{1}{2}\right) = \sum_{k=1}^{\infty} \frac{1}{k}\left(\bb P(\check S_k=0) - \frac{1}{2}\right) =\check c_0.
\end{align*}
From \cite[Chapter 4]{KV17} it follows that 
\begin{align} \label{appendix-001}
c_0 + c_- + c_+  = 0 
\end{align}
and also that 
\begin{align} \label{appendix-002}
 V(0) = - \bb E S_{\tau_0} = \frac{\sigma e^{-c_-}}{\sqrt{2}}, \quad 
 \check V(0) = - \bb E \check S_{\check \tau_0}= \bb E S_{\check \tau_0} =\frac{\sigma e^{- \check c_-}}{\sqrt{2}}=\frac{\sigma e^{- c_+}}{\sqrt{2}}.
\end{align}
Following \cite{Don89}, consider the stopping time $\tau_D^- =\inf\{k\geq 1: S_k \leq 0 \}$. 
Define the functions $U_D$ and $\check U_K$ as follows:
for any $x\in \bb R$,
\begin{align} 
U_D(x) 
&:= \sum_{k=0}^{\infty} \bb P(S_k \leq x, \tau_D^->k)
= \mathds 1_{\{x \geq 0\}} + \sum_{k=0}^{\infty} \bb P(S_k \leq x, S_1> 0,\ldots,S_k>0) \label{def of U_D-U_K-001} \\
&= \mathds 1_{\{x \geq 0\}} + \sum_{k=0}^{\infty} \bb P(\check S_k \geq -x, \check S_1< 0,\ldots,\check S_k<0)=:\check U_K(x). 
\label{def of U_D-U_K-002} 
\end{align}
The function $U_{D}$ was introduced in \cite{Don89} (it was denoted $U$ there)  
while the function $\check U_K$ was defined in \cite{KV17} (it coincides with the function $U$ defined there but 
 applied to the random walk $(\check S_k)_{k\geq 1}$).
Let $\tau_D =\inf\{k\geq 1: S_k > 0 \}=\check \tau_0$.
Consider the function $V_D$  
\begin{align*} 
V_D(x)
&=\sum_{k=0}^{\infty} \bb P(S_k\leq -x, \tau_D>k) \\
&= \mathds 1_{\{x \geq 0\}} + \sum_{k=1}^{\infty} \bb P(S_k\leq -x, S_1\leq 0,\ldots,S_k\leq 0).
\end{align*}
Note that $V_D$ coincides with the function $V$ introduced in \cite{Don89}.

Using Theorem 4.7 of \cite{KV17} and Proposition 11 of \cite{Don89} we establish the following relations 
among the functions $V, \check V, U_D, V_D$ and $\check U_K$.

\begin{lemma}\ \\ 
1. For any $x\geq 0$, $U_D(x) = \check U_K(x)  = \frac{\check V(x)}{\check V(0)}.$\\
2. When $X_1$ is lattice, for any $x\geq 0$, $\frac{V(x)}{V(0)} = e^{-c_0} V_D(x).$ 
When $X_1$ is non-lattice, $\frac{V(x)}{V(0)} = e^{-c_0} V_D(x),$ for almost all $x\geq 0$.  
\end{lemma}
\begin{proof}
Proof of part 1. 
The first equality follows from the definitions \eqref{def of U_D-U_K-001} and \eqref{def of U_D-U_K-002}.

For the second, by applying Theorem 4.7 of \cite{KV17} to the random walk $(\check S_k)_{k\geq 1}$, we have, for $x\geq 0$,
\begin{align*} 
\bb P_n(x) = \bb P(\min \{0,\check S_1,\ldots \check S_n \} \geq -x) \sim \frac{e^{-\check c_-}}{\sqrt{\pi n} } \check U_K(x).
\end{align*}
On the other hand, by \eqref{lattice-corrol-eq01} and \eqref{appendix-002}, for $x\geq 0$,
\begin{align*} 
\bb P_n(x) 
&= \bb P(x+\check S_1\geq 0,\ldots, x+\check S_n\geq 0) \\
&=\bb P(\check \tau_x >n) 
\sim \frac{2\check V(x)}{\sqrt{2\pi n} \sigma} = \frac{2\check V(0)}{\sqrt{2\pi n} \sigma} \frac{\check V(x)}{\check V(0)} 
=  \frac{e^{-\check c_-}}{\sqrt{\pi n}} \frac{\check V(x)}{\check V(0)}.
\end{align*}
From the two last equivalences we derive the second equality in the first part. 
 
Proof of part 2. 
By Proposition 11 of \cite{Don89}, it follows that
\begin{align} \label{compar005a}
 \check {\bb P}_n(x,y):=\bb P(S_n=x-y, T_x>n) 
\sim \frac{U_D(x) V_D(y)}{\sqrt{2\pi}\sigma n^{3/2}}.
\end{align}
On the other hand, from \eqref{lattice-corrol-eq01}, using \eqref{appendix-001} and \eqref{appendix-002}, we have that
\begin{align} \label{compar005b}
\check {\bb P}_n(x,z)
&=\bb P(x+\check S_n=y, x+\check S_1 \geq 0,\ldots, x+\check S_n \geq 0  )
=\bb P(x+\check S_n=y, \check \tau_x>n) \notag \\
&\sim \frac{2\check V(x) V(y)}{\sqrt{2\pi}\sigma^3 n^{3/2}} 
= \frac{2\check V(0) V(0)}{\sqrt{2\pi}\sigma^3 n^{3/2}} \frac{\check V(x)}{\check V(0)} \frac{V(y)}{V(0)}
= \frac{e^{c_0}}{\sqrt{2\pi}\sigma n^{3/2}} \frac{\check V(x)}{\check V(0)} \frac{V(y)}{V(0)}. 
\end{align}  
Comparing \eqref{compar005a} and  \eqref{compar005b} and using the first part of the lemma, we get
$V_{D}(y)=e^{c_0} \frac{V(y)}{V(0)}$.

In the non-lattice case, by Proposition 18 of \cite{Don89}, for any $x,y \geq 0$ and $v>0$, 
\begin{align} \label{non-lett-compar06a}
 \check {\bb P}_n(x,y,v):=\bb P(S_n\in (x-y-v,x-y], T_x>n) 
\sim \frac{U_D(x) \int_{y}^{y+v} V_D(z) dz}{\sqrt{2\pi}\sigma n^{3/2}}.
\end{align}
In the same way, from \eqref{non-lattice-corrol-geneq01}, applied to random walk $(\check S_n)_{n\geq 0}$, 
using \eqref{appendix-001} and \eqref{appendix-002},
we have that,  for any $x,y\in \bb R$,
\begin{align} \label{non-lett-compar06b}
\check {\bb P}_n(x,y,v)
&= \bb{P} \left(x+\check S_n \in [y,y+v], \tau_x >n-1\right) \\
& \sim  \frac{2\check V(x) \int_{y}^{y+v}  V(z) dz}{\sigma^3 n^{3/2}} 
= \frac{e^{c_0}}{\sqrt{2\pi}\sigma n^{3/2}} \frac{\check V(x)}{\check V(0)} \frac{\int_{y}^{y+v}  V(z) dz}{V(0)}. 
\end{align}
Comparing \eqref{non-lett-compar06a} and \eqref{non-lett-compar06b} and using the first part of the lemma, we get
$V_{D}(y)=e^{c_0} \frac{V(y)}{V(0)}$, for almost  all $y \geq 0$.
\end{proof}

Introduce the functions $\check U_D, \check V_D$ and $U_K$. 
Consider the stopping time $\check \tau_D^- =\inf\{k\geq 1: \check S_k \leq 0 \}$. 
For any $x\in \bb R$,
\begin{align*} 
\check U_D(x) 
&:= \sum_{k=0}^{\infty} \bb P(\check S_k \leq x, \check \tau_D^->k)
= \mathds 1_{\{x \geq 0\}} + \sum_{k=0}^{\infty} \bb P(\check S_k \leq x, \check S_1> 0,\ldots,\check S_k>0) 
\notag\\
&= \mathds 1_{\{x \geq 0\}} + \sum_{k=0}^{\infty} \bb P(S_k \geq -x, S_1< 0,\ldots,S_k<0)=: U_K(x). 
\end{align*}
Let $\check \tau_D =\inf\{k\geq 1: \check S_k > 0 \}=\check \tau_0$.
Consider the function $\check V_D$  
\begin{align*} 
\check V_D(x)
&=\sum_{k=0}^{\infty} \bb P(\check  S_k\leq -x, \check \tau_D>k) \\
&= \mathds 1_{\{x \geq 0\}} + \sum_{k=1}^{\infty} \bb P(\check S_k\leq -x, \check S_1\leq 0,\ldots,\check S_k\leq 0).
\end{align*}
In the same way one can establish the following:
\begin{lemma}\ \\ 
1. For any $x\geq 0$, $\check U_D(x) = U_K(x)  = \frac{V(x)}{V(0)}.$\\
2. When $X_1$ is lattice, for any $x\geq 0$, $\frac{\check  V(x)}{\check  V(0)} = e^{-c_0} \check  V_D(x).$ 
When $X_1$ is non-lattice, $\frac{\check  V(x)}{\check  V(0)} = e^{-c_0} \check  V_D(x),$ for almost all $x\geq 0$. 
\end{lemma}
In particular, from these two lemmas, it follows that, for any $x,y\geq 0$, 
\begin{align} \label{product vcheckV-001}
V(x)\check V(y) 
 = \check U_D(x)  U_D(y) \frac{\sigma^2e^{c_0}}{2}.
\end{align}
When $X_1$ is lattice, for $x,y\geq 0$,
\begin{align} \label{product vcheckV-002}
V(x)\check V(y) & = V_D(x)  U_D(y) \frac{\sigma^2}{2} = \check U_D(x)  \check V_D(y) \frac{\sigma^2}{2}. 
\end{align}
In the case when $X_1$ is non-lattice, 
for any $x\geq 0$, the identity 
$$V(x)\check V(y) = \check U_D(x)  \check V_D(y) \frac{\sigma^2}{2}$$ 
holds for almost all $y\geq 0$,
and
for any $y\geq 0$, the identity 
$$V(x)\check V(y) = V_D(x)  U_D(y) \frac{\sigma^2}{2}$$ 
holds for almost all $x\geq 0$.


\end{document}